\documentclass[12pt, reqno]{amsart}
\usepackage{amssymb}
\usepackage{graphicx}
\usepackage{xcolor} 
\usepackage{tensor}
\usepackage{hyperref}
\usepackage{enumitem}

\usepackage{tikz}
\usepackage{fullpage} 

\definecolor{green}{rgb}{0,0.8,0} 

\newtheorem{theorem}{Theorem}[section]

\newtheorem{lemma}[theorem]{Lemma}
\newtheorem{proposition}[theorem]{Proposition}
\newtheorem*{convention}{Convention}
\theoremstyle{definition}

\newtheorem{example}[theorem]{Example}
\theoremstyle{remark}
\newtheorem{remark}[theorem]{Remark}
\numberwithin{equation}{section}
\newcommand{\nrm}[1]{\Vert#1\Vert}
\newcommand{\wnrm}[1]{\vert\hskip-.1em\vert\hskip-.1em\vert  #1 \vert\hskip-.1em\vert\hskip-.1em\vert}
\newcommand{\abs}[1]{\vert#1\vert}
\newcommand{\brk}[1]{\langle#1\rangle}
\newcommand{\set}[1]{\{#1\}}

\newcommand{\tr}{\textrm{tr}}
\newcommand{\supp}{{\mathrm{supp}}}

\renewcommand{\Re}{\mathrm{Re}}

\newcommand{\aeq}{\simeq}
\newcommand{\aleq}{\lesssim}

\newcommand{\lap}{\triangle}

\newcommand{\ud}{\mathrm{d}}
\newcommand{\rd}{\partial}
\newcommand{\nb}{\nabla}

\newcommand{\bb}{\Big}

\newcommand{\alp}{\alpha}
\newcommand{\bt}{\beta}
\newcommand{\gmm}{\gamma}
\newcommand{\Gmm}{\Gamma}
\newcommand{\dlt}{\delta}
\newcommand{\Dlt}{\Delta}
\newcommand{\eps}{\epsilon}

\newcommand{\kpp}{\kappa}
\newcommand{\lmb}{\lambda}

\newcommand{\sgm}{\sigma}
\newcommand{\Sgm}{\Sigma}
\newcommand{\tht}{\theta}
\newcommand{\vartht}{\vartheta}
\newcommand{\Tht}{\Theta}

\newcommand{\omg}{\omega}
\newcommand{\Omg}{\Omega}

\newcommand{\bfd}{{\bf d}}
\newcommand{\bfe}{{\bf e}}

\newcommand{\bfn}{{\bf n}}

\newcommand{\bfD}{{\bf D}}

\newcommand{\bfI}{{\bf I}}

\newcommand{\bfN}{{\bf N}}

\newcommand{\bfS}{{\bf S}}
\newcommand{\bfT}{{\bf T}}

\newcommand{\bfZ}{{\bf Z}}

\newcommand{\bbC}{\mathbb C}

\newcommand{\bbR}{\mathbb R}
\newcommand{\bbS}{\mathbb S}

\newcommand{\calA}{\mathcal A}

\newcommand{\calC}{\mathcal C}
\newcommand{\calD}{\mathcal D}

\newcommand{\calH}{\mathcal H}

\newcommand{\calL}{\mathcal L}

\newcommand{\calO}{\mathcal O}

\newcommand{\calQ}{\mathcal Q}
\newcommand{\calR}{\mathcal R}
\newcommand{\calS}{\mathcal S}
\newcommand{\calT}{\mathcal T}
\newcommand{\calU}{\mathcal U}

\newcommand{\covD}{{}^{(A)}\bfD}			
\newcommand{\covT}{\bfT}			
\newcommand{\covZ}{\bfZ}
\newcommand{\covS}{\bfS}			
\newcommand{\covN}{\bfN}			
\newcommand{\covBox}{{}^{(A)} \Box}		
\newcommand{\covsD}{{}^{(A)} \! \not \!\! \bfD}		
\newcommand{\covud}{{}^{(A)}\ud}			
\newcommand{\covdlt}{{}^{(A)}\dlt}			
\newcommand{\covLD}{{}^{(A)}\calL}			
\newcommand{\p}{+}
\newcommand{\m}{-}
\newcommand{\np}{}					

\newcommand{\hT}{\tau}				
\newcommand{\hY}{y}				

\newcommand{\met}{\eta}				
\newcommand{\LD}{\calL}	
\newcommand{\Vol}{\sgm}				
\newcommand{\ed}{\bfe}				
\newcommand{\EM}{\calQ}			
\newcommand{\vC}[1]{{}^{(#1)} P}		
\newcommand{\sC}[1]{{}^{(#1)} Q}		
\newcommand{\defT}[1]{{}^{(#1)} \pi}		

\newcommand{\covSgmD}{{}^{(A, \Sgm_{0})}\bfD}			
\newcommand{\covSgmTD}{{}^{(A, \Sgm_{t_{0}})}\bfD}			
\newcommand{\covHTD}{{}^{(A, \calH_{\hT_{0}})}\bfD}			
\newcommand{\DD}{\calD^{+}}			

\newcommand{\LieGrp}{\mathfrak{G}}
\newcommand{\LieAlg}{\mathfrak{g}}
\newcommand{\LieBr}[2]{[ #1, #2 ]}
\newcommand{\LieMet}[2]{\langle #1, #2 \rangle_{\LieAlg}}

\newcommand{\llbracket}{[\![}	
\newcommand{\rrbracket}{]\!]}	
\newcommand{\bbrk}[1]{\llbracket #1 \rrbracket}

\newcommand{\CSD}{\mathrm{CSD}}
\newcommand{\CSH}{\mathrm{CSH}}

\newcommand{\frkN}{\mathfrak{N}}

\newcommand{\bq}{\begin{equation}}
\newcommand{\eq}{\end{equation}}

\newcommand{\ga}{\gamma}

\newcommand{\bbr}{{ \mathbb{R}  }}

\newcommand{\itr}[1]{^{(#1)}}

\begin{document}

\title[]{Small data global existence and decay for \\ relativistic Chern--Simons equations}
\author{Myeongju Chae}%
\address{Department of Mathematics, Hankyung University, Anseong-si, Gyeonggi-do, Korea}%
\email{mchae@hknu.ac.kr}%

\author{Sung-Jin Oh}%
\address{Department of Mathematics, UC Berkeley, Berkeley, CA, USA}%
\email{sjoh@math.berkeley.edu}%

\thanks{The authors thank Hyungjin Huh for helpful discussions. M.~Chae was partially supported by NRF-2011-0028951. S.-J. Oh is a Miller Research Fellow, and acknowledges support from the Miller Institute. }%

\begin{abstract}
We establish a general small data global existence and decay theorem for Chern--Simons theories with a general gauge group, coupled with a massive relativistic field of spin 0 or 1/2. Our result applies to a wide range of relativistic Chern--Simons theories considered in the literature, including the abelian/non-abelian self-dual Chern--Simons--Higgs equation and the Chern--Simons--Dirac equation. A key idea is to develop and employ a gauge invariant vector field method for relativistic Chern--Simons theories, which allows us to avoid the long range effect of charge.
\end{abstract}
\maketitle
\setcounter{tocdepth}{1}
\tableofcontents

%

\section{Introduction}
In this article, we consider Chern--Simons gauge theories with a general gauge group $\LieGrp$ coupled with a massive field of spin 0 (Higgs) or 1/2 (Dirac) on $\bbR^{1+2}$. Our main result (Theorems~\ref{thm:CSH} and \ref{thm:CSD}) is global existence of a unique solution to small compactly supported initial data, along with sharp decay rates. We give a unified proof that applies to a wide range of relativistic Chern--Simons theories; some well-studied examples include the self-dual Chern--Simons--Higgs equation with abelian and non-abelian gauge groups, and the abelian Chern--Simons--Dirac equation.

As will be explained in Section~\ref{subsec:main-ideas} in more detail, the main difficulty in the proof is the long range effect of charge, which manifests as the slow spatial decay of the magnetic potential. In order to overcome this difficulty, we develop a gauge covariant vector field approach 
\cite{MR1672001, MR2131047, Lindblad:2006vh, Bieri:2014lq} for relativistic Chern--Simons theories, which allows us to avoid the gauge potential in the analysis. An issue for executing this strategy is the anomalous commutation property of the Chern--Simons coupled Klein--Gordon equation; see \eqref{eq:comm-problem} for details. This issue is taken care of by adapting the ODE technique of \cite{MR2188297, MR2056833}, developed to address possible long range effects, to our gauge covariant setting.

We begin by describing the Chern--Simons--Higgs and Dirac equations with a general gauge group in Sections~\ref{subsec:CSH} and \ref{subsec:CSD}, respectively. In Section~\ref{subsec:main-results}, we state the main results of the paper in precise terms. Section~\ref{subsec:main-ideas} contains an explanation of the main ideas of our proof, and in Section~\ref{subsec:history} we give a brief discussion of the history of the problem and related results. The introduction ends with a short outline of the rest of the paper in Section~\ref{subsec:outline}.


\subsection{Non-abelian self-dual Chern--Simons--Higgs equation} \label{subsec:CSH}
Here we first give a general formulation of the Chern--Simons--Higgs equation with a general gauge group $\LieGrp$; see \eqref{eq:CSH-general}. This formulation  requires a choice of a real scalar potential $\calU(\varphi)$. We then describe a particular choice of $\calU(\varphi)$ leading to the \emph{self-dual Chern--Simons--Higgs equation} \eqref{eq:CSH}. For concreteness, our first main theorem (Theorem~\ref{thm:CSH}) is stated for this equation, but our proof is clearly valid for more general potentials $\calU(\varphi)$; see Remark~\ref{rem:CSH-general}. Important special cases of \eqref{eq:CSH} include the \emph{abelian self-dual equation} ($\LieGrp = \mathrm{U}(1)$ and $V = \bbC$) and the \emph{non-abelian self-dual equation with adjoint coupling} ($\LieGrp = \mathrm{SU}(n)$ and $V = \mathrm{sl}(n; \bbC)$); see Examples~\ref{ex:a-CSH} and \ref{ex:na-CSH} below.

Consider a Lie group $\LieGrp$ with the associated Lie algebra $\LieAlg$, which possesses a positive-definite metric $\LieMet{\cdot}{\cdot}$ that is bi-invariant (i.e., invariant under the adjoint action $\LieGrp \times \LieAlg \ni (g, a) \mapsto g a g^{-1} \in \LieAlg$).
Let $V$ be a complex vector space equipped with an inner product $\brk{\cdot, \cdot}_{V}$, on which the group $\LieGrp$ acts via a unitary representation $\rho: \LieGrp \to \mathrm{U} (V)$.
In what follows, the subscript $V$ in $\brk{\cdot, \cdot}_{V}$ will often be omitted.

\begin{remark} 
When $\LieGrp$ is \emph{compact}, which is the case in all examples below, a bi-invariant metric always exists, since any left-invariant metric can be made bi-invariant by averaging its right-translates using the Haar measure; recall that the Haar measure is finite and bi-invariant on compact Lie groups \cite[Chapter 1]{MR3136522}. Moreover, for any representation $\rho: \LieGrp \to \mathrm{GL}(V)$ there exists an inner product on $V$ which makes $\rho$ unitary, by starting with any inner product $\brk{\cdot, \cdot}$ and averaging its left-translates $\brk{\rho(g) \, \cdot , \rho(g) \, \cdot }$ using the Haar measure. 
\end{remark}

Let $\bbR^{1+2}$ denote the (2+1)-dimensional Minkowski space equipped with the metric 
\begin{equation*}
\eta_{\mu \nu} = (\eta^{-1})^{\mu \nu} = \mathrm{diag} \, (\m1,\p1,\p1)
\end{equation*}
in the rectilinear coordinates $(x^{0}, x^{1}, x^{2})$.
Let $E$ be a vector bundle with fiber $V$ over $\bbR^{1+2}$ with structure group $\LieGrp$.
We refer to the sections of $E$ as \emph{scalar multiplet fields}.
Since $\bbR^{1+2}$ is contractible, every fiber bundle over this space is trivial, i.e.,
$E$ is (smoothly) equivalent to the product bundle $\bbR^{1+2} \times V$. Hence the scalar multiplet fields may be concretely realized as the $V$-valued functions on $\bbR^{1+2}$; see Section~\ref{subsec:gauge-str} below. 

In order to differentiate a scalar multiplet field, we introduce the notion of a \emph{covariant derivative} $\covD$ on $E$,  described by a $\LieAlg$-valued 1-form $A(\cdot)$ in the following fashion:
\begin{equation} \label{eq:covd-def}
	\covD_{X} \varphi = \nb_{X} \varphi + A(X) \cdot \varphi.
\end{equation}
Here $\varphi$ is a scalar multiplet field (i.e., a $V$-valued function), $X$ is a vector on $\bbR^{1+2}$, $\nb_{X}$ is the usual directional derivative of $\varphi$ (viewed as a $V$-valued function) in the direction $X$ and $A(X) \in \LieAlg$ acts on $\varphi$ by the infinitesimal representation $\ud \rho \restriction_{I}: \LieAlg \to \mathrm{u}(V)$. 
Given two vector fields $X, Y$ on $\bbR^{1+2}$, the associated \emph{curvature 2-form} $F = F[A]$ is defined by the relation
\begin{equation} \label{eq:curv-def}
	F(X, Y) \varphi = \big( \covD_{X} \covD_{Y}  - \covD_{Y} \covD_{X} - \covD_{\LieBr{X}{Y}} \big) \varphi.
\end{equation}
In terms of the connection 1-form $A$, the curvature $F$ takes the form
\begin{equation} \label{eq:curv-eq1}
	F = \ud A + \frac{1}{2} [A \wedge A]
\end{equation}
(see Section~\ref{subsec:extr-calc} for the notation) or in coordinates, 
\begin{equation}  \label{eq:curv-eq2}
	F_{\mu \nu} = \rd_{\mu} A_{\nu} - \rd_{\nu} A_{\mu} + \LieBr{A_{\mu}}{A_{\nu}}.
\end{equation}
%
%
In analogy with Maxwell's theory of electromagnetism, the $F_{12}$ component of the curvature 2-form is sometimes called the \emph{magnetic field}, and $F_{01}, F_{02}$ are referred to as the \emph{electric field}. The components $A_{1}, A_{2}$ are alternatively referred to as the \emph{magnetic potential}, and $A_{0}$ as the \emph{electric potential}.

The Lagrangian density for the Chern--Simons--Higgs system is given by
\begin{equation*}
	L[A, \varphi] = \frac{\kpp}{2} L_{CS}[A] - \brk{\covD^{\mu} \varphi, \covD_{\mu} \varphi} - \calU(\varphi)
\end{equation*}
where $\kpp \in \bbR \setminus \set{0}$ is called the coupling constant and $\calU(\varphi)$ is a real-valued scalar potential. The term $L_{CS}[A]$ is the \emph{Chern--Simons Lagrangian}, defined as
\begin{equation*}
	L_{CS}[A] = \eps^{\mu \nu \rho} \bb( \brk{A_{\mu}, \rd_{\nu} A_{\rho}}_{\LieAlg} + \frac{1}{3} \brk{A_{\mu}, [A_{\nu}, A_{\rho}]}_{\LieAlg} \bb).
\end{equation*}
We say that $(A, \varphi)$ is a solution to the \emph{Chern--Simons--Higgs equation} if it is a formal critical point of the action $(A, \varphi) \mapsto \calS[A, \varphi] = \int_{\bbR^{1+2}} L[A, \varphi] \, \ud t \ud x$. The corresponding Euler--Lagrange equation satisfied by the formal critical points takes the form
\begin{equation}  \label{eq:CSH-general}
\left\{
\begin{aligned}
	\covBox \varphi =& \frac{1}{2} \frac{\dlt \calU}{\dlt \varphi}, \\
	F = & \frac{1}{\kpp} (\star J_{\CSH}), \\
	J_{\CSH} =&  \brk{\calT \varphi, \covud \varphi} + \brk{\covud \varphi, \calT \varphi}.
\end{aligned}
\right.
\end{equation}
Here $\covBox = \covD^{\mu} \, \covD_{\mu}$ is the covariant d'Alembertian, $\covud \varphi = \covD_{\mu} \varphi \, \ud x^{\mu}$ is the covariant differential of $\varphi$ and $\star$ is the Hodge star (see Section~\ref{subsec:extr-calc}). The notation $\frac{\dlt \calU}{\dlt \varphi}$ refers to the functional derivative of $\calU(\varphi)$, characterized by
\begin{equation*}
	\frac{\ud}{\ud s} \bb\vert_{s = 0} \bb( \int_{\bbR^{1+2}} \calU(\varphi + s f)  \, \ud t \ud x \bb) 
	= \int_{\bbR^{1+2}} \Re \brk{\frac{\dlt \calU}{\dlt \varphi}, f} \, \ud t \ud x
\end{equation*}
for all $V$-valued $f \in C^{\infty}_{0}(\bbR^{1+2})$. The linear operator $\calT : V \to \LieAlg \otimes_{\bbR} V$ is defined as follows: 
Given an orthonormal basis $\set{e_{A}} \subseteq \LieAlg$ with respect to $\brk{\cdot, \cdot}_{\LieAlg}$, let
\begin{equation*}
	\calT v = \sum_{A} e_{A} \otimes \calT^{A} v, \quad \hbox{ where } \calT^{A} : V \to V, \ v \mapsto e_{A} \cdot v \hbox{ for each index } A.
\end{equation*}
A compact way of denoting $\calT^{A}$ while respecting the difference between upper and lower indices is to write $\calT^{A} v = \sum_{A'} \dlt^{A A'} e_{A'} \cdot v$, where $\dlt^{AA'}$ is the diagonal symbol that equals $1$ when $A = A'$ and vanishes otherwise.
The inner product between $\calT v \in \LieAlg \otimes_{\bbR} V$ and $w \in V$ is naturally defined to be an element of $\LieAlg$ by the formula
\begin{equation*}
	\brk{\calT v, w} = \sum_{A} \brk{\calT^{A} v, w} e_{A}, \quad
	\brk{w, \calT v} = \sum_{A} \brk{w, \calT^{A} v} e_{A}.
\end{equation*}
According to these definitions, note that $\brk{a \cdot v, w} = \brk{a, \brk{\calT v , w}}_{\LieAlg}$ for $a \in \LieAlg$ and $v, w \in V$.
Note also that $\brk{\calT \varphi, \covud \varphi}$ and $\brk{\covud \varphi, \calT \varphi}$ in \eqref{eq:CSH-general} define $\LieAlg$-valued 1-forms.

In the study of the Chern--Simons--Higgs equation, a special emphasis is given to the \emph{self-dual} case, in which the energy functional has a particular structure so that its minima can be found by solving a simpler first order elliptic equation (Bogomol'nyi equation).
In this case, the scalar potential $\calU$ is given by 
\begin{equation} \label{eq:CSH-ptnl-lagrange}
\calU(\varphi) 
= \frac{1}{\kpp^{2}} \bb\vert \dlt_{AA'}  \brk{\calT^{A} \varphi, \varphi} \calT^{A'} \varphi+ v^{2} \varphi \bb\vert^{2},
\end{equation}
where $v \in \bbR \setminus \set{0}$ is a constant playing the role of the mass parameter for $\varphi$.
Computing the functional derivative, we are led to the following \emph{self-dual Chern--Simons--Higgs equation}:
\begin{equation} \label{eq:CSH} \tag{CSH}
\left\{
\begin{aligned}
	\covBox \varphi - \frac{v^{4}}{\kpp^{2}} \varphi =& U_{\CSH}(\varphi), \\
	F = & \frac{1}{\kpp} (\star J_{\CSH}), \\
	J_{\CSH} =&  \brk{\calT \varphi, \covud \varphi} + \brk{\covud \varphi, \calT \varphi}.
\end{aligned}
\right.
\end{equation}
where
\begin{equation} \label{eq:CSH-ptnl}
\begin{aligned}
	U_{\CSH}(\varphi) = & \frac{4 v^{2}}{\kpp^{2}} \dlt_{A A'} \brk{\calT^{A} \varphi, \varphi} \calT^{A'} \varphi \\
			& + \frac{1}{\kpp^{2}} \dlt_{A A'} \dlt_{B B'} \brk{\calT^{A} \varphi, \varphi} \brk{(\calT^{A'} \calT^{B'} + \calT^{B'} \calT^{A'}) \varphi, \varphi} \calT^{B} \varphi  \\
			& + \frac{1}{\kpp^{2}} \dlt_{A A'} \dlt_{B B'}\brk{\calT^{A} \varphi, \varphi} \brk{\calT^{B} \varphi, \varphi} \calT^{A'} \calT^{B'} \varphi .
\end{aligned}
\end{equation}
Our first main theorem (Theorem~\ref{thm:CSH}) is small data global existence for the general self-dual Chern--Simons--Higgs equation \eqref{eq:CSH}. We remark that Theorem~\ref{thm:CSH} is stated for \eqref{eq:CSH} only for the sake of concreteness. In fact, due to the perturbative nature of the proof, self-duality is not essential for this theorem to hold; see Remark~\ref{rem:CSH-general}.

We now describe important special cases of \eqref{eq:CSH}. We begin with the case of the abelian gauge group $\LieGrp = \mathrm{U}(1)$, which has been extensively studied.
\begin{example}[Abelian self-dual Chern--Simons--Higgs {\cite[Section~IV.A]{dunne1995self}}] \label{ex:a-CSH}
Let $\LieGrp = \mathrm{U}(1) = \set{e^{i \tht} \in \bbC}$, so that $\LieAlg = \mathrm{u}(1) = i \bbR$ and $\brk{i a, i b}_{\LieAlg} = ab$ for $a, b \in \bbR$. Take $V = \bbC$, equipped with the usual inner product $\brk{z, w} = z \overline{w}$, and let $\rho(e^{i \tht}) z= e^{i \tht} z$ for $e^{i \tht} \in \mathrm{U}(1)$ and $z \in \bbC$. Using $i$ as a basis for $\LieAlg = \mathrm{u}(1)$, we may write $\calT v = i v$ and $\covD = \nb + i A$ for a real-valued 1-form $A$. Therefore,
\begin{equation*}
	J_{\CSH} = i \big( \varphi \overline{\covud \varphi} - \overline{\varphi} \, \covud \varphi \big).
\end{equation*}
The self-dual potential is given by
\begin{equation*}
	\calU(\varphi) = \frac{1}{\kpp^{2}} \abs{\varphi}^{2} \big( \abs{\varphi}^{2} - v^{2} \big)^{2}.
\end{equation*}
for some $v \in \bbR \setminus \set{0}$. Hence $U_{\CSH}(\varphi)$ takes the form
\begin{equation*}
	U_{\CSH}(\varphi) = \frac{1}{\kpp^{2}} \bb( - 4 v^{2} \abs{\varphi}^{2} \varphi + 3 \abs{\varphi}^{4} \varphi \bb).
\end{equation*}
\end{example}

Another important special case of \eqref{eq:CSH} is when the structure group $\LieGrp$ is $\mathrm{SU}(N)$ $(N > 1)$, and it acts on the space $\mathrm{sl}(N, \bbC)$ (complexification of the Lie algebra $\mathrm{su}(N)$) by the adjoint action.

\begin{example}[Non-abelian self-dual Chern--Simons--Higgs with adjoint coupling {\cite[Section~V.B]{dunne1995self}}] \label{ex:na-CSH} 
%
Let $\LieGrp = \mathrm{SU}(N)$ $(N > 1)$ be the group of $N \times N$ unitary matrices with unit determinant, so that $\LieAlg = \mathrm{su}(N)$ is the Lie algebra of $N \times N$ anti-hermitian matrices with zero trace and $\brk{a, b}_{\LieAlg} = \tr(a b^{\dagger})$ for matrices $a, b$. 
We take the state space to be the complexification of the Lie algebra $\LieAlg = \mathrm{su}(N)$, i.e., $V = \mathrm{sl}(N, \bbC)$ is the space of $N \times N$ complex matrices with zero trace and $\brk{v, w}_{V} = \tr (v w^{\dagger})$. The group $\LieGrp$ acts on $V$ via the adjoint action $\rho(g) v = g v g^{-1}$ for $g \in \mathrm{SU}(N)$ and $v \in \mathrm{sl}(N, \bbC)$.

In this case, the current $J_{\CSH}$ and the self-dual potential $\calU(\varphi)$ take the form
\begin{align*}
	J_{\CSH} =& -\LieBr{\varphi^{\dagger}}{\covud \varphi} + \LieBr{(\covud \varphi)^{\dagger}}{\varphi}, \\
	\calU(\varphi) = & \frac{1}{\kpp^{2}} \bb\vert \LieBr{\LieBr{\varphi}{\varphi^{\dagger}}}{\varphi} + v^{2} \varphi \bb\vert^{2},
\end{align*}
for some $v \in \bbR \setminus \set{0}$. Hence $U_{\CSH}(\varphi)$ is given by
\begin{equation*}
	U_{\CSH}(\varphi) = \frac{4 v^{2}}{\kpp^{2}} \LieBr{\varphi}{\LieBr{\varphi}{\varphi^{\dagger}}}
				+ \frac{1}{\kpp^{2}} \bb( 
				2 \LieBr{\LieBr{\varphi}{\LieBr{\varphi^{\dagger}}{\LieBr{\varphi}{\varphi^{\dagger}}}}}{\varphi}
				+ \LieBr{\LieBr{\varphi}{\LieBr{\varphi}{\varphi^{\dagger}}}}{\LieBr{\varphi}{\varphi^{\dagger}}} \bb).
\end{equation*}
\end{example}

\subsection{Non-abelian Chern--Simons--Dirac equations} \label{subsec:CSD}
Here we describe the Chern--Simons--Dirac equation with a general gauge group $\LieGrp$. Our formulation includes the well-studied abelian case \cite{MR2290338} as a special case; see Example~\ref{ex:a-CSD}.

Let $\LieGrp$ be a Lie group with a positive-definite bi-invariant metric, $W$ be a complex vector space with an inner product $\brk{\cdot, \cdot}_{W}$, and $\rho : \LieGrp \to U(W)$ be a unitary representation. In order to describe the Chern--Simons--Dirac system with a general gauge group $\LieGrp$, we first need to describe the \emph{spinor multiplet fields} on $\bbR^{1+2}$.

Let $\gmm^{\mu}$ ($\mu = 0,1,2$) be the \emph{gamma matrices}, which are $\bbC$-valued $2 \times 2$ matrices satisfying 
\begin{equation} \label{eq:gmm-mat}
	\gmm^{\mu} \gmm^{\nu} + \gmm^{\nu} \gmm^{\mu} = - 2 (\eta^{-1})^{\mu \nu} \, \bfI_{2 \times 2} \, .
\end{equation}
The standard representations of $\gmm^{\mu}$ are given by 
\begin{align*}
\gmm^{0} = \left(
\begin{array}{cc}
1 & 0 \\ 0 & - 1
\end{array}
\right), \quad
\gmm^{1} = \left(
\begin{array}{cc}
0 & 1 \\ -1 & 0
\end{array}
\right), \quad
\gmm^{2} = \left(
\begin{array}{cc}
0 & -i \\ -i & 0
\end{array}
\right).
\end{align*}
The space of \emph{spinors} associated to the Minkowski space $(\bbR^{1+2}, \eta)$ is simply $\Dlt = \bbC^{2}$, on which the gamma matrices act by matrix multiplication, and the spinor bundle is the trivial bundle $S = \bbR^{1+2} \times \Dlt$. Let $\tilde{E}$ be a vector bundle with fiber $W$ and structure group $\LieGrp$. 
The bundle of \emph{spinor multiplets} is the tensor product $E = S \otimes_{\bbC} \tilde{E}$, whose fiber is $V = \Dlt \otimes_{\bbC} W$.
Using the triviality of the bundle $E$, we will identify the sections (or \emph{spinor multiplet fields}) of $E$  with $V$-valued functions on $\bbR^{1+2}$.

The gamma matrices $\gmm^{\mu}$ and the elements $g \in \LieGrp$, $a \in \LieAlg$ act on $V$ by the rules
\begin{equation*}
	\gmm^{\mu} (s \otimes w) = \gmm^{\mu} s \otimes w, \quad
	g \cdot (s \otimes w) = s \otimes \rho(g) w, \quad 
	a \cdot (s \otimes w) = s \otimes (\ud \rho \restriction_{I} \! (a) w), 
\end{equation*}
where $s \in \Dlt$ and $w \in W$. Moreover, the inner products on $\Dlt$ and $W$ induce an inner product $\brk{\cdot, \cdot}_{V}$ on $V$, characterized by
\begin{equation*}
	\brk{s_{1} \otimes w_{1}, s_{2} \otimes w_{2}}_{V}
	= (s_{2}^{\dagger} s_{1}) \brk{w_{1}, w_{2}}_{W},
\end{equation*}
where $s_{1}, s_{2} \in \Dlt$ and $w_{1}, w_{2} \in W$. Note that $\gmm^{0}$ is hermitian, $\gmm^{j}$ ($j=1,2$) is anti-hermitian, $g \in \LieGrp$ is unitary and $a \in \LieAlg$ is anti-hermitian with respect to $\brk{\cdot, \cdot}_{V}$. 

Given a $\LieAlg$-valued connection 1-form $A$, a spinor multiplet field $\psi$ and a vector $X$ on $\bbR^{1+2}$, we define the gauge covariant derivative $\covD_{X}$ in the direction $X$ associated to $A$ by
\begin{equation} \label{eq:covd-def-CSD}
	\covD_{X} \psi = \nb_{X} \psi +A(X) \cdot \psi.
\end{equation}
The curvature 2-form $F$ is defined by \eqref{eq:curv-def} as in the case of Chern--Simons--Higgs. In addition, we introduce the \emph{covariant Dirac operator}, defined by
\begin{equation*}
	\covsD := \gmm^{\mu} \,  \covD_{\rd_{\mu}} .
\end{equation*}

The Chern--Simons--Dirac Lagrangian density is given by
\begin{equation*}
	L[A, \psi] = \frac{\kpp}{2} L_{CS} + i \brk{\covsD \psi, \gmm^{0} \psi} + m \brk{\psi, \gmm^{0} \psi}.
\end{equation*}
where $\kpp \neq 0$ is the coupling constant and $m > 0$ is the mass of the spinor multiplet field $\psi$ and $\brk{\cdot, \cdot} = \brk{\cdot, \cdot}_{V}$. 
The \emph{Chern--Simons--Dirac equation} for $(A, \psi)$ is the Euler--Lagrange equation for the action $\calS[A, \psi]  = \int_{\bbR^{1+2}} L[A, \psi] \, \ud t \ud x$, and takes the form
\begin{equation}\label{eq:CSD} \tag{CSD}
\left\{
\begin{aligned}
i \covsD \psi + m \psi =& 0 \\
F =& \frac{1}{\kpp} (\star J_{\CSD}) \\
J_{\CSD}(\rd_{\mu}) =& - i \eta_{\mu \nu} \brk{\gmm^{0} \gmm^{\nu} \calT \psi, \psi}.
\end{aligned}
\right.
\end{equation}
Here $\calT v \in \LieAlg \otimes V$ for $v \in V$ is again defined as $\calT v = \sum_{A} e_{A} \otimes \calT^{A} v$ with $\calT^{A} v = \sum_{A'} \dlt^{A A'} e_{A'} \cdot v$, where $\set{e_{A}}$ is any orthonormal basis for $\LieAlg$ with respect to $\brk{\cdot, \cdot}_{\LieAlg}$. The matrix $\gmm^{0} \gmm^{\nu}$ acts on $\calT v$ in the natural fashion, i.e., $\gmm^{0} \gmm^{\nu} \calT v = \sum_{A} e_{A} \otimes \gmm^{0} \gmm^{\nu} \calT^{A} v$.

An important special case of \eqref{eq:CSD} is when the gauge group is abelian, i.e., $\LieGrp = \mathrm{U}(1)$.
\begin{example}[Abelian Chern--Simons--Dirac \cite{MR2290338}] \label{ex:a-CSD}
Let $\LieGrp = \mathrm{U}(1)$, $\LieAlg = \mathrm{u}(1) = i \bbR$ and $\brk{i a, i b}_{\LieAlg} = ab$ for $a, b \in \bbR$. Taking $W = \bbC$ with the usual action of $\mathrm{U}(1)$, we have the natural equivalence $V = \Dlt \otimes_{\bbC} \bbC \cong \Dlt = \bbC^{2}$ and $e^{i \tht} \in \mathrm{U}(1)$ acts on this space by component-wise multiplication. Then the 1-form $J_{\CSD}$ takes the form
\begin{equation*}
	J_{\CSD}(\rd_{\mu}) = \eta_{\mu \nu} (\psi^{\dagger} \gmm^{0} \gmm^{\nu} \psi).
\end{equation*}
\end{example}


\subsection{Main theorems} \label{subsec:main-results}
We now state precisely the main theorems of this paper, which are small data global existence and decay results for the general Chern--Simons--Higgs and Dirac equations formulated above.

We begin with the case of \eqref{eq:CSH}.
We say that a triplet $(a, f, g)$ of a $\LieAlg$-valued 1-form $a = a_{1} \ud x^{1} + a_{2} \ud x^{2}$ and $V$-valued functions $f, g$ on $\Sgm_{0} = \set{0} \times \bbR^{2}$ is an \emph{initial data set for \eqref{eq:CSH}} if it obeys the \emph{\eqref{eq:CSH} constraint equation}, i.e.,
\begin{equation}
	\rd_{1} a_{2} - \rd_{2} a_{1} + \LieBr{a_{1}}{a_{2}} = - \frac{1}{\kpp} \bb( \brk{\calT f, g} + \brk{g, \calT f} \bb).
\end{equation}
We say that $(A, \varphi)$ is a solution to the initial value problem (IVP) for \eqref{eq:CSH} with data $(a, f, g)$ if $(A, \varphi)$ solves \eqref{eq:CSH} and obeys
\begin{equation*}
	(A , \varphi, \covD_{0} \varphi) \restriction_{\Sgm_{0}} = (a, f, g),
\end{equation*}
where the notation $\restriction_{\Sgm_{0}}$ refers to the pullback along the embedding $\Sgm_{0} \hookrightarrow \bbR^{1+2}$; see the end of Section~\ref{subsec:extr-calc} for the precise definition.
Note that the constraint equation is precisely the pullback of the equation $F = \frac{1}{\kpp} \star J_{\CSH}$ along the embedding $\Sgm_{0} \hookrightarrow \bbR^{1+2}$; hence it necessarily holds for $(a, f, g)$ if a solution to the IVP exists.

The precise statement of the main theorem for \eqref{eq:CSH} is as follows.
\begin{theorem} \label{thm:CSH}
Consider the IVP for \eqref{eq:CSH} with $v \neq 0$ and $\kpp \neq 0$. There exists a positive function $\dlt_{1}(R)$ of $R \in (0, \infty)$ such that the following holds.
Let $(a, f, g)$ be a smooth initial data set for \eqref{eq:CSH} obeying
\begin{equation} \label{eq:CSH-id}
	\supp \, (f,g) \subseteq B_{R}, \quad
	\sum_{k=1}^{5} \nrm{(\covSgmD^{(k)} f, \covSgmD^{(k-1)} \, g)}_{L^{2}(\bbR^{2})}
	+ \nrm{f}_{L^{2}(\bbR^{2})} \leq \eps,
\end{equation}
where $B_{R} := \set{x \in \bbR^{2} : \abs{x} < R}$ and $\covSgmD$ is the (induced) gauge covariant derivative on $\set{0} \times \bbR^{2}$. If $\eps \leq \dlt_{1}(R)$, then a smooth solution to the IVP exists globally, and it is unique up to smooth local gauge transformations.
Moreover, the solution $(\phi, A)$ exhibits the following gauge invariant asymptotic behavior:
\begin{equation}
	\abs{\phi(t,x)} + \abs{\covD \phi(t,x)} < C \eps (1+\abs{t})^{-1}.
\end{equation} 
\end{theorem}

By \emph{uniqueness up to smooth local gauge transformations}, we mean the following: Given two solutions $(A, \varphi)$, $(A', \varphi')$ to the IVP for \eqref{eq:CSH}, there exists an open covering $\set{O_{\alp}}_{\alp \in \calA}$ of $\bbR^{1+2}$ and smooth functions (local gauge transformations) $\set{U_{\alp} : O_{\alp} \to \LieGrp}_{\alp \in \calA}$ such that the gauge transform of $(A, \varphi) \restriction_{O_{\alp}}$ by $U_{\alp}$ equals $(A', \varphi')$, i.e.,
\begin{equation*}
	(A', \varphi')(t,x) = (U_{\alp} A U_{\alp}^{-1} - \ud U_{\alp} U_{\alp}^{-1}, U_{\alp} \cdot \varphi)(t,x) \quad \hbox{ for every } (t,x) \in O_{\alp}.
\end{equation*}
\begin{remark} \label{rem:CSH-general}
As one may expect from the perturbative nature of the statement, the exact self-duality of \eqref{eq:CSH} is unnecessary for Theorem~\ref{thm:CSH} to hold. It will be clear from our proof that the important points are: $\calU(\varphi)$ has a positive mass term $m^{2} \abs{\varphi}^{2}$ $(m \neq 0)$ and the remaining terms of $\calU(\varphi)$ are quartic or higher in $\varphi$, so that $U(\varphi)$ is cubic or higher.
\end{remark}

Next, we consider the case of \eqref{eq:CSD}. We say that a pair $(a, \psi_{0})$ of a $\LieAlg$-valued 1-form $a = a_{1} \ud x^{1} + a_{2} \ud x^{2}$ and $V = \Dlt \otimes W$-valued functions $\psi_{0}$ on $\Sgm_{0}$ is an \emph{initial data set for \eqref{eq:CSD}} if it obeys the \emph{\eqref{eq:CSD} constraint equation}, i.e.,
\begin{equation} \label{eq:CSD-constraint}
	\rd_{1} a_{2} - \rd_{2} a_{1} + \LieBr{a_{1}}{a_{2}} = - \frac{i}{\kpp} \brk{\calT \psi_{0}, \psi_{0}}
\end{equation}
We say that $(A, \psi)$ is a solution to the IVP for \eqref{eq:CSD} with data $(a, \psi_{0})$ if $(A, \psi)$ solves \eqref{eq:CSD} and obeys
\begin{equation*}
	(A, \psi) \restriction_{\Sgm_{0}} = (a, \psi_{0}).
\end{equation*}
Again, since the constraint equation \eqref{eq:CSD-constraint} is a part of \eqref{eq:CSD}, it necessarily holds for $(a, \psi_{0})$ if a solution to the IVP exists.

We now state our main theorem for \eqref{eq:CSD}.
\begin{theorem} \label{thm:CSD}
Consider the IVP for \eqref{eq:CSD} with $m \neq 0$ and $\kpp \neq 0$. There exists a positive function $\dlt_{2}(R)$ of $R \in (0, \infty)$ such that the following holds: Let $(a, \psi_{0})$ be a smooth initial data set obeying
\begin{equation} \label{eq:CSD-id}
	\supp \, \psi_{0} \subseteq B_{R}, \quad \sum_{k=0}^{5} \nrm{\covSgmD^{(k)} \psi_{0}}_{L^{2}(\bbR^{2})} < \eps.
\end{equation}
If $\eps \leq \dlt_{2}(R)$, then a smooth solution to the IVP exists globally on $\bbR^{1+2}$, and it is unique up to smooth local gauge transformations. Moreover, the solution $(\psi, A)$ exhibits the following gauge invariant asymptotic behavior:
\begin{equation}
	\abs{\psi(t,x)} + \abs{\covD \psi(t,x)} < C \eps (1+\abs{t})^{-1}.
\end{equation} 
\end{theorem}
The notion of uniqueness up to smooth local gauge transformations is defined as in the case of \eqref{eq:CSH}. 

We conclude this section with a few remarks.
\begin{remark} 
For \eqref{eq:CSH}, global existence and regularity for initial data of arbitrary size have been already established in the abelian case (Example~\ref{ex:a-CSH}); see \cite{Chae:2002eu, Selberg:2012vb, Oh:2013bq}. 
This result is essentially proved by iterating a local well-posedness theorem with the help of the conserved energy of the system, with respect to which \eqref{eq:CSH} is subcritical. Even when global regularity is known, however, Theorem \ref{thm:CSH} provides complementary information about the asymptotic decay of the solution, at least in the regime of small compactly supported initial data. 

On the other hand, for \eqref{eq:CSD} a similar global regularity statement is not available even in the abelian case; to our knowledge, Theorem~\ref{thm:CSD} is the first global existence result for \eqref{eq:CSD}.
\end{remark}

\begin{remark} 
Dependence of $\dlt_{1}$ and $\dlt_{2}$ on the size $R$ of the support of the matter field is a technical condition, which is common in the literature of nonlinear Klein-Gordon equations. It arises from the use of foliation by hyperboloids (see Subsection \ref{subsec:polar-coords}), which only covers the domain of dependence of a ball in $\set{t=0}$ (or equivalently, an outgoing null cone). One idea for removing this condition is to prove a separate global existence and decay theorem in the domain of dependence of $\set{t=0} \setminus B_{R}$. In this region, one may exploit the improved rate of decay for solutions to the free Klein-Gordon equation, namely $t^{-N}$ for any $N$ as opposed to $t^{-1}$ in the case considered in the present paper. 
\end{remark}

%

\subsection{Main ideas} \label{subsec:main-ideas}
In this subsection, we discuss the key difficulties of the problem and thereby motivate the main ideas of the paper.
To keep the discussion simple and concrete, we mostly focus on the special case of the abelian self-dual Chern--Simons--Higgs equation (Example~\ref{ex:a-CSH}), where we furthermore fix $v = \kpp = 1$. Unless otherwise specified, we let $(A, \phi)$ denote a solution to this system on $\bbR^{1+2}$, which is assumed to be smooth and suitably decaying in space.

\subsubsection*{The problem of magnetic charge}
The main difficulty for studying the precise asymptotic behavior of a solution is the possible long range effect of the total magnetic charge of the system, which is defined by
\begin{equation*}
	q = \int_{\set{t} \times \bbR^{2}} F.
\end{equation*}
By integrating the equation $\ud F = \ud^{2} A = 0$ over sets of the form $(t_{1}, t_{2}) \times \bbR^{2}$ and applying Stoke's theorem, it follows that $q$ is conserved in time. On the other hand, integrating $\ud A = F$ over a ball of the form $\set{t} \times B_{R}$, where $t \in \bbR$ and $R > 0$, we see that
\begin{equation*}
	\int_{\set{t} \times \rd B_{R}} A = \int_{\set{t} \times B_{R}} F \to q \quad \hbox{ as } R \to \infty.
\end{equation*}
For generic initial data, the total magnetic charge $q$ would be non-zero. In this case, the preceding computation shows that a part of $A$ has a long range tail $q r^{-1}$ as $r \to \infty$. This behavior is potentially problematic, since upon expansion the covariant Klein--Gordon equation for $\phi$ has a quadratic term of the form $2 i A^{\mu} \rd_{\mu} \phi$ and $r^{-1}$ is not integrable.

\subsubsection*{Gauge covariant vector field method for Chern--Simons theories}
To overcome the above difficulty, we observe that the $r^{-1}$ tail manifests itself in gauge dependent fields, such as $A$, but not for gauge covariant fields, such as $F$. In fact, note that $F$ is compactly supported if $\phi$ is. These considerations suggest that it might be favorable to analyze the long time behavior of solutions to Chern--Simons theories in a \emph{gauge covariant} fashion. 
To this end, we develop and employ a gauge covariant version of the celebrated vector field method, which originated in \cite{MR784477,Klainerman:tc, MR1316662} in the context of the wave equation, and was first used in the context of Klein--Gordon equations in \cite{MR803252}. 
The key idea of the gauge covariant vector field method is to replace the commuting vector fields $Z_{\mu \nu}$ (see Section~\ref{subsec:KillingVF}) by their gauge covariant analogues 
\begin{equation*}
	Z_{\mu \nu} = x_{\mu} \rd_{\nu } - x_{\nu} \rd_{\mu} \quad \mapsto \quad \covZ_{\mu \nu} = \covD_{Z_{\mu \nu}} = x_{\mu} \covD_{\nu} - x_{\nu} \covD_{\mu},
\end{equation*}
expressed in the rectilinear coordinates $(x^{0}, x^{1}, x^{2})$. On one hand, we develop a geometric formalism based on exterior differential calculus and Hodge duality, which seem natural for Chern--Simons theories, to compute iterated commutators of $\covZ_{\mu \nu}$ with the Chern--Simons system (see Section~\ref{sec:comm}). On the other hand, we establish a gauge invariant Klainerman--Sobolev inequality (Proposition~\ref{prop:KlSob}), which converts boundedness of generalized energy constructed by commutation of $\covZ_{\mu \nu}$ to pointwise decay.

A gauge invariant version of the Klainerman--Sobolev inequality for the Klein--Gordon equation was first proved by Psarelli in the work \cite{MR1672001, MR2131047} on the massive Maxwell--Klein--Gordon and Maxwell--Dirac equations in $\bbR^{1+3}$. We remark that a gauge covariant vector field method was employed in the study of massless Maxwell--Klein--Gordon equation in $\bbR^{1+3}$ as well; see \cite{Lindblad:2006vh, Bieri:2014lq}. Furthermore, a suitable version of this method proved to be useful in the small data global existence problem for the closely related Chern--Simons--Schr\"odinger equation in $\bbR^{1+2}$ \cite{Oh:2013tx}. 

\subsubsection*{The problem of anomalous commutation}
The success of the vector field method relies on a good commutation property of the system with the commuting vector fields, which in this case are $\covZ_{\mu \nu}$.
However, it turns out that the Chern--Simons theories exhibit an \emph{anomalous} commutation property with $\covZ_{\mu \nu}$, which is a priori problematic. To demonstrate this issue in more detail, we begin by computing (up to the main term) the commutator between $\covZ_{\mu \nu}$ and the covariant Klein--Gordon operator $\covBox \m 1$ using the Chern--Simons equation $F = \star J$:
\begin{equation} \label{eq:comm-problem}
\begin{aligned}
\LieBr{\covZ_{\mu \nu}}{\covBox - 1} \varphi
= & \iota_{Z_{\mu \nu}} \star ( (\ud J) \wedge \varphi)
- 2 \iota_{Z_{\mu \nu}} \star ( J \wedge \covud \varphi)
+ (\hbox{l.o.t.}) \\
= & N_{cubic} + (\hbox{l.o.t.})
\end{aligned}
\end{equation}
where $(\hbox{l.o.t.})$ denotes terms which are quintic and higher in $\varphi$ and
\begin{equation*}
	N_{cubic} = \iota_{Z_{\mu \nu}} \star (\varphi \wedge \overline{\covud \varphi} \wedge \covud \varphi).
\end{equation*}
A simple computation (see Lemma~\ref{lem:ptwise-N}) shows that, in general, the best one can say is
\begin{equation*}
\abs{N_{cubic}} = \abs{\iota_{Z_{\mu \nu}} \star (\varphi \wedge \overline{\covud \varphi} \wedge \covud \varphi)}
\leq C \abs{\varphi} \abs{\covT \varphi} \abs{\covS \varphi},
\end{equation*}
where $\bfT$ denotes one of $\set{\covD_{0}, \covD_{1}, \covD_{2}}$ and $\bfS$ is the gauge covariant analogue the scaling vector field, i.e., $\bfS = x^{\mu} \, \covD_{\mu}$. 

The appearance of $\bfS \varphi$ is undesirable, since $\covS$ does not commute well with $\covBox \m 1$. Indeed, comparing\footnote{Let $f$ be a solution to the free Klein--Gordon equation $(\Box \m 1) f = 0$ with compactly supported data. In general, the sharp decay rate for $\nb f$ and $f$ is $\hT^{-1}$, where $\hT = \sqrt{t^{2} - r^{2}}$. Since each $Z_{\mu \nu}$ commutes with $\Box \m 1$, note that $\abs{Z_{\mu \nu} f} \aleq \hT^{-1}$ as well. If, in addition, $\abs{S f} \aleq \hT^{-\alp}$ for any $\alp > 0$, then it can be proved that $\abs{\nb f} \aleq \hT^{-\min \set{1+\alp, 2}}$, which is impossible in general.} with the free case, we see that $\abs{\covS \varphi}$ should not exhibit any decay in time. In general, one may hope for uniform boundedness of $\abs{\covS \varphi}$ at best. This fact renders the nonlinearity $N_{cubic}$ essentially quadratic, which is borderline for closing the proof of global existence with only the (gauge invariant) Klainerman--Sobolev inequality.

In fact, even when we assume the sharp decay rate
\begin{equation} \label{eq:sharp-decay}
	\abs{\varphi} + \abs{\covN \varphi} \aleq \eps t^{-1},
\end{equation}
where $\bfN = \frac{1}{\sqrt{t^{2} - r^{2}}} \bfS$ is the normalization of $\bfS$, 
the gauge covariant vector field method discussed so far seems to only lead to a weak decay rate
\begin{equation} \label{eq:weak-decay}
	\abs{\covZ^{(m)} \varphi} + \abs{\covN \covZ^{(m-1)} \varphi} \aleq \eps t^{-1} \log^{m+2} (1+t) 
\end{equation}
due to the above anomalous commutation property. In particular, this decay is insufficient to recover the sharp decay rate \eqref{eq:sharp-decay}.

%

\subsubsection*{Gauge covariant ODE method for decay}
To solve the problem of anomalous commutation, we begin by observing that the equation for $\varphi$ itself without any commutation with $\covZ_{\mu \nu}$,
\begin{equation} \label{eq:KG4phi}
	(\covBox \m 1) \varphi = U_{\CSH}(\varphi),
\end{equation}
is favorable in the sense that $U_{\CSH}(\varphi)$ is at least cubic or higher in $\varphi$, and no nonlinearity containing $\bfS \varphi$ is present. If one is able to work directly with this equation, then one may hope to prove that at least the undifferentiated field $\varphi$ obeys the sharp decay rate $t^{-1}$. 

Fortunately, this is indeed the case. We first rewrite $\covBox \m 1$ as
\begin{equation} \label{eq:KG4phi-ODE}
\begin{aligned}
	(\covBox \m 1) \varphi
	= &- \frac{1}{\hT^{2}} \covD_{\hT}( \hT^{2} \covD_{\hT} \varphi) - \varphi + \lap_{A, \calH_{\hT}}  \varphi \\
	= & \frac{1}{t} \bb[ - \covD_{\hT}^{2} (t \varphi) - t \varphi + O\bb( \frac{\eps^{3}}{t^{1+}} \bb) \bb] 
\end{aligned}
\end{equation}
where $\hT = \sqrt{t^{2} - r^{2}}$ and $\lap_{A, \calH_{\hT}}$ is the covariant on constant $\hT$-hypersurfaces; see Section~\ref{subsec:ODE} for more details. The last equality can be justified just using the weak decay bounds \eqref{eq:weak-decay}. By \eqref{eq:weak-decay}, \eqref{eq:KG4phi} and \eqref{eq:KG4phi-ODE}, it follows that
\begin{equation*}
	\covD_{\hT}^{2} (t \varphi) + t \varphi = O\bb( \frac{\eps^{3}}{t^{1+}} \bb)
\end{equation*}
which may be viewed as \emph{gauge covariant ODE} for $t \varphi$. Multiplying by $\overline{\covD_{\hT} (t \varphi)}$ and integrating in $\hT$, we recover the sharp decay rate \eqref{eq:sharp-decay} in terms of the initial data, which allows us to close the whole proof. 


We note that such an ODE technique has been used effectively in non-gauge covariant setting to handle nonlinear Klein--Gordon equations exhibiting modified scattering; see, for instance, \cite{MR2188297}. We also the mention the work \cite{MR2056833}, where a similar ODE technique was used.

\subsubsection*{Squaring the covariant Dirac equation}
Finally, we remark that it is possible to treat Chern--Simon--Dirac with mass on the same footing as Chern--Simons--Higgs, using the well-known fact that squaring the (covariant) Dirac operator leads to a (covariant) Klein--Gordon operator. The lower order terms turn out to be cubic in $\psi$, which is acceptable; see Section~\ref{subsec:uni} for more details. We remark that the same observation was used by Psarelli \cite{MR2131047} to treat the small data global existence problem for the massive Maxwell--Dirac equation in $\bbR^{1+3}$ essentially in the same fashion as the massive Maxwell--Klein--Gordon equation in the same spacetime.


\subsection{History of the problem and related results} \label{subsec:history}
The  relativistic Chern--Simons model in $\bbr^{1+2}$ was first suggested by Hong--Kim--Pac \cite{hong:1990} and Jackiw--Weinberg \cite{jackiw:1990}
to study vortex solutions of the Abelian Higgs model carrying both electric and magnetic charges. 
When the potential in the Lagrangian is self-dual (Example \ref{ex:a-CSH}), the minimum of energy is saturated if and only if $(A, \varphi)$ satisfies a simpler system of first order equations called the self-dual equations, or  the Bogomol'nyi equations. The self-dual equations can be further reduced to a single elliptic equation by the Jaffe--Taubes reduction \cite{jaffe:1980}. According to boundary conditions $|\varphi| \to 0$ or $|\varphi|\to 1$ at infinity, the solutions are called topological or  non-topological, respectively. The topological solution was constructed earlier by Wang \cite{wang:1991}. The general multi-vortex non-topological solution was later constructed by Chae--Imanuvilov \cite{chae:2000}. 

The relativistic non-abelian Chern--Simons model was proposed by Kao and Lee \cite{kao:1994},
and  Dunne \cite{du1, du2}. The supersymmetric Chern--Simons model was discussed in \cite{gu1, gu2, lo}.  Topological solutions were constructed by Yang \cite{yang:1997}. The existence of non-topological solutions 
was obtained  very recently and the general theory is still limited. 
For the recent developments  we refer to \cite{ao:2014, lin:2013, huang:2014, choe:2015}. Most of the known results consider  $\mathcal B = SU(3)$ as the gauge group.

In recent years, the initial value problem for relativistic Chern--Simons theories has been studied by many authors. Most of the work in the literature (to the best of our knowledge) concern well-posedness of such equations under a certain gauge condition. 

The most investigated case so far is the abelian Chern--Simons--Higgs equation (Example~\ref{ex:a-CSH}). This equation is \emph{energy subcritical}: After neglecting the lower order linear and cubic terms in the potential, the scaling critical Sobolev space is $(\phi, \rd_{t} \phi) \in \dot{H}_{x}^{1/2} \times \dot{H}_{x}^{-1/2}$, whereas the energy (essentially) controls the $\dot{H}_{x}^{1} \times L_{x}^{2}$ norm. Global well-posedness of the IVP with sufficiently smooth initial data was proved by Chae--Choe \cite{Chae:2002eu} in the Coulomb gauge $\rd_{1} A_{1} + \rd_{2} A_{2} = 0$, by combining higher order energy estimate with the Br\'ezis--Gallouet inequality \cite{MR582536}.  Afterwards, building on the work of Huh \cite{MR2274820, MR2812958} and Bournaveas \cite{MR2539222} on low regularity local well-posedness, global well-posedness for arbitrary finite energy data was established by Selberg--Tesfahun \cite{Selberg:2012vb} in the Lorenz gauge $- \rd_{0} A_{0} + \rd_{1} A_{1} + \rd_{2} A_{2} = 0$. 
The regularity condition for local well-posedness has been subsequently improved in various gauges (Lorenz, Coulomb and temporal $A_{0} = 0$) by various authors \cite{Oh:2012uq, Oh:2013bq, Pecher:2014pd, Pecher:2015yg}. 

The local well-posedness theory for the abelian Chern--Simons--Dirac equation (Example \ref{ex:a-CSD}) parallels that of the abelian Chern--Simons--Higgs equation; see \cite{MR2290338, Oh:2012uq, MR3163407, Bournaveas:2013vk, Pecher:2014ul}.
 We note however that the scaling critical Sobolev space for this equation is $\psi \in L_{x}^{2}$, which coincides with the only known coercive conserved quantity of the equation (charge). Consequently, large data global well-posedness is far more difficult to establish in the Dirac case compared to the Higgs case, and remains a major open problem.

The IVP for relativistic Chern--Simons equations with general non-abelian gauge groups has not been addressed much in the literature.
In the small data case, the local well-posedness theory in the abelian case extends without much difficulty. However, new issues arise when considering data of arbitrary size. For instance, the classical result of Uhlenbeck (see Proposition~\ref{prop:id-temporal}) on the existence of a regular gauge transformation into the Coulomb gauge requires a certain smallness condition, which makes the existing proofs of global well-posedness of the abelian Chern--Simons--Higgs equation fail in the non-abelian case. Nevertheless, in the forthcoming work of the second author, global well-posedness for any finite energy data is proved for the Chern--Simons--Higgs equation with general non-abelian gauge groups, using the Yang--Mills heat flow gauge introduced in \cite{MR3190112, MR3357182}.

Finally, we mention the recent development concerning a non-relativistic version of Chern--Simons theory, namely the (abelian) \emph{Chern--Simons--Schr\"odinger} equation. This equation is critical with respect to the conserved mass, i.e., the $L^{2}$-norm of the Schr\"odinger field. After the initial work of Berg\'e--de~Bouard--Saut \cite{MR1328596}, local well-posedness for data small in $H^{s}$ for any $s > 0$ was established in the interesting work of Liu--Smith--Tataru \cite{MR3286341} using the heat gauge $A_{0} = \rd_{1} A_{1} + \rd_{2} A_{2}$. 

We are aware of two works on Chern--Simons--Schr\"odinger equation on the global in time behavior of the solutions. One is the recent work of Liu--Smith \cite{Liu:2013xr}, where large data global well-posedness and scattering for subthreshold mass was established under equivariance symmetry. Another is the work \cite{Oh:2013tx} of the second author with Pusateri, where analogue of the main theorems of this paper (i.e., global existence and optimal pointwise decay rate of the solution with small localized data) was established for Chern--Simons--Schr\"odinger without any symmetry assumptions. In fact, by revealing a new genuinely cubic null structure of the Chern--Simons--Sch\"odinger equation in the Coulomb gauge, it was furthermore proved in \cite{Oh:2013tx} that the solutions scatter to free waves in this gauge. At the moment, scattering to free waves in any gauge is open for \eqref{eq:CSH} and \eqref{eq:CSD}. 

\subsection{Structure of the paper} \label{subsec:outline}
This paper is structured as follows.
\begin{itemize} 
\item In Section~\ref{sec:setup}, the basic geometric setup (e.g., tensor notation, vector bundles, exterior differential calculus, Killing vector fields etc.) is given. 
\item Next, in Section~\ref{sec:reduction}, preliminary reductions of the main theorems are performed. For instance, a unified system of equations \eqref{eq:CS-uni} is introduced, which allows us to treat \eqref{eq:CSH} and \eqref{eq:CSD} concurrently. By the end of this section, the proof of the main theorems is reduced to showing the main a priori estimates, Proposition~\ref{prop:main}.
\item In Section~\ref{sec:covVF}, the main analytic tools of the paper are presented, including a gauge covariant vector field method (energy inequality and Klainerman--Sobolev inequality). Also introduced are a gauge covariant ODE argument for establishing the sharp decay rate, and gauge invariant Gagliardo--Nirenberg inequalities.
\item Section~\ref{sec:comm} is the algebraic heart of the paper; we use the formalism of exterior differential calculus for vector-valued forms to derive the commutation properties of the Chern--Simons systems with respect to the Killing vector fields $Z_{\mu \nu}$.
\item Finally, in Section~\ref{sec:BA}, we use the tools developed in Sections~\ref{sec:covVF} and \ref{sec:comm} to establish Proposition~\ref{prop:main}, thereby completing the proof of the main theorems.
\item In Appendix~\ref{app:gauge}, we record the reduced systems in the temporal and Cronstr\"om gauges. These computations are used in Section~\ref{sec:reduction}.
\end{itemize}

\section{Geometric setup and notation} \label{sec:setup}
In this section we provide the basic geometric setup used in this paper. We also take this opportunity to fix the notation and conventions.

\subsection{Tensor notation}
In this paper, we mostly use the invariant notation for tensors. All tensor products, unless otherwise specified, are taken over $\bbR$. The metric dual 1-form of a vector field $X$ will be denoted $X^{\flat}$, and the metric dual $k$-contravariant tensor of a $k$-covariant tensor $T$ will be denoted by $T_{\sharp}$.
The Levi-Civita connection associated to the Minkowski metric $\met$ will be denoted by $\nb$. This connection is trivial (i.e., has vanishing Christoffel symbols) in the rectilinear coordinates $(t = x^{0}, x^{1}, x^{2})$. 

Greek indices (e.g., $\mu, \nu$) run over $0,1,2$, and are used either to indicate tensor components in the rectilinear coordinates $(t=x^{0}, x^{1}, x^{2})$ or to parametrize Killing vector fields on $\bbR^{1+2}$; see Subsection \ref{subsec:KillingVF} below. We employ the Einstein summation convention of summing up repeated upper and lower indices. Furthermore, indices are raised or lowered using the metric $\met$, e.g., $T^{\mu} = (\met^{-1})^{\mu \nu} T_{\nu}$. 

Sometimes it will be convenient to employ the \emph{abstract index notation}, which we now briefly explain. The abstract indices $a, b, c, \ldots$ are not numbers (like $\mu, \nu = 0, 1, 2$), but rather placeholders which indicates the type of a tensor. For example, a vector field is written as $X^{a}$ and a $k$-covariant tensor is denoted by $T_{a_{1} \cdots a_{k}}$. Contraction is indicated by repeated upper and lower abstract indices as in the Einstein summation convention, e.g., $T(X, Y) = T_{ab} X^{a} Y^{b}$ for a 2-covariant tensor $T$ and vector fields $X$ and $Y$. This elegant representation of the contraction operation is a key advantage of the abstract index notation. Finally, abstract indices are raised and lowered using the metric $\eta$. When applied to all indices of a vector field $X^{a}$ or a $k$-covariant tensor $T_{a_{1} \cdots a_{k}}$, this is equivalent to taking their respective metric dual, i.e., $X_{a} = X^{\flat}_{a}$ and $T^{a_{1} \cdots a_{k}} = T_{\sharp}^{a_{1} \cdots a_{k}}$.

\subsection{Vector bundles and gauge structure of \eqref{eq:CSH} and \eqref{eq:CSD}} \label{subsec:gauge-str}

The proper way to describe gauge theories is to use the language of vector bundles. For a general introduction to the theory of vector bundles, we refer to \cite{Kobayashi:1963uh, Kobayashi:1969ub}. 

In this paper, we only need to consider the trivial $V$-bundle $E := \bbR^{1+2} \times V$ for a complex vector space $V$, equipped with a metric $\brk{\cdot, \cdot}_{V}$, as well as the restricted bundles on subsets $\calO \subseteq \bbR^{1+2}$. For simplicity, we will often omit the subscript $V$ and write $\brk{\cdot, \cdot} = \brk{\cdot, \cdot}_{V}$. 

The sections of $E$ may be identified with the $V$-valued functions on $\bbR^{1+2}$ by the following procedure. Take a global orthonormal frame field $\set{\Tht_{\mathfrak{a}}}$, i.e., $\dim V$-many sections which form an orthonormal basis with respect to $\brk{\cdot, \cdot}$ at every point; it exists thanks to the triviality of $E$. Then identifying the frame $\set{\Tht_{\mathfrak{a}}(p)}$ at each point $p \in \bbR^{1+2}$ with a fixed basis $\set{\tht_{\mathfrak{a}}}$ of $V$, we obtain the desired identification. Note that this procedure works equally well for any vector bundle equipped with a real or complex inner product (e.g., the adjoint $\LieAlg$-bundle) on any contractible subset of $\bbR^{1+2}$. In this paper, this identification is freely used.

A $\LieGrp$-valued function $U$ acts naturally (on the left) on a $V$-valued function $\phi$ by the pointwise action, i.e,. $(U \cdot \phi) (p) = U(p) \cdot \phi(p)$. Geometrically, this corresponds to a change of frame at each point $p$ by an appropriate action (on the right) of $U(p)$. We call $U$ a \emph{gauge transformation}, and $U \cdot \phi$ the \emph{gauge transform} of $\phi$ by $U$. 

Given a section $\phi$ of $E$, realized as a $V$-valued function, a gauge covariant derivative $\covD$ of $\phi$ can be written in reference to $\nb$ as in \eqref{eq:covd-def}; it is characterized by a $\LieAlg$-valued 1-form $A$, called the corresponding \emph{connection 1-form}. The commutator of two gauge covariant derivatives leads to the curvature 2-form $F$ by \eqref{eq:curv-def}. 
Under a gauge transformation $U$, the connection 1-form $A$ and the curvature 2-form $F$ transform under the rules
\begin{equation} \label{eq:gt}
	A \mapsto U A U^{-1} - (\ud U) U^{-1}, \quad
	F \mapsto U F U^{-1}.
\end{equation}
As a consequence, note that $F$ takes values in the adjoint $\LieAlg$-bundle, whereas $A$ does not.

Let $A$ be any connection 1-form. As the representation $\rho$ is unitary, $\LieAlg$ acts on $V$ by anti-hermitian operators; hence we have the following Leibniz rule for $V$-valued functions:
\begin{equation} \label{eq:leibniz-V}
	\nb \brk{\phi^{1}, \phi^{2}} = \brk{\covD \phi^{1}, \phi^{2}} + \brk{\phi^{1}, \covD \phi^{2}}.
\end{equation}
Similarly, by the bi-invariance of $\brk{\cdot, \cdot}_{\LieAlg}$, we have
\begin{equation} \label{eq:leibniz-g}
	\nb \brk{a^{1}, a^{2}}_{\LieAlg} = \brk{\covD a^{1}, a^{2}}_{\LieAlg} + \brk{a^{1}, \covD a^{2}}_{\LieAlg}
\end{equation}
for $\LieAlg$-valued functions $a^{1}, a^{2}$. 

Finally, if we define $\covD a$ (where $a$ is $\LieAlg$-valued) by the adjoint action (i.e., Lie bracket), then 
\begin{equation} \label{eq:leibniz-gV-0}
	\covD (a \cdot \phi) = (\covD a) \cdot \phi + a \cdot \covD \phi.
\end{equation}
where $a$ and $\phi$ are $\LieAlg$- and $V$-valued functions, respectively.

\subsection{Exterior differential calculus} \label{subsec:extr-calc}
We now introduce basic operations of the exterior differential calculus, which will be our main tool for computing commutation relations. 
A standard reference is \cite[Chapter 1]{Kobayashi:1963uh}.
Our notation is as follows: $\wedge$ denotes the wedge product, $\ud$ is the exterior derivative and $\iota_{X}$ is the interior product\footnote{Our convention is that the contraction takes place in the left-most slot, i.e., $(\iota_{X} \omg)(Y_{1}, \ldots, Y_{k-1})= \omg(X, Y_{1}, \ldots, Y_{k-1})$. } with a vector field $X$. The Lie derivative with respect to $X$ will be denoted by $\LD_{X}$. This operation makes sense for any tensor field; in particular, one has $\LD_{X} f = X f$ for a function $f$ and  $\LD_{X} Y = [X, Y]$ for a vector field $Y$. 

We also need to develop the exterior differential calculus of vector- and Lie algebra-valued forms.
Let $V$ be a complex vector space, equipped with an inner product $\brk{\cdot, \cdot}_{V}$. Consider also the Lie algebra $\LieAlg$ associated with $\LieGrp$, whose action on $V$ is denoted by $a \cdot v$ $(a \in \LieAlg, v \in V)$. When $V = \LieAlg$, we let $\LieAlg$ act by the adjoint action, i.e., $a \cdot v = \LieBr{a}{v}$.

A \emph{$V$-valued $k$-form at a point $p \in \bbR^{1+2}$} is a totally anti-symmetric multilinear form that takes in $k$ tangent vectors at $p$ and gives an element of $V$. 
A (smooth) \emph{$V$-valued $k$-form} on an open subset $\calU$ of $\bbR^{1+2}$ is a (smooth) association of points $p$ with a $V$-valued $k$-form at $p$.
A $\LieAlg$-valued $k$-form is defined similarly.
In order to distinguish from these objects, the usual $k$-forms on $\bbR^{1+2}$ will be referred to as being \emph{real-valued}. Any $V$-valued $k$-form can be decomposed to a linear combination of tensor products of the form $\phi \otimes \omg$, where $\phi$ is a $V$-valued function and $\omg$ is a real-valued $k$-form.

The operations $\ud$ and $\iota_{X}$ are naturally extended (component-wisely) to $V$- and $\LieAlg$-valued $k$-forms, as well as the wedge product $v \wedge \omg$ of a $V$-valued $k$-form $v$ and a real-valued $\ell$-form $\omg$. 
On the other hand, we define the wedge product $a \wedge v$ of a $\LieAlg$-valued $k$-form $a$ and a $V$-valued $\ell$-form $v$ using the action (on the left) of $\LieAlg$ on $V$. This product is characterized by the relation
\begin{equation*}
	(b \otimes \omg^{1}) \wedge (\phi \otimes \omg^{2}) = (b \cdot \phi) \omg^{1} \wedge \omg^{2}
\end{equation*}
for a $\LieAlg$-valued function $b$, a $V$-valued function $\phi$ and real-valued differential forms $\omg^{1}, \omg^{2}$.
In particular, when $V = \LieAlg$, the wedge product of two $\LieAlg$-valued forms $a, b$ is defined using the adjoint action, or the Lie bracket; for this reason, we use the notation $[a \wedge b]$ for this product.

Throughout the paper, the following convention is in effect:
\begin{convention}
Unless otherwise specified by parentheses, wedge products are understood to be taken from the right to the left.
\end{convention}
Note that, due to the lack of associativity of the Lie bracket, the wedge product of $\LieAlg$-valued forms generally \emph{fails} to be associative, i.e., we have $[[a \wedge b] \wedge c] \neq [a \wedge [b \wedge c]]$ for $\LieAlg$-valued forms $a, b, c$ in general. Similarly, for a $V$-valued form $v$, in general we have $[a \wedge b] \wedge v \neq a \wedge (b \wedge v)$. 

Given a connection 1-form $A$, we define the \emph{gauge covariant exterior derivative} $\covud$ of a $V$-valued $k$-form $v$ to be
\begin{equation} \label{eq:covud}
	\covud v = \ud v + A \wedge v.
\end{equation}
Furthermore, the \emph{gauge covariant Lie derivative} is defined as
\begin{equation} \label{eq:covLD}
	\covLD_{X} v = \calL_{X} v + (\iota_{X} A) v.
\end{equation}
Observe that for $V$-valued functions, both definitions coincide with the gauge covariant derivative, i.e., $\covud_{A} \phi (X) = \covLD_{X} \phi = \covD_{X}  \phi$. 

The \emph{Hodge star operator} associated to $\met$ is denoted by $\star$. This operator linearly maps a real-valued $k$-form ($k=0,1,2,3$) to a real-valued $(3-k)$-form, and is characterized by the relation
\begin{equation} \label{eq:star-def}
	\omg^{1} \wedge \star \omg^{2}
	= \eta^{-1}(\omg^{1}, \omg^{2}) \eps,
\end{equation}
where $\eps = \ud x^{0} \wedge \ud x^{1} \wedge \ud x^{2}$ is the volume form on $\bbR^{1+2}$ and $\eta^{-1}(\cdot, \cdot)$ is the induced Minkowski metric\footnote{The induced Minkowski metric for real-valued $k$-forms is defined so that given any orthonormal set of 1-forms $\set{e^{0}, e^{1}, e^{2}}$, $\set{e^{i_{1}} \wedge \cdots \wedge e^{i_{k}} : i_{1}, \ldots, i_{k} = 0, 1,2}$ is orthonormal.} on real-valued $k$-forms. This definition naturally extends to $V$- and $\LieAlg$-valued differential forms componentwisely. Equivalently, the Hodge star operator $\star$ on a $V$- [resp. $\LieAlg$-]valued $k$-form is characterized by the relation
\begin{equation*}
	\star (\phi \otimes \omg) = \phi \otimes \star \omg
\end{equation*}
for $\phi \in V$ [resp. $\phi \in \LieAlg$] and a $k$-form $\omg$.


In order to measure the size of real-valued forms, we use the auxiliary Euclidean metric $(\ud x^{0})^{2} + (\ud x^{1})^{2} + (\ud x^{2})^{2}$, which has the benefit of being parallel. Hence for a real-valued $k$-form $\omg$, we define
\begin{equation*}
	\abs{\omg}^{2} = \sum_{\mu_{1} < \cdots < \mu_{k}} \abs{\omg(T_{\mu_{1}}, \ldots, T_{\mu_{k}})}^{2} ,
\end{equation*}
where $T_{\mu} = \rd_{\mu}$ in the rectilinear coordinates. The norm of a $V$- or $\LieAlg$-valued $k$-form is defined similarly, using in addition $\brk{\cdot, \cdot}_{V}$ or $\brk{\cdot, \cdot}_{\LieAlg}$, respectively.

Given a $V$-valued $k$-form $v$ on $\bbR^{1+2}$ and an embedded submanifold $\Sgm \subset \bbR^{1+2}$ of $\bbR^{1+2}$, we denote by $v \restriction_{\Sgm}$ the \emph{pullback} of $v$ along the inclusion map $\iota: \Sgm \hookrightarrow \bbR^{1+2}$, which is a $V$-valued $k$-form on $\Sgm$ characterized by
\begin{equation*}
	v \restriction_{\Sgm}(X_{1}, \ldots, X_{k}) = v(\ud \iota (X_{1}), \ldots, \ud \iota (X_{k})) 
\end{equation*}
for every $p \in \Sgm$ and $X_{1}, \ldots, X_{k} \in T_{p} \Sgm$, where $\ud \iota : T_{p} \Sgm \to T_{\iota(p)} \bbR^{1+2}$ is the differential of the map $\iota$ at $p$. In particular, if $v$ is a $V$-valued $0$-form (i.e., a $V$-valued function) then $v \restriction_{\Sgm}$ is simply the restriction of $v$ to $\Sgm$.

For further formulae and results in exterior differential calculus, we refer to Section~\ref{subsec:extr-calc-2}.

\subsection{Killing vector fields on $\bbR^{1+2}$} \label{subsec:KillingVF}
A vector field on a Lorentzian (more generally, pseudo-Riemannian) manifold is said to be \emph{Killing} if it generates a one-parameter group of isometries. As is well-known, there are $6$ linearly independent Killing vector fields on $\bbR^{1+2}$, given in the rectilinear coordinates $(t = x^{0}, x^{1}, x^{2})$ by
\begin{itemize}
\item {\it Translations:} $T_{\mu} = \displaystyle{\rd_{\mu}}$ 
\item {\it Lorentz transforms and rotations:} $Z_{\mu \nu} = x_{\mu} \rd_{\nu} - x_{\nu} \rd_{\mu}$
\end{itemize}
where $\mu, \nu = 0, 1, 2$ and $x_{\mu} = \eta_{\mu \lmb} x^{\lmb}$. These vector fields commute with the Klein-Gordon operator $\Box \m 1$. We also define the scaling vector field
\begin{equation*}
	S = x^{\mu} \rd_{\mu} ,
\end{equation*}
which is \emph{not} a Killing vector field. It is a conformal Killing vector field, but it does not satisfy a good commutation relation with respect to $\Box \m 1$.

The span of the vector fields $T_{\mu}$, $Z_{\mu \nu}$ and $S$ form a Lie algebra under the natural commutation operation. Schematically, their commutation relations are given as follows:
\begin{align*}
	[T, T ] = & 0, 
&	[Z, T] = & T, 
&	[Z, Z] = & Z, \\
	[T, S] = & T, 
&	[Z, S] = & 0,
&	[S, S] = & 0.
\end{align*}

Below, we will often consider covariant derivatives of $V$-valued functions with respect to the vector fields introduced above. It will be convenient to introduce the following notation:
\begin{equation*}
	\covT_{\mu} := \covD_{T_{\mu}}, \quad 
	\covZ_{\mu \nu} := \covD_{Z_{\mu \nu}} = x_{\mu} \covT_{\nu} - x_{\nu} \covT_{\mu}, \quad
	\covS = \covD_{S} = x^{\mu} \covT_{\mu}.
\end{equation*}
Furthermore, in view of \eqref{eq:normal} below, we also introduce the notation
\begin{equation*}
	\covN = \covD_{N} = \hT^{-1} \covS.
\end{equation*}
Note that these covariant differential operators coincide with the covariant Lie derivatives along the same vector fields, i.e., $\covZ_{\mu \nu} \phi = \covLD_{Z_{\mu \nu}} \phi$ etc.

\subsection{Spherical and hyperboloidal polar coordinates} \label{subsec:polar-coords}
On each constant $t$-hypersurface, we define the spherical polar coordinates $(r, \tht) \in (0, \infty) \times [0, 2 \pi)$ by 
\begin{equation*}
	x^{1} = r \cos \tht, \quad x^{2} = r \sin \tht.
\end{equation*}
In what follows, we will refer to this coordinate system simply as the \emph{polar coordinates}. In this coordinate system the metric takes the form
\begin{equation*}
	\met = \m \ud t^{2} \p \ud r^{2} \p r^{2} \ud \tht^{2}.
\end{equation*}

Define the function $\omg_{j}$ in the rectilinear coordinates $(t, x^{1}, x^{2})$ by 
\begin{equation*}
	\omg_{j} = \omg^{j} := \frac{x^{j}}{\sqrt{(x^{1})^{2} + (x^{2})^{2}}} \qquad (j = 1, 2) \, .
\end{equation*}
In the polar coordinates, we have $\omg_{1} = \cos \tht$ and $\omg_{2} = \sin \tht$.

The \emph{hyperboloidal polar coordinate system} is the Minkowski analogue of the spherical polar coordinate system on Euclidean spaces. 
The coordinates consist of $(\hT, \hY, \tht) \in (0, \infty) \times (0, \infty) \times [0, 2 \pi)$, where
\begin{equation*}
	t = \hT \cosh \hY, \quad
	r = \hT \sinh \hY.
\end{equation*}
These coordinates cover (of course, minus the axis of rotation $\set{(x^{0}, 0, 0)}$, like the standard polar coordinates) the solid future light cone 
\begin{equation*}
\calC_{0} = \set{(x^{0}, x^{1}, x^{2}): -(x^{0})^{2} + (x^{1})^{2} + (x^{2})^{2} < 0, x^{0} > 0}.
\end{equation*}
In this coordinate system, the metric and its inverse take the form
\begin{align*}
	\met =& \m \ud \hT^{2} \p \hT^{2} \ud \hY^{2} \p \hT^{2} \sinh^{2} \hY \ud \tht^{2}, \\
	\met^{-1}	
		=& \m \rd_{\hT} \otimes \rd_{\hT}
			 \p \hT^{-2} \rd_{\hY} \otimes \rd_{\hY} 
			 \p \hT^{-2} (\sinh \hY)^{-2} \rd_{\tht} \otimes \rd_{\tht} \, .
\end{align*}
We denote the constant $\hT$-hypersurface by $\calH_{\hT}$. 
Observe that the future pointing unit normal $N = n_{\calH_{\hT}}$ to $\calH_{\hT}$ is equal to $\rd_{\hT}$, which coincides with the vector field $\hT^{-1} S$, i.e.,
\begin{equation} \label{eq:normal}
	N = n_{\calH_{\hT}} = \rd_{\hT} = \hT^{-1} S.
\end{equation}
The induced volume form on $\calH_{\hT}$ is given by
\begin{equation*}
	\ud \sgm_{\calH_{\hT}} = \hT^{2} \cosh \hY \, \ud \hY \ud \tht.
\end{equation*}
The hyperboloidal polar coordinates are useful since they are Lorentz-invariant, i.e., the vector fields $Z_{\mu \nu}$ are tangent to $\calH_{\hT}$. Indeed, partial derivatives in the hyperboloidal polar coordinate system are related to $Z_{\mu \nu}$ by
\begin{align}
	\rd_{\tht} = & Z_{12}, \label{eq:dTht} \\
	\rd_{\hY}
	= & - (\omg_{1} Z_{01} + \omg_{2} Z_{02}). \label{eq:dY}
\end{align}
We also note that
\begin{align} 
	\frac{\cosh \hY}{\sinh \hY} \rd_{\tht}
	=&	-( \omg_{1} Z_{02} - \omg_{2} Z_{01}). \label{eq:wdTht} 
\end{align}
which is favorable in the region $\set{r \leq t}$. Inverting the linear system consisting of \eqref{eq:dY} and \eqref{eq:wdTht}, $Z_{01}$ and $Z_{02}$ may be written in terms of $\rd_{\hY}$ and $\rd_{\tht}$ as follows:
\begin{align}
	Z_{01}
=&	\omg_{1} \rd_{y} + \omg_{2} \frac{\cosh \hY}{\sinh \hY} \rd_{\tht} \label{eq:Z01} \\
	Z_{02}
=&	\omg_{2} \rd_{y} - \omg_{1} \frac{\cosh \hY}{\sinh \hY} \rd_{\tht}.	 \label{eq:Z02}
\end{align}





\subsection{Notation for spacetime regions} 
Given $R \in \bbR$, we denote by $\calC_{R}$ the solid future light cone with its tip at $(R, 0, 0)$, i.e.,
\begin{equation*}
	\calC_{R} = \set{(x^{0}, x^{1}, x^{2}) \in \bbR^{1+2} : - (x^{0}-R)^{2} + (x^{1})^{2} + (x^{2})^{2} < 0, \, x^{0} > R}.
\end{equation*}
As discussed above, the cone $\calC_{0}$ admits a foliation by the hyperboloids 
\begin{equation*}
	\calH_{\hT} = \set{\hT = const} = \set{(x^{0}, x^{1}, x^{2}) \in \bbR^{1+2} : - (x^{0})^{2} + (x^{1})^{2} + (x^{2})^{2} = - \hT^{2}, x^{0} > 0}.
\end{equation*}
for $\tau > 0$, i.e., $\calC_{0} = \cup_{\hT > 0} \calH_{\hT}$. 

For $R > 0$, we denote by $B_{R}(x_{0})$ the open ball in $\bbR^{2}$ of radius $R$ and centered at $x_{0}$, i.e.,
\begin{equation*}
	B_{R}(x_{0}) = \set{(x^{1}, x^{2}) \in \bbR^{2} : (x^{1} - x_{0}^{1})^{2} + (x^{2} - x_{0}^{2})^{2} < R^{2}}.
\end{equation*}
In the case $x_{0} = 0$, we will often omit $x_{0}$ and simply write $B_{R} = B_{R}(0)$.


\subsection{Norms and other conventions}
We use the standard notation $L^{p}$, $W^{k, p}$ and $H^{k}$ for the Lebesgue, $L^{p}$- and $L^{2}$-based Sobolev spaces of order $k$, respectively.
Furthermore, we also introduce a weighted norm on $\calH_{\hT}$:
\begin{equation} \label{eq:nrm-def} 
	\wnrm{\phi}_{L^{p}_{\hT}} := \nrm{\phi}_{L^{p}(\calH_{\hT}, \frac{\ud \Vol}{\cosh \hY})}.
\end{equation}
This norm arises naturally from the energy inequality; see Section~\ref{subsec:energy}.

In this paper, complicated formulae would often be simplified to its \emph{schematic} form; see for instance Propositions~\ref{prop:comm-covKG}, \ref{prop:comm-J-CSH} and \ref{prop:comm-J-CSD} below. By a \emph{schematic formula} of the form
\begin{equation*}
	(\hbox{LHS}) = \sum_{k} B_{k},
\end{equation*}
we mean precisely that the left-hand side is equal to the linear combination of terms of the right-hand side, i.e., there exist constants $c_{k}$ such that $(\hbox{LHS}) = \sum_{k} c_{k} B_{k}$.

\section{Reduction to the main a priori estimate} \label{sec:reduction}
The goal of this section is to reduce the proof of the main theorems (Theorems~\ref{thm:CSH} and \ref{thm:CSD}) to establishing a priori estimates for a unified system \eqref{eq:CS-uni} encompassing both \eqref{eq:CSH} and \eqref{eq:CSD}; see Proposition~\ref{prop:main} below. 

In order to simplify the exposition, the following convention will be in effect for the remainder of the paper: 
\begin{convention}
The non-zero parameters $\kpp, v$ and $m$ in \eqref{eq:CSH} and \eqref{eq:CSD} are normalized to 1, i.e., $\kpp = v = m = 1$.
\end{convention}
Our analysis can be adapted in an obvious fashion to the general parameters, as long as they are non-zero.

\subsection{Unified system for \eqref{eq:CSH} and \eqref{eq:CSD}} \label{subsec:uni}
In this subsection, we first show how \eqref{eq:CSD} can be reduced to a covariant Klein--Gordon equation by squaring the Dirac equation. Building on this reduction, we then introduce a single system that allows for a unified treatment of \eqref{eq:CSH} and \eqref{eq:CSD}; see \eqref{eq:CS-uni} below.

Consider the covariant Dirac equation in \eqref{eq:CSD}, i.e.,
\begin{equation*}
	(i \covsD + m) \psi = 0.
\end{equation*}
Applying the covariant Dirac operator $i \covsD$ and using \eqref{eq:gmm-mat}, we obtain the following covariant Klein-Gordon equation with mass $m^{2}$ for $\psi$:
\begin{equation}	\label{eq:CSD:KG}
\covBox \psi - m^{2} \psi 
		= \frac{1}{2} \gmm^{\mu} \gmm^{\nu} F(T_{\mu}, T_{\nu}) \cdot \psi.
\end{equation}
In what follows, we will work exclusively with the equation \eqref{eq:CSD:KG}. In particular, we may forget the spinorial structure of $\psi$, and consider $\psi$ simply as a $V$-valued function on $\bbR^{1+2}$. Similarly, we view
\begin{equation*}
	\gmm = \eta_{\mu \nu} \gmm^{\mu} \ud x^{\nu}, \quad \alp = \gmm^{0} \gmm = \eta_{\mu \nu} \alp^{\mu} \ud x^{\nu}
\end{equation*}
as $2 \times 2$ matrix-valued 1-forms on $\bbR^{1+2}$. These observations allow us to treat \eqref{eq:CSD} on the same footing as \eqref{eq:CSH}, despite its original spinorial nature.
 
We now introduce a unified system of equations that subsumes both \eqref{eq:CSH} and \eqref{eq:CSD}. 
Let $V$ be a complex vector space with inner product $\brk{\cdot, \cdot}$, with an additional structure $V = \Dlt \otimes_{\bbC} W$ in the case of \eqref{eq:CSD}. Let $\phi$ be a $V$-valued function on $\bbR^{1+2}$, which represents
\begin{equation*}
	\phi =
\left\{
\begin{array}{cl}
\varphi & \hbox{ for } \eqref{eq:CSH} \\
\psi & \hbox{ for } \eqref{eq:CSD}.
\end{array}
\right.
\end{equation*}
Let $\LieGrp$ be a Lie group with a positive-definite bi-invariant metric $\brk{\cdot, \cdot}_{\LieAlg}$, which acts on $V$ as described in Sections~\ref{subsec:CSH} and \ref{subsec:CSD}, and let $A$ be a connection 1-form on $\bbR^{1+2}$. The \emph{unified Chern--Simons system} is given by
\begin{equation} \label{eq:CS-uni}
\left\{
\begin{aligned}
	(\covBox \m 1) \phi = & U(\phi) \\
	F =& \star J(\phi).
\end{aligned}
\right.
\end{equation}
For \eqref{eq:CSH} and \eqref{eq:CSD}, $J(\phi)$ equals (respectively)
\begin{align*}
J_{\CSH}(\varphi)
=&  \brk{\calT \varphi, \covud \varphi} + \brk{\covud \varphi, \calT \varphi}, \\
J_{\CSD}(\psi)
=& \brk{\calT \psi, i \alp \psi} .
\end{align*}
In the case of \eqref{eq:CSH}, the $V$-valued potential $U(\phi) = U_{\CSH} (\varphi)$ takes the form
\begin{equation*}
	U_{\CSH}(\varphi) = 4 \LieBr{\varphi}{\LieBr{\varphi}{\varphi^{\dagger}}}
				+  2 \LieBr{\LieBr{\varphi}{\LieBr{\varphi^{\dagger}}{\LieBr{\varphi}{\varphi^{\dagger}}}}}{\varphi}
				+ \LieBr{\LieBr{\varphi}{\LieBr{\varphi}{\varphi^{\dagger}}}}{\LieBr{\varphi}{\varphi^{\dagger}}} .
\end{equation*}
In the case of \eqref{eq:CSD}, we have $U(\phi) = U_{\CSD}(\psi)$ with
\begin{equation} \label{eq:U-CSD}
U_{\CSD} (\psi) = \frac{1}{2} \eps_{\mu \nu \lmb} \gmm^{\mu} \gmm^{\nu} (J_{\CSD}^{\lmb}(\psi) \cdot \psi).
\end{equation}
Here $\eps_{\mu \nu \lmb} = \eps(T_{\mu}, T_{\nu}, T_{\lmb})$.
In what follows, we will refer to the first equation of \eqref{eq:CS-uni} and the \emph{covariant Klein--Gordon equation}, and the second equation as the \emph{Chern--Simons equation}.

Let $\Sgm \subset \bbR^{1+2}$ be a spacelike hypersurface with a future directed unit normal vector field $n_{\Sgm}$. For instance, $(\Sgm, n_{\Sgm}) = (\Sgm_{t_{0}}, T_{0})$ or $(\calH_{\hT_{0}}, N)$. The data on $\Sgm$ for a solution $(A, \phi)$ to \eqref{eq:CSH} consist of a triple $(a, f, g)$ of a $\LieAlg$-valued 1-form and $V$-valued functions $f, g$ on $\Sgm$, such that
\begin{equation} \label{eq:CS-uni-id}
	(a, f, g) =(A, \phi, \covD_{n_{\Sgm}} \phi) \restriction_{\Sgm}.
\end{equation}
As a consequence of the equation \eqref{eq:CS-uni}, such a triple $(a, f, g)$ obeys the constraint equation
\begin{equation} \label{eq:CS-uni-const}
	\ud a + \frac{1}{2} [a \wedge a] = \star J \restriction_{\Sgm}
\end{equation}
Accordingly, we say that $(a, f, g)$ is an \emph{initial data set} for \eqref{eq:CS-uni} if it solves \eqref{eq:CS-uni-const} with $(\phi, \covD_{n_{\Sgm}} \phi) \restriction_{\Sgm} = (f, g)$ on the right-hand side. Note that an initial data set $(a, f, g)$ for \eqref{eq:CSH} is also an initial data set for \eqref{eq:CS-uni}, whereas an initial data set $(a, \psi_{0})$ of \eqref{eq:CSD} gives rise to an initial data set $(a, f, g)$ for \eqref{eq:CS-uni}, where $f = \psi_{0}$ and $g$ is computed from the Dirac equation $i \covsD \psi + m \psi = 0$.

\subsection{Solving up to the initial hyperboloid}
Since both \eqref{eq:CSH} and \eqref{eq:CSD} are time reversible, it suffices to prove Theorems~\ref{thm:CSH} and \ref{thm:CSD} just in the future time direction.  As is usual in the vector field method for a Klein--Gordon equation \cite{MR803252}, the main part of our analysis takes place in the hyperboloidal foliation $\set{\calH_{\hT}}_{\hT > 0}$. To connect this analysis with the Cauchy problem for the foliation $\set{\Sgm_{t}}_{t \in \bbR}$, we first apply a time translation to place the initial data on $\Sgm_{2R}$, where we remind the reader that $R$ measures the radius of the support of the initial data. Then we use the following result, which passes from initial data posed on $\Sgm_{2R}$ to a solution up to $\calH_{2R}$.

\begin{figure}[h] 
\begin{tikzpicture}
   \fill[fill=lightgray] (0, 2) -- (0, 5/2) -- (3/2, 5/2) -- (4,5) -- (5,5) -- (5,2) -- cycle;
  \draw[->] (0,0) -- (5,0) node[right] {$r$};
  \draw[->] (0,0) -- (0,5) node[above] {$t$};
  \draw[<-] (3/2-0.1, 5/2+0.1) to [out=90, in=-45] (1,3) node[above] {$(\frac{3}{2} R, \frac{5}{2} R)$};
  \draw (0,5/2) -- (3/2,5/2) ;
  \draw [dashed] (3/2,5/2) -- (5,5/2) node[right] {$\Sgm_{\frac{5}{2} R}$};
  \draw (0,2) -- (5,2) node[right] {$\Sgm_{2R}$};
  \node[left] at (0, 1) {R};
  \node at (1, 4.5) {$\calC_{R}$};
  \node at (4, 3) {$\calQ$};
  \draw[domain=0:4,smooth,variable =\x] plot ({\x},{\x + 1}) node[above left] {$t = R + r$};
  \draw[domain=0:5,dashed,variable =\x] plot ({\x},{\x }) node[below right] {$t = r$};
   \draw[domain=0:sqrt(21), thick, smooth, red, variable =\x] plot ({\x},{sqrt(4 + \x*\x)}) node[above right] {$\calH_{2R}$};
   \draw[thick, red] (0,2) -- (1,2) node[below right] {$B_{R}$};
\end{tikzpicture}
\caption{The initial time slice and the initial hyperboloid}
\label{fig:initial-hyp}
\end{figure}
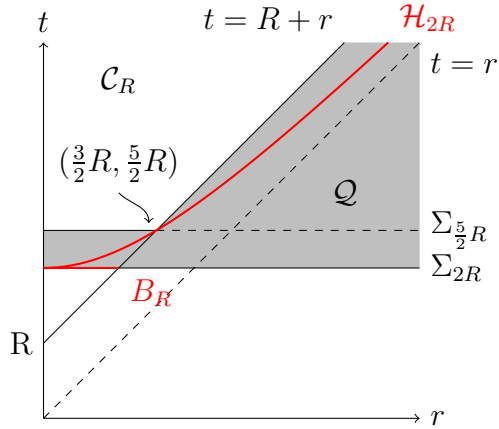
\begin{proposition}[Solution up to the initial hyperboloid] \label{prop:initial-hyp}
There exists $\dlt_{\ast\ast} = \dlt_{\ast \ast}(R) > 0$ such that the following statements hold.
Let $(a, f, g)$ be a smooth \eqref{eq:CSH} initial set obeying \eqref{eq:CSH-id} [resp. a smooth \eqref{eq:CSD} initial data set obeying \eqref{eq:CSD-id}], and consider the IVP for \eqref{eq:CSH} [resp. \eqref{eq:CSD}] with data on $\Sgm_{2R} = \set{t = 2R}$. 
If $\eps \leq \dlt_{\ast \ast}(R)$, then there exists a  smooth solution $(A, \phi)$ to the IVP on the spacetime region (see Figure~\ref{fig:initial-hyp})
\begin{equation*}
	\calQ := \bb( \set{2R \leq t \leq \tfrac{5}{2} R} \cup \set{\hT \leq 2R} \cup \calC_{R}^{c} \bb) \cap \set{t \geq 2R}.
\end{equation*}
which is unique up to smooth local gauge transformations. We have
\begin{equation} \label{eq:initial-hyp:fsp}
	\supp \, \phi \subseteq \calC_{R} \cap \set{t \geq 2R} = \set{(t, x) : t \geq 2R, \ \abs{x} \leq t + R}.
\end{equation}
Moreover, the solution $(A, \phi)$ obeys the following gauge invariant bounds:
\begin{equation} \label{eq:initial-hyp:est}
\sup_{t \in [2R, \frac{5}{2} R]} \sum_{k=0}^{5} \nrm{\covT^{(k)} \phi}_{L^{2}(\Sgm_{t})} \leq C \eps.
\end{equation}
\end{proposition}
The notion of uniqueness up to smooth local gauge transformations is defined as in the case of \eqref{eq:CSH}; see the discussion following Theorem~\ref{thm:CSH} above.
The significance of the time $t = \frac{5}{2} R$ in the definition of $\calQ$ and \eqref{eq:initial-hyp:est} is that $\calH_{2R}$ intersects the cone $\set{t = R + r}$ precisely on the circle $\set{t = \frac{5}{2} R, \ r = \frac{3}{2} R}$; see Figure~\ref{fig:initial-hyp}. 

A result like Proposition~\ref{prop:initial-hyp} is usually a quick consequence of the local well-posedness theory in the $\set{\Sgm_{t}}_{t \in \bbR}$ foliation and the finite speed of propagation; for instance, see \cite{MR2056833}. In order to properly formulate these properties for \eqref{eq:CS-uni}, we must address the issue of gauge choice, which arises at two stages: First, in finding a description of initial data obeying good bounds, and second, in the unique evolution of the initial data. 

To find a representation of the data with nice bounds, we rely on the following result. In what follows, given an open subset $O$ of $\Sgm_{t_{0}}$ and a $V$-(or $\LieAlg$-)valued $k$-form $v$, we define its Sobolev norms on $O$ by
\begin{equation*}
	\nrm{v}_{H^{m}(O)}^{2} := \sum_{k=0}^{m} \nrm{{}^{(\Sgm_{t_{0}})}\nb^{(k)} v}_{L^{2}(O)}^{2},
\end{equation*}
where ${}^{(\Sgm_{t_{0}})} \nb$ denotes the (induced) Levi-Civita connection on $\Sgm_{t_{0}}$. If $v$ is defined on a larger set $\calO \supset O$ (possibly an open set in the spacetime), then this norm is defined using the pullback along $O \hookrightarrow \calO$, i.e., $\nrm{v}_{H^{m}(O)} := \nrm{v \restriction_{O}}_{H^{m}(O)}$.

\begin{proposition}\label{prop:id-temporal}
Let $B := \set{t_{0}} \times B_{r_{0}}(x_{0}) \subseteq \Sgm_{t_{0}}$ and $(a, f, g)$ an initial data set for \eqref{eq:CS-uni} on $B$ satisfying the constraint equation \eqref{eq:CS-uni-const}. For a fixed integer $m \geq 1$, let
\begin{equation} \label{eq:id-size-temporal}
	\alp := \sum_{k=0}^{m} \nrm{\covSgmTD^{(k)} f}_{L^{2}(B)} + \sum_{k=0}^{m-1} \nrm{\covSgmTD^{(k)} g}_{L^{2}(B)},
\end{equation}
where $\covSgmTD$ is the (induced) covariant derivative on $\Sgm_{t_{0}}$.
If $\alp < \alp_{\ast}(r_{0})$, where $\alp_{\ast}(r_{0})$ is some fixed positive function depending only on $r_{0}$, then there exists a smooth gauge transformation $U$ on $B$ such that
the gauge-transformed potential $\tilde{a} = U a U^{-1} - \ud U U^{-1}$ satisfies
\begin{align*}
	\iota_{\bfn_{\rd B}} \tilde{a} =& 0 \quad \hbox{ on } \rd B = \set{x : \abs{x - x_{0}} = r_{0}}, \\
	{}^{(\Sgm_{t_{0}})} \dlt  \tilde{a} = & 0 \quad \hbox{ on } B.
\end{align*}
Here, $\bfn_{\rd B}$ is the outer normal vector field on $\rd B$ tangent to $\Sgm_{t_{0}}$ and ${}^{(\Sgm_{t_{0}})} \dlt$ is the exterior codifferential on $\Sgm_{t_{0}}$. Moreover, the gauge transformed initial data set $(\tilde{a}, \tilde{f} = U \cdot f, \tilde{g} = U \cdot g$) obeys the bounds
\begin{align*}
	\nrm{\tilde{a}}_{H^{m}(B)} \leq & C(m) \alp^{2}, \\
	\nrm{\tilde{f}}_{H^{m}(B)} + \nrm{\tilde{g}}_{H^{m-1}(B)} \leq & C(m) \alp.
\end{align*}
\end{proposition}

In the case $m = 1$, this proposition is an immediate consequence of the classical theorem of Uhlenbeck \cite[Theorem~1.3]{Uhlenbeck:1982vna} and the explicit formula for the curvature 2-form $F[a]$ in terms of $f$ and $g$ via the constraint equation \eqref{eq:CS-uni-const}. To handle the case $m > 1$, first note that $\tilde{a}$ is a solution to the boundary value problem for the div-curl system
\begin{equation*}
\left\{
\begin{aligned}
	\ud \tilde{a} =& F[\tilde{a}] - \frac{1}{2} [\tilde{a} \wedge \tilde{a}]\quad \hbox{ on } B \\
	{}^{(\Sgm_{t_{0}})} \dlt  \tilde{a} =& 0 \qquad \hbox{ on } B\\
	\iota_{\bfn_{\rd B}} \tilde{a} =& 0 \qquad \hbox{ on } \rd B,
\end{aligned}
\right.
\end{equation*}
where $F[\tilde{a}]$ is again given explicitly in terms of $\tilde{f}, \tilde{g}$ through the constraint equation. Hence standard elliptic arguments lead to higher regularity bounds for $\tilde{a}$ in terms of those for $F[\tilde{a}]$, which in turn follows from bounds for $\tilde{f}$, $\tilde{g}$. We omit the routine induction argument on $m$ that leads to the full proof.

To state a local well-posedness theorem for \eqref{eq:CS-uni}, we choose the \emph{temporal gauge}
\begin{equation*}
	\iota_{\rd_{t}} A = A_{0} = 0,
\end{equation*}
which has the nice feature of directly exhibiting the property of finite speed of propagation. 

To proceed further, we need to introduce some terminology. We say that a vector (or a direction) $X$ tangent to $\bbR^{1+2}$ is time-like [resp. null or space-like] if $\met(X, X) < 0$ [resp. $\met(X, X) = 0$ or $\met(X, X) > 0$]. A time-like or null vector $X$ is said to be future-directed [resp. past-directed] if $\met(X, \rd_{t}) < 0$ [resp. $\met(X, \rd_{t}) > 0$].  Finally, given a subset $O \subseteq \Sgm_{t_{0}} = \set{t_{0}} \times \bbR^{2}$, we define the \emph{future} [resp. past]  \emph{domain of dependence} $\DD(O)$ to be the set of all points $p \in \bbR^{1+2}$ such that all straight rays emanating from $p$ in the past- [resp. future-] directed time-like or null directions intersect with $O$. For instance, if $O = \set{t_{0}} \times B_{r_{0}}(x_{0})$ is the ball of radius $r_{0}$ centered at $x_{0}$ in $\Sgm_{t_{0}}$, then $\DD(O)$ is the cone 
\begin{equation*}
\DD(O) = \set{(t, x) \in \bbR^{1+2} : t_{0} \leq t < t_{0} + r_{0}, \ \abs{x - x_{0}} < r_{0} + t_{0} - t}
\end{equation*}
We are now ready to state local in spacetime well-posedness of \eqref{eq:CS-uni} in the temporal gauge, which includes the finite speed of propagation property.
\begin{theorem}[Local well-posedness in the temporal gauge] \label{thm:lwp-temporal}
Let $B := \set{t_{0}} \times B_{r_{0}}(x_{0}) \subseteq \Sgm_{t_{0}}$ and $(a, f, g)$ a smooth initial data set for \eqref{eq:CS-uni} on $B$. Fix $ m \geq 3$, and let 
\begin{equation*}
	\tilde{\alp}^{2} := \nrm{a}_{H^{m-1}(B)} + \nrm{{}^{(\Sgm_{t_{0}})} \dlt a}_{H^{m-1}(B)} + \nrm{f}_{H^{m}(B)}^{2} + \nrm{g}_{H^{m-1}(B)}^{2}.
\end{equation*}
If $\tilde{\alp} < \tilde{\alp}_{\ast}(r_{0})$, where $\tilde{\alp}_{\ast}(r_{0})$ is some fixed positive nondecreasing function of $r_{0}$, then there exists a unique smooth solution $(A, \phi)$ to the IVP for \eqref{eq:CS-uni} satisfying the temporal gauge condition $\iota_{\rd_{t}} A = 0$ in the set $\DD(B)$. Moreover, the solution obeys the bound
\begin{equation*}
\sup_{t \in [t_{0}, t_{0} + r_{0})} \bb( \nrm{A}_{H^{m-1}(B_{t})} + \nrm{\dlt A}_{H^{m-1}(B_{t})} + \nrm{\phi}_{H^{m}(B_{t})}^{2} + \nrm{\rd_{t} \phi}_{H^{m-1}(B_{t})}^{2} \bb) \leq C \tilde{\alp}^{2},
\end{equation*}
where $B_{t} := \Sgm_{t} \cap \DD(B)$.
\end{theorem}

In the temporal gauge $\iota_{\rd_{t}} A = 0$, the Chern--Simons system \eqref{eq:CS-uni} becomes a coupled system of a Klein--Gordon equation for $\phi$ and transport equations for $A$ and $\dlt A$ whose characteristics are precisely the constant $x$ curves. The precise form of the system can be found in Appendix~\ref{subsec:temporal}, using the formalism developed in Sections~\ref{subsec:extr-calc} and \ref{subsec:extr-calc-2}. The initial data for $(A, \phi)$ on $B$ are $(a, f, g)$ as in \eqref{eq:CS-uni-id}, whereas the initial data for $\dlt A$ on $B$ is ${}^{(\Sgm_{t_{0}})} \dlt a$, thanks to the temporal gauge condition. Theorem~\ref{thm:lwp-temporal} follows from a standard Picard iteration argument using the localized energy inequality for the wave equation in $\DD(B)$, integration along characteristics for the transport equation and the Sobolev inequality. We omit the details.

\begin{proof} [Sketch of proof of Proposition~\ref{prop:initial-hyp}]
Recall that $0 \leq \eps < \dlt_{\ast \ast}(R)$ by hypothesis. Choosing $\dlt_{\ast\ast}(R)$ sufficiently small, we may apply Proposition~\ref{prop:id-temporal} with $B = \set{2R} \times B_{3 R}(0)$ and $m = 5$ to find a gauge transform $(\tilde{a}, \tilde{f}, \tilde{g})$ of $(a, f, g)$ on $B$ obeying
\begin{equation*}
	\nrm{\tilde{a}}_{H^{5}(B)} + \nrm{\tilde{f}}_{H^{5}(B)}^{2} + \nrm{\tilde{g}}_{H^{4}(B)}^{2} \leq C \eps^{2}.
\end{equation*}
Taking $\dlt_{\ast \ast}(R)$ smaller if necessary, we may apply Theorem~\ref{thm:lwp-temporal} to $(\tilde{a}, \tilde{f}, \tilde{g})$ to construct a unique smooth solution $(\tilde{A}^{(in)}, \tilde{\phi}^{(in)})$ to \eqref{eq:CS-uni} in the temporal gauge on the set
\begin{equation*}
\calQ^{(in)}:= \DD(B) = \set{(t,x) : 2 R \leq t < 5 R, \ \abs{x} < 5 R - t},
\end{equation*}
which obeys the estimate
\begin{equation} \label{eq:initial-hyp:in-sol}
	\sup_{t \in [2R, 5R)} 
	\bb( \nrm{\tilde{A}^{(in)}}_{H^{4}(B_{t})} + \nrm{\dlt \tilde{A}^{(in)}}_{H^{4}(B_{t})} + \nrm{\tilde{\phi}^{(in)}}_{H^{5}(B_{t})}^{2} + \nrm{\rd_{t} \tilde{\phi}^{(in)}}_{H^{4}(B_{t})}^{2}
	\bb) \leq C \eps^{2},
\end{equation}
where $B_{t} = \Sgm_{t} \cap \DD(B) = \set{(t, x) : \abs{x} < 5R - t}$. 
Simply by undoing the gauge transformation from Proposition~\ref{prop:id-temporal}, we obtain from $(\tilde{A}^{(in)}, \tilde{\phi}^{(in)})$ a smooth solution to the IVP for the original data $(a, f, g)$ on $\calQ^{(in)}$, which we denote by $(A^{(in)}, \phi^{(in)})$.

To complete the proof, it remains to prove the existence of a smooth solution $(A^{(out)}, \phi^{(out)})$ to the IVP with the data $(a, f, g)$ on the set $\calQ^{(out)} := \calC_{R}^{c}$. Indeed, once we have $(A^{(out)}, \phi^{(out)})$ on $\calQ^{(out)}$, the desired solution $(A, \phi)$ may be constructed by simply patching $(A^{(in)}, \phi^{(in)})$ and $(A^{(out)}, \phi^{(out)})$ on $\calQ^{(in)} \cup \calQ^{(out)} \supset \calQ$. As we will see below, $\phi^{(out)} = 0$, which proves \eqref{eq:initial-hyp:fsp}. Furthermore, \eqref{eq:initial-hyp:est} follows from \eqref{eq:initial-hyp:fsp}, \eqref{eq:initial-hyp:in-sol} and \eqref{eq:CS-uni}. Uniqueness up to smooth local gauge transformations can be proved afterwards using Proposition~\ref{prop:id-temporal} and Theorem~\ref{thm:lwp-temporal}.

Fix a ball $B$ contained in $\Sgm_{2R} \setminus (\set{2R} \times B_{R})$. Since $(f, g) = (0, 0)$ on $B$, Proposition~\ref{prop:id-temporal} implies the existence of a smooth gauge transformation $U_{B}$, such that the gauge transform $(\tilde{a}, \tilde{f}, \tilde{g})$ of $(a, f, g)$ by $U_{B}$ is identically zero on $B$. Therefore, by Theorem~\ref{thm:lwp-temporal} the unique smooth solution to \eqref{eq:CS-uni} on $\DD(B)$ with data $(\tilde{a}, \tilde{f}, \tilde{g})$ in the temporal gauge is the zero solution. Undoing the gauge transformation $U_{B}$, we obtain $(A^{(out)}, \phi^{(out)})$ on $\DD(B)$. Since $B$ is arbitrary and sets of the form $\DD(B)$ cover $\calQ^{(out)}$, these local solutions patch up to the desired solution $(A^{(out)}, \phi^{(out)})$ on the whole region $\calQ^{(out)}$. The fact that $\phi^{(out)} = 0$ is clear from the construction. \qedhere
\end{proof}

\subsection{Local well-posedness in the hyperboloidal foliation}
In order to proceed, we need a local well-posedness theory of \eqref{eq:CS-uni} in the hyperboloidal foliation $\set{\calH_{\hT}}_{\hT >0}$.
For our purposes, it suffices to formulate analogues of Proposition~\ref{prop:id-temporal} and Theorem~\ref{thm:lwp-temporal} in this setting.

Let $\bfd_{\calH_{\hT_{0}}}(x, y)$ denote the geodesic distance between points $x$ and $y$ on $\calH_{\hT_{0}}$. We define the geodesic ball ${}^{(\calH_{\hT_{0}})} B_{r_{0}}(x_{0})$ with radius $r_{0}$ and center $x_{0}$ in $\calH_{\hT_{0}}$ by 
\begin{equation*}
	{}^{(\calH_{\hT_{0}})} B_{r_{0}}(x_{0})
	= \set{x \in \calH_{\hT_{0}} : \bfd_{\calH_{\hT_{0}}}(x, x_{0}) < r_{0}}.
\end{equation*}
Given an open subset $O$ of $\calH_{\hT_{0}}$ and a $V$-(or $\LieAlg$-)valued $k$-form $v$ on $O$, let
\begin{equation*}
	\nrm{v}_{H^{m}(O)}^{2} := \sum_{k=0}^{m} \nrm{{}^{(\calH_{\hT_{0}})} \nb^{(k)} v}_{L^{2}(O)}^{2},
\end{equation*}
where ${}^{(\calH_{\hT_{0}})} \nb$ denotes the (induced) Levi-Civita connection on $\calH_{\hT_{0}}$. If $v$ is defined on a larger set $\calO \supset O$ (possibly an open set in the spacetime), then $\nrm{v}_{H^{m}(O)} := \nrm{v \restriction_{O}}_{H^{m}(O)}$.

We are now ready to state the analogue of Proposition~\ref{prop:id-temporal} in $\calH_{\hT_{0}}$.
\begin{proposition}\label{prop:id-cronstrom}
Let $B := {}^{(\calH_{\hT_{0}})} B_{r_{0}}(x_{0}) \subseteq \calH_{\hT_{0}}$ and $(a, f, g)$ an initial data set for \eqref{eq:CS-uni} on $B$ satisfying the constraint equation \eqref{eq:CS-uni-const}. For a fixed integer $m \geq 1$, let
\begin{equation} \label{eq:id-size-cronstrom}
	\bt := \sum_{k=0}^{m} \nrm{\covHTD^{(k)} f}_{L^{2}(B)} + \sum_{k=0}^{m-1} \nrm{\covHTD_{x}^{(k)} g}_{L^{2}(B)},
\end{equation}
where $\covHTD$ is the (induced) covariant derivative on $\calH_{\hT_{0}}$.
If $\bt < \bt_{\ast}(\tau_{0}, x_{0}, r_{0})$, where $\bt_{\ast}(\tau_{0}, x_{0}, r_{0})$ is some fixed positive function depending only on $\tau_{0}, x_{0}$ and $r_{0}$, then there exists a smooth gauge transformation $U$ on $B$ such that the gauge-transformed potential $\tilde{a} = U a U^{-1} - \ud U U^{-1}$ satisfies
\begin{align*}
	\iota_{\bfn_{\rd B}} \tilde{a} =& 0 \quad \hbox{ on } \rd B, \\
	{}^{(\calH_{\hT_{0}})} \dlt  \tilde{a} = & 0 \quad \hbox{ on } B.
\end{align*}
Here, $\rd B = \set{x \in \calH_{\hT_{0}} : \bfd_{\calH_{\hT_{0}}}(x, x_{0}) = r_{0}}$, $\bfn_{\rd B}$ is the outer normal vector field on $\rd B$ tangent to $\calH_{\hT_{0}}$ and ${}^{(\calH_{\hT_{0}})} \dlt$ is the exterior codifferential on $\calH_{\hT_{0}}$. Moreover, the gauge transformed initial data set $(\tilde{a}, \tilde{f} = U \cdot f, \tilde{g} = U \cdot g$) obeys the bounds
\begin{align*}
	\nrm{\tilde{a}}_{H^{m}(B)} \leq & C(m) \bt^{2}, \\
	\nrm{\tilde{f}}_{H^{m}(B)} + \nrm{\tilde{g}}_{H^{m-1}(B)} \leq &C(m) \bt.
\end{align*}
\end{proposition}
The proof of Proposition~\ref{prop:id-cronstrom} proceeds exactly as that of Proposition~\ref{prop:id-temporal}; we skip the details.

We now turn to the task of formulating a local (in spacetime) well-posedness theorem in the $\set{\calH_{\hT}}_{\hT > 0}$ foliation. In order to fix the gauge ambiguity, we use the \emph{Cronstr\"om gauge condition}, which reads $x^{\mu} A_{\mu} = 0$ in the rectilinear coordinates. In the hyperboloidal polar coordinates, the gauge condition takes the form
\begin{equation} \label{eq:cronstrom}
\iota_{\rd_{\hT}} A = A_{\hT} = 0.
\end{equation}
This gauge is an analogue of the temporal gauge in the hyperboloidal foliation $\set{\calH_{\hT}}_{\hT > 0}$.

In the Cronstr\"om gauge, the analogue of Theorem~\ref{thm:lwp-temporal} reads as follows.
\begin{theorem}[Local well-posedness in the Cronstr\"om gauge] \label{thm:lwp}
Let $B:={}^{(\calH_{\hT_{0}})} B_{r_{0}}(x_{0}) \subseteq \calH_{\hT_{0}}$ and $(a, f, g)$ a smooth initial data set for \eqref{eq:CS-uni} on $B$. Fix $ m \geq 3$, and let 
\begin{equation*}
	\tilde{\bt}^{2} := \nrm{a}_{H^{m}(B)} + \nrm{{}^{(\calH_{\hT_{0}})} \dlt a}_{H^{m}(B)} + \nrm{f}_{H^{m}(B)}^{2} + \nrm{g}_{H^{m-1}(B)}^{2}.
\end{equation*}
If $\tilde{\bt} < \tilde{\bt}_{\ast}(\hT_{0}, x_{0}, r_{0})$, where $\tilde{\bt}_{\ast}(\hT_{0}, x_{0}, r_{0})$ is some positive nondecreasing function of $r_{0}$ for each fixed $\hT_{0}$ and $x_{0}$, then there exists a unique smooth solution $(A, \phi)$ to the IVP for \eqref{eq:CS-uni} satisfying the Cronstr\"om gauge condition \eqref{eq:cronstrom} in the set $\DD(B)$. Moreover, the solution obeys the bound
\begin{equation*}
\sup_{\hT \geq \hT_{0}} \bb( \nrm{A}_{H^{m-1}(B_{\hT})} + \nrm{\dlt A}_{H^{m-1}(B_{\hT})} + \nrm{\phi}_{H^{m}(B_{\hT})}^{2} + \nrm{\rd_{\hT} \phi}_{H^{m-1}(B_{\hT})}^{2} \bb) \leq C \tilde{\bt}^{2},
\end{equation*}
where $B_{\hT} := \calH_{\hT} \cap \DD(B)$.
\end{theorem}
As in the case of the temporal gauge, the Chern--Simons system \eqref{eq:CS-uni} becomes a coupled system of a Klein-Gordon equation for $\phi$ and transport equations for $A$ and $\dlt A$ whose characteristics are precisely the integral curves of the scaling vector field $S$ (or equivalently $\rd_{\hT}$). For the precise form of the system, we refer to Appendix~\ref{subsec:cronstrom}. Hence Theorem~\ref{thm:lwp} is again proved by a standard Picard iteration argument as in the case of Theorem~\ref{thm:lwp-temporal}.
We remark that the finite speed of propagation of the transport equation (more precisely, the fact that the solution to the tranport equation on $\DD(B)$ is determined solely on the data on $B$) follows from the fact that the characteristics are causal curves. 




\subsection{Reduction to the main a priori estimate}
Our goal now is to reduce the proof of the main theorems (Theorems~\ref{thm:CSH} and \ref{thm:CSD}) to establishing a priori estimates for solutions to \eqref{eq:CS-uni} with initial data obeying \eqref{eq:CSH-id} or \eqref{eq:CSD-id} with sufficiently small $\eps$. 

Before we state the main a priori estimates, we need to specify the class of solutions to which these estimates apply. As in the hypothesis of Proposition~\ref{prop:initial-hyp}, let $(a, f, g)$ be a smooth \eqref{eq:CSH} initial set obeying \eqref{eq:CSH-id} [resp. a smooth \eqref{eq:CSD} initial data set obeying \eqref{eq:CSD-id}] with $\eps \leq \dlt_{\ast \ast}(R)$, and consider the IVP for \eqref{eq:CSH} [resp. \eqref{eq:CSD}] with data on $\set{t = 2R}$. Applying Proposition~\ref{prop:initial-hyp}, there exists a smooth solution $(A, \phi)$ (unique up to smooth local gauge transformations) to the IVP on $(\set{\tau \leq 2R} \cup \calC_{R}^{c}) \cap \set{t \geq 2 R}$. Applying Theorem~\ref{thm:lwp}, $(A, \phi)$ extends as a smooth solution (again, unique up to smooth local gauge transformations) to a region of the form
\begin{equation*}
	\calO_{T} := (\set{\hT \leq T} \cup \calC_{R}^{c}) \cap \set{t \geq 2R}
\end{equation*}
for some $T > 2R$. 

In order to show that $(A, \phi)$ extends to a global solution to the future, we would need to show that $T$ can be taken to be $+ \infty$, if $\eps > 0$ is sufficiently small. The following a priori estimates is a key step.

\begin{proposition}[Main a priori estimate] \label{prop:main}
There exists $\dlt_{\ast} = \dlt_{\ast}(R) > 0$ such that for any $T > 2R$, the following statements hold. Let $(A, \phi)$ be a smooth solution to the IVP for \eqref{eq:CS-uni} on the spacetime region $\calO_{T}$ constructed by Proposition~\ref{prop:initial-hyp} and Theorem~\ref{thm:lwp} as above.
If $\eps \leq \dlt_{\ast}(R)$, then the solution obeys the following estimates for $2R \leq \hT \leq T$:
\begin{itemize}
\item {\bf $L^{2}$ bounds with growth.} For $0 \leq m \leq 4$,
\begin{equation} \label{eq:main:L2}
	\nrm{\cosh \hY \covZ^{(m)} \phi}_{L^{2}(\calH_{\hT}, \frac{\ud \Vol}{\cosh \hY})}
	+ \nrm{\covT \covZ^{(m)} \phi}_{L^{2}(\calH_{\hT}, \frac{\ud \Vol}{\cosh \hY})}
	\aleq \eps  \log^{m} (1+\hT) .
\end{equation}
\item {\bf Sharp $L^{\infty}$ decay.}
\begin{equation} 	\label{eq:main:Linfty}
	\nrm{\cosh \hY \phi}_{L^{\infty}(\calH_{\hT}, \frac{\ud \Vol}{\cosh \hY})} 
	+ \nrm{\cosh \hY \, \covN \phi}_{L^{\infty}(\calH_{\hT}, \frac{\ud \Vol}{\cosh \hY})} 
	+ \nrm{\covT \phi}_{L^{\infty}(\calH_{\hT}, \frac{\ud \Vol}{\cosh \hY})} 
	\aleq \eps \hT^{-1}.
\end{equation}
\end{itemize}
\end{proposition}

Assuming the validity of Proposition~\ref{prop:main} for the moment, we may now prove Theorems~\ref{thm:CSH} and \ref{thm:CSD}.
\begin{proof} [Proof of Theorems~\ref{thm:CSH} and \ref{thm:CSD} assuming Proposition~\ref{prop:main}]
First, we note that uniqueness up to smooth local gauge transformations follows from Proposition~\ref{prop:id-temporal} and Theorem~\ref{thm:lwp-temporal}. Hence it only remains to show global existence of the smooth solution $(A, \phi)$ to the future. 
For this purpose, it suffices to show that if $(A, \phi)$ obeys the hypothesis of Proposition~
\ref{prop:main} on some $\calO_{T}$, then it can be extended as a smooth solution to $\calO_{T'}$ for some $T' > T$.
Indeed, using Proposition~\ref{prop:main} we may then set up a simple continuity argument to prove global existence of $(A, \phi)$ to the future, as well as the desired estimates stated in Theorems~\ref{thm:CSH} and \ref{thm:CSD}. 

Such an extension statement is a consequence of \eqref{eq:main:L2}, Proposition~\ref{prop:id-cronstrom} and Theorem~\ref{thm:lwp}. 
Given any point $x_{0} \in \calH_{T}$, there exists $r_{\ast} > 0$, depending on $\covT \covZ \phi, \ldots, \covT \covZ^{(4)} \phi$ and $x_{0}$, such that Proposition~\ref{prop:id-cronstrom} (with $m = 5$) applies to $(a, f, g) = (A, \phi, \covD_{\rd_{\hT}} \phi) \restriction_{B}$ on $ B = {}^{(\calH_{T})} B_{r_{\ast}}(x_{0})$, which produces a gauge-transformed data set $(\tilde{a}, \tilde{f}, \tilde{g})$ on $B$. Choosing $r_{\ast} > 0$ smaller if necessary, we may use Theorem~\ref{thm:lwp} to find a unique smooth solution $(\tilde{A}, \tilde{\phi})$ to \eqref{eq:CS-uni} on $\DD(B)$ in the Cronstr\"om gauge with data $(\tilde{a}, \tilde{f}, \tilde{g})$. Undoing the gauge transformation from Proposition~\ref{prop:id-cronstrom}, we arrive at a smooth extension of $(A, \phi)$ to $\DD({}^{(\calH_{T})} B_{r_{\ast}}(x_{0}))$. Thanks to the support property \eqref{eq:initial-hyp:fsp}, such an extension need to be performed only for $x_{0}$ on the compact set $\calC_{R} \cap \calH_{T}$. Therefore, we can find $\calO_{T'}$ with $T' > T$ to which $(A, \phi)$ extends as desired. \qedhere

\end{proof}
As described in Section~\ref{subsec:outline}, Sections~\ref{sec:covVF}--\ref{sec:BA} of this article are devoted to the proof of the main a priori estimates (Proposition~\ref{prop:main}).

\section{Gauge covariant vector field method} \label{sec:covVF}
In this section, we develop the machinery of gauge covariant vector field method for the covariant Klein--Gordon equation
\begin{equation*}
	\covBox \phi \m \phi = N.
\end{equation*}
Throughout this section, we denote by $A$ a $\LieAlg$-valued connection 1-form, and by $\phi$ a $V$-valued function. We furthermore assume that $\phi$ and $N$ are sufficiently smooth and decaying towards the spatial infinity (for instance, $\phi(t), N(t) \in \calS$ uniformly in $t$). 
 
\subsection{A covariant energy inequality} \label{subsec:energy}
In this subsection, we derive an energy inequality for the covariant Klein--Gordon equation on constant $\hT$-hypersurfaces $\calH_{\hT}$ using the time-like Killing vector field $T_{0}$. This inequality is fundamental to our development of the gauge covariant version of the vector field method. The main result is as follows:
\begin{proposition}[Energy inequality for covariant Klein--Gordon equation] \label{prop:en}
Suppose that the curvature 2-form $F$ satisfies the bound 
\begin{equation} \label{eq:en:F}
	\int_{\hT_{0}}^{\hT_{1}} \sup_{\calH_{\hT}} \bb( \sum_{\mu} \abs{F(T_{\mu}, T_{0})}^{2} \bb)^{1/2} \, \ud \tau \leq C_{F}.
\end{equation}
for some constant $0 < C_{F} < \infty$.
Then there exists a constant $C = C(C_{F}) > 0$ such that for all $\hT_{0} \leq \hT \leq \hT_{1}$, we have
\begin{align}
\bb( \int_{\calH_{\hT}} \ed_{\calH_{\hT}}[\phi] \, \ud \Vol_{\calH_{\hT}} \bb)^{1/2}
\leq C \bb( \int_{\calH_{\hT_{0}}} \ed_{\calH_{\hT_{0}}}[\phi] \, \ud \Vol_{\calH_{\hT_{0}}} \bb)^{1/2}
	+ C \int_{\hT_{0}}^{\hT} \wnrm{\cosh \hY (\covBox \m 1) \phi(\tau')}_{L^{2}_{\tau'}}\, \ud \tau',  \label{eq:en}
\end{align}
where the energy density $\ed_{\calH_{\hT}}[\phi]$ is defined in \eqref{eq:ed} below. Moreover, there exist constants $c, C' > 0$ such that the integral of $\ed_{\calH_{\hT}}[\phi]$ obeys the following lower bounds:
\begin{align} \label{eq:en-phi}
	\int_{\calH_{\hT}} \ed_{\calH_{\hT}}[\phi] \, \ud \Vol_{\calH_{\hT}} 
	\geq& c \bb( \wnrm{\cosh \hY \phi}_{L^{2}_{\hT}}^{2} 
		+ \wnrm{\covN \phi}_{L^{2}_{\hT}}^{2}
		+ \wnrm{\hT^{-1} \covZ \phi}_{L^{2}_{\hT}}^{2} 
		+ \wnrm{\covT \phi}_{L^{2}_{\hT}}^{2} \bb)
\end{align}
\begin{equation} \label{eq:en-Nphi}
\begin{aligned}
	\int_{\calH_{\hT}} \ed_{\calH_{\hT}}[\phi] \, \ud \Vol_{\calH_{\hT}} 
	+ \sum_{\mu, \nu} \frac{C'}{\hT^{2}} \int_{\calH_{\hT}} \ed_{\calH_{\hT}}[\covZ_{\mu \nu} \phi] \, \ud \Vol_{\calH_{\hT}}  
\geq c \wnrm{\cosh \hY \covN \phi}_{L^{2}_{\hT}}^{2}
\end{aligned}
\end{equation}
\end{proposition}

We remind the reader that $\wnrm{\cdot}_{L^{p}_{\hT}} = {\nrm{\cdot}_{L^{p}(\calH_{\hT}, \frac{\ud \sgm}{\cosh \hY})}}$.

\begin{proof} 
In this proof, it will be convenient to employ the abstract index notation for tensors. 
Then the energy-momentum tensor associated to the covariant Klein--Gordon equation $(\covBox - 1) \phi = 0$ may be written as
\begin{equation*}
	\EM[\phi]_{a b} = \Re \brk{\covD_{a} \phi, \covD_{b} \phi} - \frac{1}{2} \met_{a b} (\brk{\covD_{c} \phi, \covD^{c} \phi} + \brk{\phi, \phi}).
\end{equation*}
It can be easily verified that $\EM[\phi]_{ab}$ is symmetric in $a,b$ and satisfies 
\begin{align}
	\nb^{a} \EM[\phi]_{a b} 
=&	\Re \brk{(\covBox \phi - \phi), \covD_{b} \phi} + \Re \brk{F_{ab} \cdot \phi, \covD^{a} \phi}.	\label{eq:div4EM}
\end{align}
Given a vector field $X$, we may define the 1- and 0-currents associated to $X$ by 
\begin{align*}
	\vC{X}[\phi]_{a} :=& \EM[\phi]_{ab} X^{b}, \\
	\sC{X}[\phi] :=& \frac{1}{2} \EM[\phi]_{ab} (\defT{X}_{\sharp})^{ab},
\end{align*}
where $\defT{X}_{ab}$ is the deformation tensor of $X$, given by
\begin{equation*}
	\defT{X}_{ab} = \nb_{a} X_{b} + \nb_{b} X_{a}
\end{equation*}
and $(\defT{X}_{\sharp})^{ab}$ is its metric dual, i.e., $(\defT{X}_{\sharp})^{ab} := \defT{X}_{cd} \, \met^{ca} \met^{db}$. The currents ${}^{(X)}P_{a}$ and $\sC{X}$ satisfy the divergence identity
\begin{equation} \label{eq:en-div}
	\nb^{a} (\vC{X}[\phi])_{a} = \sC{X}[\phi] + (\nb^{a} \EM_{ab}[\phi]) X^{b}.
\end{equation}

We now derive the energy inequality that will be used below, by considering the associated currents to the time-like Killing vector field $T_{0} = \rd_{t}$. Since $T_{0}$ is Killing, it satisfies
\begin{equation*}
	\nb^{a} T_{0}^{b} + \nb^{b} T_{0}^{a} = 0.
\end{equation*}
It follows that $\sC{T_{0}} = 0$ and therefore, by \eqref{eq:div4EM}, we have
\begin{equation} \label{eq:en-div-T}
	\nb^{a} (\vC{T_{0}}[\phi])_{a} 
	= \bb( \Re \brk{(\covBox \phi - \phi), \covD_{b} \phi} + \Re \brk{F_{ab} \cdot \phi, \covD^{a} \phi} \bb) (T_{0})^{b}.
\end{equation}
Integrating this identity over the spacetime region $\set{(\hT', y, \tht) : \hT_{0} \leq \hT' \leq \hT}$ and applying the divergence theorem, we obtain 
\begin{align*}
	\int_{\calH_{\hT}} \ed_{\calH_{\hT}}[\phi] \, \ud \sgm_{\calH_{\hT}}
	= & \int_{\calH_{\hT_{0}}} \ed_{\calH_{\hT_{0}}}[\phi] \, \ud \sgm_{\calH_{\hT_{0}}}
	- \int_{\hT_{0}}^{\hT} \int \Re\brk{(\covBox \m 1) \phi, \covT_{0} \phi} \, \ud \sgm_{\calH_{\hT'}} \, \ud \hT' \\
	& -  \int_{\hT_{0}}^{\hT} \int (\eta^{-1})^{\mu \nu} \Re \brk{F(T_{\mu}, T_{0}) \cdot \phi, \covT_{\nu} \phi} \, \ud \sgm_{\calH^{\hT'}} \, \ud \hT'
\end{align*}
and the energy density $\ed_{\calH_{\hT}}[\phi]$ is defined as
\begin{equation} \label{eq:ed}
	\ed_{\calH_{\hT}}[\phi] 
	= \vC{T_{0}}[\phi](N) = \EM[\phi](T_{0}, N).
\end{equation}
Assume, for the moment, that $E(\hT)$ obeys the lower bound \eqref{eq:en-phi}. We introduce the function
\begin{equation*}
	E(\hT) = \sup_{\hT_{0} \leq \hT' \leq \hT} \int_{\calH_{\hT'}} \ed_{\calH_{\hT'}}[\phi] \, \ud \sgm_{\calH_{\hT'}},
\end{equation*}
which is non-decreasing. By \eqref{eq:en-phi}, Cauchy--Schwarz and H\"older's inequality, we arrive at the bound
\begin{align*}
	E(\hT)
	\leq & E(\hT_{0})
	+ \frac{1}{c^{1/2}} \int_{\hT_{0}}^{\hT} \wnrm{\cosh \hY (\covBox \m 1) \phi}_{L^{2}_{\hT'}} E(\hT')^{1/2} \, \ud \hT'  \\
	& + \frac{1}{c} \int_{\hT_{0}}^{\hT} \bb( \sup_{\calH_{\hT'}} \bb( \sum_{\mu} \abs{F(T_{\mu}, T_{0})}^{2} \bb)^{1/2} \bb) E(\hT') \, \ud \hT'.
\end{align*}
Using the fact that $E(\hT)$ is non-decreasing, we may pull out a factor of $E(\hT)^{1/2}$ from each term on the right-hand side, which can then be cancelled on both sides. Then applying Gronwall's inequality to handle the last term, \eqref{eq:en} follows.

To complete the proof of the proposition, it only remains to verify the bounds \eqref{eq:en-phi} and \eqref{eq:en-Nphi}.
In the hyperboloidal polar coordinates, $T_{0}$ can be written as
\begin{equation*}
	T_{0} = \cosh \hY \rd_{\hT} - \sinh \hY \frac{\rd_{\hY}}{\hT}.
\end{equation*}
Note furthermore that $N = \rd_{\hT}$ is the future-pointing unit normal to each $\calH_{\hT}$. Therefore, the energy density associated to $T_{0}$ on $\calH_{\hT}$ is given by
\begin{equation} \label{eq:ed-first}
 \begin{aligned}
	\ed_{\calH_{\hT}}[\phi] 
	=&  \cosh \hY \EM[\phi](\rd_{\hT}, \rd_{\hT}) - \sinh \hY \EM[\phi](\frac{\rd_{\hY}}{\hT}, \rd_{\hT}) \\
	=& \frac{1}{2} \cosh \hY \bb( \abs{\covD_{\hT} \phi}^{2} +  \abs{\frac{1}{\hT} \covD_{\hY} \phi}^{2} \bb) 
		- \sinh \hY \, \Re \brk{\frac{1}{\hT} \covD_{\hY} \phi, \covD_{\hT} \phi} \\
	&	+ \frac{1}{2} \cosh \hY \bb( \abs{\frac{1}{\hT \sinh \hY} \covD_{\tht} \phi}^{2} + \abs{\phi}^{2} \bb) 
\end{aligned}
\end{equation}
By Cauchy-Schwarz, we have
\begin{align*}
	\ed[\phi] 
	\geq & \frac{1}{2} \bb( \cosh \hY \abs{\phi}^{2} 
					+e^{-\hY} \abs{\covD_{\hT} \phi}^{2} + e^{-\hY}  \abs{\frac{1}{\hT}\covD_{\hY} \phi}^{2}  + \cosh \hY \abs{\frac{1}{\hT \sinh \hY} \covD_{\tht} \phi}^{2} \bb) \\
	\geq & \frac{1}{2 \cosh \hY} \bb(  \abs{\cosh \hY \phi}^{2} + \abs{\covD_{\hT} \phi}^{2} + \abs{\frac{1}{\hT} \covD_{\hY} \phi}^{2}
			+ \abs{\frac{1}{\hT } \frac{\cosh \hY}{\sinh \hY} \covD_{\tht} \phi}^{2}	\bb).		
\end{align*}
Integrating over $\calH_{\hT}$ with respect to the induced measure $\ud \sgm_{\calH_{\hT}}$, then applying \eqref{eq:normal}, \eqref{eq:dTht}, \eqref{eq:Z01} and \eqref{eq:Z02}, we obtain 
\begin{equation*}
	\int_{\calH_{\hT}} \ed[\phi] \, \ud \sgm_{\calH_{\hT}}
	\geq c \bb( \wnrm{\cosh \hY \phi}_{L^{2}_{\hT}}^{2} + \wnrm{\hT^{-1} \covZ \phi}_{L^{2}_{\hT}}^{2} + \wnrm{\covN \phi}_{L^{2}_{\hT}}^{2} \bb).
\end{equation*}
Combined with the simple pointwise bound
\begin{equation*}
\abs{\covT \phi} \leq C \cosh \hY ( \abs{\covN \phi} + \hT^{-1} \abs{\covZ \phi}),
\end{equation*}
the desired lower bound \eqref{eq:en-phi} follows.

To prove \eqref{eq:en-Nphi}, note first that
\begin{align*}
\sinh \hY \, \Re \brk{\frac{1}{\hT} \covD_{\hY} \phi, \covD_{\hT} \phi}
\leq& \cosh \hY \abs{\hT^{-1} \covD_{\hY} \phi}^{2} + \frac{1}{4} \cosh \hY \abs{\covD_{\hT} \phi}^{2} \\
\leq& \frac{C'}{\hT^{2}} \sum_{\mu, \nu} \ed[\covZ_{\mu \nu} \phi]  + \frac{1}{4} \cosh \hY \abs{\covD_{\hT} \phi}^{2},
\end{align*}
for some $C' > 0$. Combined with \eqref{eq:ed-first}, we see that there exists $c > 0$ such that
\begin{equation} \label{eq:energy4S}
	\ed[\phi] + \frac{C'}{\hT^{2}} \sum_{\mu, \nu} \ed[\covZ_{\mu \nu} \phi] 
	\geq c \cosh \hY \abs{\covD_{\hT} \phi}^{2}  = \frac{c}{\cosh \hY}  \abs{\cosh \hY \covN \phi}^{2} ,
\end{equation}
from which \eqref{eq:en-Nphi} follows. \qedhere
%
%
\end{proof}

As a consequence of the identity \eqref{eq:en-div-T} in the proceeding proof, we may relate the energy on the initial hyperboloid $\calH_{\hT_{0}}$ with the energy on the constant time hypersurface $\Sgm_{t_{0}} = \set{t = t_{0}}$. 

This result will be used later to prove Lemma~\ref{lem:BA:ini}, which justifies the bootstrap assumptions at the initial hypersurface $\calH_{2R}$.
\begin{lemma} \label{lem:ini-en}
Let $\hT_{0} \geq t_{0}$. Then we have
\begin{align*}
\int_{\calH_{\hT_{0}}} \ed_{\calH_{\hT}}[\phi](\hT_{0}) \, \ud \Vol_{\calH_{\hT_{0}}}
\leq& \int_{\Sgm_{t_{0}}} \ed_{\Sgm_{t_{0}}}[\phi] \, \ud x^{1} \ud x^{2} \\
	& + \int_{\calR_{t=t_{0}}^{\hT=\hT_{0}}} \bb( \abs{(\covBox \m 1) \phi } \abs{\covT_{0} \phi} 
	+ \sum_{\mu} \abs{F(T_{\mu}, T_{0})} \abs{\covT_{\mu} \phi} \bb)\, \ud t \ud x^{1} \ud x^{2}
\end{align*}
where $\Sgm_{t_{0}} = \set{t=t_{0}}$, $\ed_{\Sgm_{t_{0}}}[\phi] = \frac{1}{2} \sum_{\mu} \abs{\covT_{\mu} \phi}^{2} + \frac{1}{2} \abs{\phi}^{2}$ and 
\begin{equation} \label{eq:ini-en:region}
	\calR_{t=t_{0}}^{\hT = \hT_{0}} := \set{(x^{0}, x^{1}, x^{2}) \in \bbR^{1+2} : x^{0} \geq t_{0}, \, (x^{0})^{2} - (x^{1})^{2} - (x^{2})^{2} \leq \hT_{0}}.
\end{equation}
\end{lemma}
\begin{proof} 
Note that $\ed_{\Sgm_{t_{0}}} = \vC{T_{0}}(T_{0})$, where $T_{0} = \rd_{t}$ is the future pointing unit normal to $\Sgm_{t_{0}}$. The lemma is an immediate consequence of integrating the identity \eqref{eq:en-div-T} over $\calR_{t=t_{0}}^{\hT = \hT_{0}}$ and applying the divergence theorem.  \qedhere
\end{proof}

\subsection{Gauge invariant Klainerman--Sobolev inequality}
In this subsection, we derive a gauge invariant version of the Klainerman--Sobolev inequality, which constitutes another key ingredient of the gauge covariant vector field method. 

\begin{proposition} \label{prop:KlSob}
Let $\phi$ be a smooth $V$-valued function on $\calH_{\hT}$. Then we have
\begin{equation} \label{eq:KlSob}
	\hT \nrm{\cosh \hY \phi}_{L^{\infty}(\calH_{\hT}, \frac{\ud \Vol}{\cosh \hY})}
	\leq C \sum_{k : 0 \leq k \leq 2} \nrm{\cosh \hY \covZ^{(k)} \phi}_{L^{2}(\calH_{\hT}, \frac{\ud \Vol}{\cosh \hY})}.
\end{equation}
\end{proposition}

Gauge invariant Klainerman--Sobolev inequality of this type was first established by Psarelli \cite{MR2131047,MR1672001} in $\bbR^{1+3}$. To make the present paper self-contained, we sketch a proof of Proposition~\ref{prop:KlSob}.

\begin{remark} 
Recall that in the rectilinear coordinates $(t, x^{1}, x^{2})$,  we have $t = \hT \cosh \hY$. Therefore, if the norm on the right-hand side were bounded, then \eqref{eq:KlSob} would imply that $\phi$ decays with the rate $t^{-1}$, which is sharp for the Klein--Gordon equation on $\bbR^{1+2}$. In our application below, however, the norm on the right-hand side will grow in $\hT$, which will result in a loss of decay. 
\end{remark}

In our proof of Proposition \ref{prop:KlSob}, we will employ some standard Sobolev inequalities on $\bbR^{2}$ and $\bbS^{1}$. For the reader's convenience, we state (without proofs) the necessary inequalities in the next lemma.

\begin{lemma} \label{lem:stdSob}
The following statements hold.
\begin{enumerate}
\item Let $\phi$ be a function in the Sobolev space $W^{1, 2}(\set{x \in \bbR^{2} : \abs{x} \leq 1})$. Then we have
\begin{equation} \label{eq:stdSob:L2L4}
	\nrm{\phi}_{L^{4}(\set{x \in \bbR^{2} : \abs{x} \leq 1}, \ud x)} \leq C \nrm{\phi}_{W^{1,2}(\set{x \in \bbR^{2} : \abs{x} \leq 1}, \ud x)}.
\end{equation}
\item Let $\phi$ be a function in the Sobolev space $W^{1, 4}(\set{x \in \bbR^{2} : \abs{x} \leq 1})$. Then we have
\begin{equation} \label{eq:stdSob:L4Linfty}
	\nrm{\phi}_{L^{\infty}(\set{x \in \bbR^{2} : \abs{x} \leq 1}, \ud x)} \leq C \nrm{\phi}_{W^{1,4}(\set{x \in \bbR^{2} : \abs{x} \leq 1}, \ud x)}.
\end{equation}
\item Let $\phi$ be a function in the Sobolev space $W^{1,2}(\bbS^{1})$. Then we have
\begin{equation} \label{eq:stdSob:S1}
	\nrm{\phi}_{L^{\infty}(\bbS^{1})} \leq C \nrm{\phi}_{W^{1,2}(\bbS^{1})}.
\end{equation}
\end{enumerate}
\end{lemma}

Another important ingredient of our proof is a version of the \emph{diamagnetic inequality} (also commonly referred to as \emph{Kato's inequality}), which allows us to relate covariant derivatives of a $V$-valued function with ordinary derivatives of its amplitude. 
\begin{lemma}[Diamagnetic inequality] \label{lem:diamag}
For any vector field $X$ and smooth $V$-valued function $\phi$ on $\calH_{1}$ or $\bbS^{1}$, we have
\begin{equation} \label{eq:diamag}
	\rd_{X} \abs{\phi} \leq \abs{\covD_{X} \phi},
\end{equation}
in the sense of distributions, i.e., the inequality holds after testing against smooth non-negative compactly supported functions on $\calH_{1}$ or $\bbS^{1}$. By the dual characterization of $L^{p}$ norms, it follows that $\rd_{X} \abs{\phi} \in L^{p}$ and $\nrm{\rd_{X} \abs{\phi}}_{L^{p}} \leq \nrm{\covD_{X} \phi}_{L^{p}}$ for all $1 \leq p < \infty$.
\end{lemma}
We omit the standard proof. 

\begin{proof} [Proof of Proposition \ref{prop:KlSob}]

By scaling, it suffices to prove the following inequality for smooth compactly supported functions $\phi$ on $\calH_{1}$:
\begin{equation} \label{eq:KlSob:key}
	\cosh \hY \abs{\phi(\hY, \tht)} \leq C \sum_{\alp : 0 \leq \abs{\alp} \leq 2} 
			\bb( \int_{0}^{\infty} \int_{\bbS^{1}} \cosh \hY' \abs{\bfZ^{\alp} \phi(\hY', \tht')}^{2}\,  \sinh \hY' \ud \tht' \ud \hY'\bb)^{\frac{1}{2}}.
\end{equation}

We begin by establishing \eqref{eq:KlSob:key} in the region $\hY \leq 1$. In this region, the point is that \eqref{eq:KlSob:key} reduces to its unweighted analogue on $\bbR^{2}$ through the relations 
\begin{equation} \label{eq:KlSob:cptY}
	\cosh \hY \aeq 1, \quad 
	\sinh \hY \aeq \hY.
\end{equation}
Here, the notation $A \aeq B$ means that there exist positive constants $0 < c \leq C$ such that $cA \leq B \leq CA$.

By \eqref{eq:stdSob:L2L4}, \eqref{eq:KlSob:cptY}, the diamagnetic inequality \eqref{eq:diamag} with $X = \rd_{\hY}, \frac{1}{\hY} \rd_{\tht}$ and the relations \eqref{eq:dY}, \eqref{eq:wdTht}, we have
\begin{align*}
	\nrm{\phi}_{L^{4}(\calH_{1} \cap \set{\hY \leq 1})} 
	\leq & C \bb( \nrm{\rd_{\hY}\abs{\phi}}_{L^{2}(\calH_{1} \cap \set{\hY \leq 1})} 
			+ \nrm{\frac{1}{\hY} \rd_{\tht}\abs{\phi}}_{L^{2}(\calH_{1} \cap \set{\hY \leq 1})} 
			+ \nrm{\phi}_{L^{2}(\calH_{1} \cap \set{\hY \leq 1})} \bb) \\
	\leq & C \sum_{\alp: 0 \leq \abs{\alp} \leq 1} \nrm{\covZ^{\alp} \phi}_{L^{2}(\calH_{1} \cap \set{\hY \leq 1})},
\end{align*}
and similarly
\begin{align*}
	\nrm{\covZ \phi}_{L^{4}(\calH_{1} \cap \set{\hY \leq 1})} 
	\leq & C \sum_{\alp: 1 \leq \abs{\alp} \leq 2} \nrm{\covZ^{\alp} \phi}_{L^{2}(\calH_{1} \cap \set{\hY \leq 1})}.
\end{align*}
Repeating the preceding argument with \eqref{eq:stdSob:L2L4} replaced by \eqref{eq:stdSob:L4Linfty}, we have
\begin{align*}
	\nrm{\phi}_{L^{\infty}(\calH_{1} \cap \set{\hY \leq 1})} 
	\leq & C \sum_{\alp: 0 \leq \abs{\alp} \leq 1} \nrm{\covZ^{\alp} \phi}_{L^{4}(\calH_{1} \cap \set{\hY \leq 1})}.
\end{align*}
Putting together the previous three inequalities and using $\cosh \hY \aeq 1$ to build in the appropriate weights, the desired inequality \eqref{eq:KlSob:key} in the region $\set{\hY \leq 1}$ follows.

Next, we turn to the task of proving \eqref{eq:KlSob:key} in the region $\hY \geq 1$. Using the fundamental theorem of calculus, Cauchy-Schwarz and \eqref{eq:dY}, we compute
\begin{align*}
	\cosh^{2} \hY \abs{\phi(\hY, \tht)}^{2} 
\leq & 2 \frac{\cosh \hY}{\sinh \hY} \abs{\int_{\hY}^{\infty} \cosh \hY' \Re \brk{\phi, \covD_{\rd_{\hY}} \phi}(\hY', \tht) \sinh \hY' \, \ud \hY' } \\
\leq & C \int_{0}^{\infty} \cosh \hY' (\abs{\phi}^{2} + \abs{\covZ \phi}^{2})(\hY', \tht) \sinh \hY' \, \ud \hY'.
\end{align*}
We have used the fact that $\frac{\cosh \hY}{\sinh \hY} \leq C$, which holds since $\hY \geq 1$. Applying the previous computation to $\covD_{\rd_{\tht}} \phi = \covZ_{12} \phi$, we obtain
\begin{align*}
	\cosh^{2} \hY \abs{\covD_{\rd_{\tht}} \phi(\hY, \tht)}^{2} 
\leq & C \int_{0}^{\infty} \cosh \hY' (\abs{\covZ \phi}^{2} + \abs{\covZ^{(2)} \phi}^{2})(\hY', \tht) \sinh \hY' \, \ud \hY'.
\end{align*}
Integrating the preceding two inequalities over $\tht \in \bbS^{1}$, we obtain
\begin{align*}
&	\cosh^{2} \hY \int_{\bbS^{1}} \abs{\phi(\hY, \tht)}^{2} + \abs{\covD_{\rd_{\tht}} \phi(\hY, \tht)}^{2} \, \ud \tht \\
& \quad	\leq C \sum_{\alp : 0 \leq \abs{\alp} \leq 2} 
			 \int_{0}^{\infty} \int_{\bbS^{1}} \cosh^{2} \hY' \abs{\bfZ^{\alp} \phi(\hT, \hY', \tht')}^{2}\,  \frac{\sinh \hY' \ud \tht' \ud \hY}{\cosh \hY'}.
\end{align*}
Now the desired inequality \eqref{eq:KlSob:key} follows from the combination of the standard Sobolev inequality \eqref{eq:stdSob:S1} and the diamagnetic inequality \eqref{eq:diamag} (with $X = \rd_{\tht}$) on $\bbS^{1}$. \qedhere
\end{proof}

\subsection{A gauge invariant ODE argument for sharp decay} \label{subsec:ODE}
Due to the specific structure of our problem, it turns out that the combination of the energy and the Klainerman--Sobolev inequality is insufficient.
What we need is a version of the ODE argument \cite{MR2188297, MR2056833} devised to handle the modified scattering behavior due to a long range effect, adapted to the gauge covariant setting. 

\begin{proposition} \label{prop:ODE}
For every $(\hY, \tht) \in [0, \infty) \times \bbS^{1}$ and $\hT \in (0, \infty)$, the following inequality holds:
\begin{equation} \label{eq:ODE}
\begin{aligned}
&\hskip-2em	
	\abs{\covD_{\hT} (\hT \cosh \hY \phi) (\hT, \hY, \tht)} + \hT \cosh \hY \abs{\phi (\hT, \hY, \tht)} \\
	\leq & C \bb( \abs{\covD_{\hT} (\hT \cosh \hY \phi) (\hT_{0}, \hY, \tht)} + \hT \cosh \hY \abs{\phi (\hT_{0}, \hY, \tht)} \bb) \\
	& + C \sum_{k: 1 \leq k \leq 2} \int_{\hT_{0}}^{\hT} \frac{\cosh \hY}{\hT'} \abs{\covZ^{(k)} \phi(\hT', \hY, \tht)} \, \ud \hT' \\
	& + C \int_{\hT_{0}}^{\hT} \hT' \cosh \hY \abs{(\covBox  \m 1)\phi(\hT', \hY, \tht)} \, \ud \hT'.
\end{aligned}\end{equation}
\end{proposition}
Our proof of this proposition is based on the following algebraic computation, which relates the induced covariant Laplacian on $\calH_{\hT}$ with $\omg_{j}$'s and $\covZ_{\mu \nu}$'s. 
\begin{lemma} \label{lem:covLapOnH}
Let $\lap_{A, \calH_{\hT}}$ be the induced covariant Laplacian on $\calH_{\hT}$, i.e.,
\begin{equation*}
\lap_{A, \calH_{\hT}}
:= \frac{1}{\hT^{2}} \bb( \frac{1}{\sinh \hY} \covD_{\rd_{\hY}} (\sinh \hY \covD_{\rd_{\hY}}) + \frac{1}{\sinh^{2} \hY} \covD_{\rd_{\tht}}^{2} \bb).
\end{equation*}
Then the following identity holds.
\begin{equation} \label{eq:covLapOnH}
\begin{aligned}
\lap_{A, \calH_{\hT}}
=& - \frac{\sinh^{2} \hY}{\hT^{2} \cosh^{2} \hY} \bb( \omg_{1}^{2} \covZ_{02}^{2} + \omg_{2}^{2} \covZ_{01}^{2}
		- \omg_{1} \omg_{2} (\covZ_{01} \covZ_{02} + \covZ_{02} \covZ_{01}) \bb) \\
&  + \frac{1}{\hT^{2}} (\covZ_{01}^{2} + \covZ_{02}^{2} ) - \frac{\sinh \hY}{\hT^{2} \cosh \hY} (\omg_{1} \covZ_{01} + \omg_{2} \covZ_{02})
\end{aligned}	
\end{equation}
\end{lemma}

\begin{proof} 
Using \eqref{eq:Z01}, \eqref{eq:Z02} and the fact that $\omg_{1}^{2} + \omg_{2}^{2} =1$, we compute
\begin{align*}
	\frac{1}{\sinh \hY} \covD_{\rd_{\hY}} (\sinh \hY \covD_{\rd_{\hY}})
	=&	  \covD_{\rd_{\hY}}^{2} +  \frac{\cosh \hY}{\sinh \hY} \covD_{\rd_{\hY}} \\
	=&	(\omg_{1} \covZ_{01} + \omg_{2} \covZ_{02})^{2} 
		- \frac{\cosh \hY}{\sinh \hY} (\omg_{1} \covZ_{01} + \omg_{2} \covZ_{02}) \\
	=&	\omg_{1}^{2} \covZ_{01}^{2} + \omg_{2}^{2} \covZ_{02}^{2} 
		+ \omg_{1} \omg_{2} (\covZ_{01} \covZ_{02} + \covZ_{02} \covZ_{01}) 	\\
	&	- \frac{\cosh \hY}{\sinh \hY} (\omg_{1} \covZ_{01} + \omg_{2} \covZ_{02}).
\end{align*}
Similarly, using \eqref{eq:wdTht}, we have
\begin{align*}
	\frac{\cosh^{2} \hY}{\sinh^{2} \hY} \covD_{\rd_{\tht}}^{2}
	=& \bb( \omg_{1} \covZ_{02} - \omg_{2} \covZ_{01} \bb)^{2} \\
	=& \omg_{1} \bb( \omg_{1} \covZ_{02} - \omg_{2} \covZ_{01} \bb) \covZ_{02} 
	- \omg_{2} \bb( \omg_{1} \covZ_{02} - \omg_{2} \covZ_{01} \bb) \covZ_{01} \\
	&	- \frac{\cosh \hY}{\sinh \hY} \rd_{\tht} \omg_{1} \covZ_{02} + \frac{\cosh \hY}{\sinh \hY} \rd_{\tht} \omg_{2} \covZ_{01} \\
	=&	\omg_{1}^{2} \covZ_{02}^{2} + \omg_{2}^{2} \covZ_{01}^{2}
		- \omg_{1} \omg_{2} (\covZ_{01} \covZ_{02} + \covZ_{02} \covZ_{01}) 
		+ \frac{\cosh \hY}{\sinh \hY} (\omg_{1} \covZ_{01} + \omg_{2} \covZ_{02}).
\end{align*}
Hence
\begin{align*}
& \frac{1}{\sinh \hY} \covD_{\rd_{\hY}} (\sinh \hY \covD_{\rd_{\hY}}) + \frac{1}{\sinh^{2} \hY} \covD_{\rd_{\tht}}^{2} \\
& \quad = \frac{1}{\sinh \hY} \covD_{\rd_{\hY}} (\sinh \hY \covD_{\rd_{\hY}}) + \frac{\cosh^{2} \hY}{\sinh^{2} \hY} \covD_{\rd_{\tht}}^{2} - \covD_{\rd_{\tht}}^{2} \\
&\quad = \covZ_{01}^{2} + \covZ_{02}^{2} 
	- \frac{\sinh^{2} \hY}{\cosh^{2} \hY} \bb( \omg_{1}^{2} \covZ_{02}^{2} + \omg_{2}^{2} \covZ_{01}^{2}
		- \omg_{1} \omg_{2} (\covZ_{01} \covZ_{02} + \covZ_{02} \covZ_{01}) \bb) \\
& \qquad - \frac{\sinh \hY}{\cosh \hY} (\omg_{1} \covZ_{01} + \omg_{2} \covZ_{02}). 
\end{align*}
Recalling the definition of $\lap_{A, \calH_{\hT}}$, the lemma follows. \qedhere
\end{proof}

With Lemma \ref{lem:covLapOnH} in hand, we are ready to prove Proposition \ref{prop:ODE}.
\begin{proof} [Proof of Proposition \ref{prop:ODE}]
We begin by expanding
\begin{align*}
	\covBox \phi \m \phi
	= & \m \frac{1}{\hT^{2}} \covD_{\hT} (\hT^{2} \covD_{\hT} \phi) \m \phi  \p \lap_{A, \calH_{\hT}} \phi.
\end{align*} 

Then by Lemma \ref{lem:covLapOnH}, we have
\begin{align*}
& \hskip-2em
\np \covD_{\hT}^{2} (\hT \cosh \hY \phi) \p (\hT \cosh \hY \phi) \\
=& 	\m \frac{\sinh^{2} \hY}{\hT \cosh \hY} \bb( \omg_{1}^{2} \covZ_{02}^{2} + \omg_{2}^{2} \covZ_{01}^{2}
	\m \omg_{1} \omg_{2} (\covZ_{01} \covZ_{02} + \covZ_{02} \covZ_{01}) \bb)\phi \\
&  	\p \frac{\cosh \hY}{\hT} (\covZ_{01}^{2} + \covZ_{02}^{2} )\phi 
	\m \frac{\sinh \hY}{\hT} (\omg_{1} \covZ_{01} + \omg_{2} \covZ_{02})\phi
	- (\hT \cosh \hY) (\covBox \phi \m \phi).
\end{align*}
Taking the inner product with $\covD_{\hT} (\hT \cosh \hY \phi)$, the left-hand side becomes 
\begin{equation*}
\frac{1}{2} \rd_{\hT} \bb( \abs{\covD_{\hT} (\hT \cosh \hY \phi)}^{2} + \abs{\hT \cosh \hY \phi}^{2} \bb).
\end{equation*}
Integrating in $\hT$ from $\hT_{0}$ and using Cauchy-Schwarz, \eqref{eq:ODE} follows.
\end{proof}

\subsection{Gauge invariant interpolation inequalities with weights}
In this subsection, we derive various interpolation inequalities involving $\covZ$ and weights of the form $\cosh \hY$. 

\begin{lemma} \label{lem:noZ}
Let $\phi$ be a smooth compactly supported $V$-valued function on $\calH_{\hT}$. 
Then for $1 \leq r \leq p \leq q \leq \infty$ and $0 \leq \vartht \leq 1$ defined by $\frac{1}{p} =  \frac{1-\vartht}{q} + \frac{\vartht}{r}$, we have
\begin{equation} \label{eq:noZ}
	\nrm{\cosh \hY  \phi}_{L^{p}(\calH_{\hT}, \frac{\ud \Vol}{\cosh \hY})}
	\leq C \nrm{\cosh \hY \phi}^{1- \vartht}_{L^{q}(\calH_{\hT}, \frac{\ud \Vol}{\cosh \hY})} 
		\nrm{\cosh \hY \phi}^{\vartht}_{L^{r}(\calH_{\hT}, \frac{\ud \Vol}{\cosh \hY})} 
\end{equation}
\end{lemma}

By scaling, we may take $\hT = 1$. Then this lemma is an easy consequence of H\"older's inequality with respect to the measure $(\cosh \hY)^{-1} \ud \Vol_{\calH_{1}}$. 

\begin{lemma}[Covariant Gagliardo-Nirenberg inequality with weights] \label{lem:Z}
Let $\phi$ be a smooth compactly supported function on $\calH_{\hT}$. 
Then for $2 \leq p, q, r \leq \infty$ and $\frac{2}{p} = \frac{1}{q} + \frac{1}{r}$, we have
\begin{equation} \label{eq:Z}
\begin{aligned}
&	\nrm{\cosh \hY \covZ \phi}_{L^{p}(\calH_{\hT}, \frac{\ud \Vol}{\cosh \hY})} \\
& \quad \leq C \nrm{\cosh \hY  \phi}_{L^{q}(\calH_{\hT}, \frac{\ud \Vol}{\cosh \hY})}^{\frac{1}{2}} 
	\bb( \sum_{k: 0 \leq k \leq 2} \nrm{\cosh \hY \covZ^{(k)} \phi}_{L^{r}(\calH_{\hT}, \frac{\ud \Vol}{\cosh \hY})} \bb)^{\frac{1}{2}}
\end{aligned}
\end{equation}
\end{lemma}

\begin{proof} 
To simplify the exposition, we will use the following notation: For $1 \leq r \leq \infty$ and $k \geq 0$ an integer, we will write
\begin{equation*}
\nrm{\cosh \hY \covZ^{(\leq k)} \phi}_{L^{r}(\calH_{\hT}, \frac{\ud \Vol}{\cosh \hY})}
:= \sum_{0 \leq k' \leq k} \nrm{\cosh \hY \covZ^{(k')} \phi}_{L^{r}(\calH_{\hT}, \frac{\ud \Vol}{\cosh \hY})}.
\end{equation*}
Also, by scaling, it suffices to consider the case $\hT = 1$. Below, we will omit $\calH_{1}$ and simply write $L^{p} = L^{p}(\calH_{1}, \frac{\ud \Vol}{\cosh \hY})$.

Before we embark on the proof, we will introduce a few necessary ingredients. Our first ingredient is the following integration by parts formula on $\calH_{1}$: For $Z = Z_{\mu \nu}$ ($\mu, \nu = 0, 1, 2$) and $f, g$ smooth compactly supported real-valued functions on $\calH_{1}$, it holds that
\begin{equation} \label{eq:intbyparts}
	\int_{\calH_{1}} (Z f) g \, \ud \sgm_{\calH_{1}} = - \int_{\calH_{1}} f Z g \, \ud \sgm_{\calH_{1}}.
\end{equation}
This identity, which is equivalent to saying that the orthogonal projection of $Z_{\mu \nu}$ to $\calH_{1}$ is divergence-free, is an immediate consequence of the fact that $Z_{\mu \nu}$ is a Killing vector field in the ambient Minkowski space $\bbR^{1+2}$ that is tangent to $\calH_{1}$.
Observe also that
\begin{equation*}
	\abs{Z (\cosh \hY)} \leq \cosh \hY. 
\end{equation*}
A quick way to verify this inequality is to note that $\cosh \hY = \frac{t}{\hT}$ and 
\begin{equation*}
\abs{Z_{\mu \nu} (\frac{t}{\hT})} = \frac{1}{\hT}\abs{x_{\mu} \dlt_{\nu}^{0}  - x_{\nu} \dlt_{\mu}^{0} } \leq \frac{t}{\hT}
\hbox{ on } \calH_{1}.
\end{equation*}

With these preparations, we are now ready to prove \eqref{eq:Z}. With \eqref{eq:intbyparts}, as well as the Leibniz rule \eqref{eq:leibniz-V}, this inequality can be easily shown using H\"{o}lder's inequality as follows: Writing $p = 2+ 2b$, where $b \geq 0$ since $p \geq 2$, we have
\begin{align*}
\wnrm{\cosh \hY \covZ \phi}_{L^{p}}^{p}
=& \int (\cosh \hY)^{p-1} \brk{\covZ \phi, \covZ\phi} \brk{\covZ \phi, \covZ \phi}^{b} \, \ud \Vol \\
= & - \int (\cosh \hY)^{p-1} \brk{\phi, \covZ^{2} \phi} \brk{\covZ \phi, \covZ \phi}^{b} \, \ud \Vol \\
&	+ \int (\cosh \hY)^{p-1} Z\brk{\phi, \covZ \phi} \brk{\covZ \phi, \covZ \phi}^{b} \, \ud \Vol \\
= 	& - \int (\cosh \hY)^{p} \brk{\phi, \covZ^{2} \phi} \brk{\covZ \phi, \covZ \phi}^{b} \, \frac{\ud \Vol}{\cosh \hY} \\
 & - 2 b \int (\cosh \hY)^{p} \brk{\phi, \covZ \phi} \brk{\covZ \phi, \covZ \phi}^{b-1} \brk{\covZ \phi, \covZ^{2} \phi} \, \frac{\ud \Vol}{\cosh \hY} \\
	& - (p-1) \int (\cosh \hY)^{p-1} \abs{Z (\cosh \hY)} \brk{\phi, \covZ \phi} \brk{\covZ \phi, \covZ \phi}^{b} \, \frac{\ud \Vol}{\cosh \hY}.
\end{align*}
Then the absolute value of the last expression is bounded from the above by
\begin{equation*}
\leq C  \wnrm{\cosh \hY \phi}_{L^{q}} 
		\wnrm{\cosh \hY\covZ^{(\leq 2)} \phi}_{L^{r}}
		\wnrm{\cosh \hY \covZ \phi}_{L^{p}}^{p-2},
\end{equation*}
from which \eqref{eq:Z} follows.  \qedhere
\end{proof}

\begin{lemma} \label{lem:GN}
Let $\phi$ be a smooth compactly supported function on $\calH_{\hT}$. For $0 \leq k \leq m-1$, we have
\begin{equation} \label{eq:GN}
\begin{aligned}
& \sum_{0 \leq \ell \leq k} \| \cosh \hY \, \covZ^{(\ell)} \phi \|_{L^{p}(\calH_{\hT}, \frac{\ud \Vol}{\cosh \hY})}   \\
& \qquad \leq C  \|  \cosh \hY \phi  \|^{1- \frac{k}{m}}_{L^{\infty}(\calH_{\hT}, \frac{\ud \Vol}{\cosh \hY})} 
\bb( \sum_{0 \leq \ell \leq m} \| \cosh \hY \, \covZ^{(\ell)} \phi\|_{L^2(\calH_{\hT}, \frac{\ud \Vol}{\cosh \hY})}  \bb)^{\frac{k}{m}} 
\end{aligned}
\end{equation}
where $\frac{1}{p} =  \frac{k}{2m}$.
\end{lemma}

\begin{proof} 
We will use the same notation and convention as the previous proof. Fix $m \geq 1$. We claim that the following holds: For $1 \leq k \leq m-1$ and $r_{k} := \frac{2m}{k}$,
\begin{equation} \label{eq:Zk}
	\wnrm{\cosh \hY \covZ^{(\leq k)} \phi}_{L^{r_{k}}}
	\leq C \wnrm{\cosh \hY \phi}_{L^{\infty}}^{\frac{1}{k+1}}
			\wnrm{\cosh \hY \covZ^{(\leq k+1)} \phi}_{L^{r_{k+1}}}^{\frac{k}{k+1}}\, .
\end{equation}
Indeed, \eqref{eq:GN} would follow from \eqref{eq:Zk} by induction on $k$.

To prove the inequality \eqref{eq:Zk}, we use a separate induction argument on $k$. The $k=1$ case follows from \eqref{eq:noZ} and \eqref{eq:Z} by taking $p = 2m$, $q = \infty$ and $r =m$. Next, assume that \eqref{eq:Zk} holds for some integer $k-1$ such that $1\leq k-1 \leq m-1$. Then for every integer $\ell$ satisfying $1 \leq \ell \leq k$, by \eqref{eq:Z} and the induction hypothesis, we have
\begin{align} \label{eq:GN:pf:1}
	&\wnrm{\cosh \hY \covZ^{(\ell)} \phi}_{L^{r_{k}}} \notag \\
	& \quad \leq C \wnrm{\cosh \hY \covZ^{(\ell+1)} \phi}_{L^{r_{k+1}}}^{\frac{1}{2}}
		\wnrm{\cosh \hY \covZ^{(\ell-1)} \phi}_{L^{r_{k-1}}}^{\frac{1}{2}} \\
	& \quad \leq C \wnrm{\cosh \hY \covZ^{(\leq k+1)} \phi}_{L^{r_{k+1}}}^{\frac{1}{2}}
		\bb( \wnrm{\cosh \hY \covZ^{(\leq k)} \phi}_{L^{r_{k}}}^{\frac{k-1}{k}} \wnrm{\cosh \hY \phi}_{L^{\infty}}^{\frac{1}{k}} \bb)^{\frac{1}{2}}. \notag
\end{align}
For $\ell = 0$, we have
\begin{equation} \label{eq:GN:pf:2}
	\wnrm{\cosh \hY \phi}_{L^{r_{k}}}
	\leq C \wnrm{\cosh \hY \phi}_{L^{\infty}}^{\frac{1}{k+1}} \wnrm{\cosh \hY \phi}_{L^{r_{k+1}}}^{\frac{k}{k+1}}
\end{equation}
by \eqref{eq:noZ}. Summing up \eqref{eq:GN:pf:1} for $1 \leq \ell \leq k$ and \eqref{eq:GN:pf:2}, we arrive at
\begin{align*}
& \wnrm{\cosh \hY \covZ^{(\leq k)} \phi}_{L^{r_{k}}} \\
& \quad \leq C \wnrm{\cosh \hY \covZ^{(\leq k+1)} \phi}_{L^{r_{k+1}}}^{\frac{1}{2}}
		\bb( \wnrm{\cosh \hY \covZ^{(\leq k)} \phi}_{L^{r_{k}}}^{\frac{k-1}{k}} \wnrm{\cosh \hY \phi}_{L^{\infty}}^{\frac{1}{k}} \bb)^{\frac{1}{2}}, 
\end{align*}
which then implies the desired estimate for $k$.
\end{proof}

\section{Commutation relations and structure of the equations} \label{sec:comm}
The purpose of this section is to compute the equation satisfied by $\covZ^{(k)} \phi$, using the commutation properties of \eqref{eq:CSH} and \eqref{eq:CSD} with respect to $\covZ_{\mu \nu}$. This is the main algebraic ingredient of our proof of Theorems~\ref{thm:CSH} and \ref{thm:CSD}. 

The main tool for our computation is the formalism of exterior differential calculus for $V$- and $\LieAlg$-valued forms introduced in Section~\ref{subsec:extr-calc}, which is summarized and further developed in Section~\ref{subsec:extr-calc-2}. Then in Section~\ref{subsec:comm-covBox}, we compute the commutator between $\covBox-1$ and $\covZ_{\mu \nu}$, under the Chern--Simons equation $F = \star J$. The (general order) commutator can be expressed in terms of covariant Lie derivatives of the current $J$ (i.e., $\covLD_{Z}^{(k)} J$); the latter is computed for \eqref{eq:CSH} and \eqref{eq:CSD} in Sections~\ref{subsec:comm-CSH} and \ref{subsec:comm-CSD}, respectively. 
In Section~\ref{subsec:comm-U}, we establish commutation properties of the $V$-valued potential $U(\phi)$. 
Finally, in Section~\ref{subsec:ptwise}, we provide rudimentary pointwise bounds for various expressions introduced in this section, which will be basic to the analysis performed in Section~\ref{sec:BA}.

\subsection{More on the exterior differential calculus} \label{subsec:extr-calc-2}
This section is a continuation of Section~\ref{subsec:extr-calc}. We begin by summarizing the key formulae of the exterior differential calculus of real-valued differential forms.
\begin{lemma}[Exterior differential calculus] \label{lem:extr-calc}
Given a real-valued $k$-form $\omg$ and vector fields $X$, $Y$, the following identities hold.
\begin{align}
	[\LD_{X}, \iota_{Y}] \omg = & \iota_{[X,Y]} \omg ,\\
	[\LD_{X}, \LD_{Y}] \omg = & \LD_{[X,Y]} \omg , \\
	[\LD_{X}, \ud] \omg =& 0, \\
	\ud^{2} \omg =& 0.
\end{align}
The following identity, called \emph{Cartan's formula}, also holds.
\begin{align} 
	\iota_{X} \ud \omg + \ud \iota_{X} \omg =& \LD_{X} \omg. \label{eq:cartan-eq}
\end{align}
Moreover, given a real-valued $\ell$-form $\omg'$, the following Leibniz rules hold.
\begin{align}
	\LD_{X} (\omg \wedge \omg') =& (\LD_{X} \omg) \wedge \omg' + \omg \wedge \LD_{X} \omg'			\\
	\iota_{X} (\omg \wedge \omg') = & (\iota_{X} \omg) \wedge \omg' + (-1)^{k} \omg \wedge \iota_{X} \omg'	\\
	\ud (\omg \wedge \omg') = & (\ud \omg) \wedge \omg' + (-1)^{k} \omg \wedge \ud \omg'.
\end{align}
\end{lemma}
Along with the facts that $\ud f$ is the usual differential on functions and $\iota_{X} \ud f = X f$, these identities completely characterize the operations $\LD_{X}$, $\iota_{X}$ and $\ud$.

Analogous calculus rules hold for $V$-valued differential forms.
\begin{lemma} \label{lem:extr-calc-V}
Given a $V$-valued $k$-form $v$ and a real-valued $\ell$-form $\omg$, we have
\begin{equation}
	\omg \wedge v = (-1)^{k \ell} v \wedge \omg.
\end{equation}
Let $A$ be a connection 1-form and $F$ the associated curvature 2-form. For any vector fields $X, Y$, the following identities hold.
\begin{align}
	[\covLD_{X}, \iota_{Y}] v = & \iota_{[X,Y]} v ,\label{eq:covLDiX}\\
	[\covLD_{X}, \covLD_{Y}] v = & \covLD_{[X,Y]} v + (\iota_{Y} \iota_{X} F) v, \\
	[\covLD_{X}, \covud] v =& (\iota_{X} F) \wedge v, \label{eq:covLDcovud},\\
	\covud^{2} v =& F \wedge v.		\label{eq:covud-covud}
\end{align}
The following version of Cartan's formula also holds.
\begin{align}
	\iota_{X} \covud v + \covud \iota_{X} v =& \covLD_{X} v. \label{eq:cartan-eq-V} 
\end{align}
Finally, given an additional real-valued $\ell$-form $\omg$, the following Leibniz rules hold.
\begin{align}
	\covLD_{X} (v \wedge \omg) =& (\covLD_{X} v) \wedge \omg + v \wedge \LD_{X} \omg
	\label{eq:leibniz-LD}	\\
	\iota_{X} (v \wedge \omg) = & (\iota_{X} v) \wedge \omg + (-1)^{k} v \wedge \iota_{X} \omg	\\
	\covud (v \wedge \omg) = & (\covud v) \wedge \omg + (-1)^{k} v \wedge \ud \omg. \label{eq:leibniz-covud}
\end{align}
\end{lemma}
The key difference from the real-valued case is that $\covud^{2} \neq 0$, but instead \eqref{eq:covud-covud} holds. When $v$ is a $0$-form (i.e., a $V$-valued function), this is precisely the definition of the curvature 2-form $F$. Then the general case of a $k$-form follows from \eqref{eq:leibniz-covud}, which in turn is straightforward. The proof of the rest of the lemma is more routine, using the definitions \eqref{eq:covud} and \eqref{eq:covLD}, as well as Lemma~\ref{lem:extr-calc}; we omit the details.

For a $\LieAlg$-valued $k$-form $a$, the covariant differential $\covud a$ and the covariant Lie derivative $\covLD_{X} a$ are defined using the adjoint action. The formulae in Lemma~\ref{lem:extr-calc-V} in hold verbatim for $\LieAlg$-valued differential forms. Moreover, it is clear that the following additional Leibniz rules hold.
\begin{lemma} \label{lem:leibniz-gV}
Let $A$ be a connection 1-form. Given a $\LieAlg$-valued $k$-form $a$ and a $V$-valued $\ell$-form $v$, the following Leibniz rules hold.
\begin{align}
	\covLD_{X} (a \wedge v) =& (\covLD_{X} a) \wedge v + a \wedge \covLD_{X} v			\\
	\iota_{X} (a \wedge v) = & (\iota_{X} a) \wedge v + (-1)^{k} a \wedge (\iota_{X} v)	\\
	\covud (a \wedge v) = & (\covud a) \wedge v + (-1)^{k} a \wedge (\covud v).
\end{align}
In particular, these Leibniz rules hold in the case $V = \LieAlg$, where $v = b$ is a $\LieAlg$-valued $\ell$-form. In this case, we have
\begin{equation}
	[a \wedge b] = (-1)^{k \ell+1} [b \wedge a].
\end{equation}
\end{lemma}

Next, we introduce some useful definitions for computations concerning the current $J$, performed in Sections~\ref{subsec:comm-CSH} and \ref{subsec:comm-CSD}. For a pair $\phi^{1}, \phi^{2} \in V$, let
\begin{equation} \label{eq:bbrk-def}
	\bbrk{\phi^{1}, \phi^{2}} = \frac{1}{2} \bb( \brk{\calT \phi^{1}, \phi^{2}} + \brk{\phi^{2}, \calT \phi^{1}} \bb).
\end{equation}
Observe that $\bbrk{\phi^{1}, \phi^{2}}$ is a $\LieAlg$-valued bilinear (over $\bbR$) form in $\phi^{1}, \phi^{2}$, which is anti-symmetric thanks to the anti-hermitian property of $\calT$. It obeys the following important Leibniz rule.
\begin{lemma} \label{lem:leibniz-bbrk}
Let $A$ be a connection 1-form. Given $V$-valued functions $\phi^{1}, \phi^{2}$ and a vector field $X$, the following Leibniz rule holds.
\begin{equation} \label{eq:leibniz-bbrk}
	\covD_{X} \bbrk{\phi^{1}, \phi^{2}} =  \bbrk{ \covD_{X} \phi^{1}, \phi^{2}} + \bbrk{\phi^{1}, \covD_{X} \phi^{2}}.
\end{equation}
\end{lemma}
\begin{proof} 
By symmetry, it suffices to prove 
\begin{equation} \label{eq:leibniz-bbrk-key}
	\covD_{X} \brk{\calT \phi^{1}, \phi^{2}} =  \brk{ \calT \covD_{X} \phi^{1}, \phi^{2}} + \brk{\calT \phi^{1}, \covD_{X} \phi^{2}}.
\end{equation}

Recall that $\covD_{X} \phi = \nb_{X} \phi + A(X) \cdot \phi$ on a $V$- (or $\LieAlg$-)valued function $\phi$. We introduce the shorthand $a = A(X)$. Fix an orthonormal basis $\set{e_{A}}$ of $\LieAlg$, so that $\calT^{A'} \varphi = \dlt^{A'A} e_{A} \cdot \varphi$ and $a = a^{A} e_{A}$. We denote the structure constants by $\set{c_{AB}^{C}} \subseteq \bbR$, where $\LieBr{e_{A}} {e_{B}} = c_{AB}^{C} e_{C}$.

Using \eqref{eq:leibniz-V} and the above conventions, the left-hand side of \eqref{eq:leibniz-bbrk-key} equals
\begin{align*}
& \hskip-2em
	\dlt^{A A'} \brk{e_{A} \cdot \phi^{1}, \phi^{2}} [a, e_{A'}]
	+ \dlt^{A A'} \bb( \brk{\covD_{X} (e_{A} \cdot \phi^{1}), \phi^{2}}  + \brk{e_{A} \cdot \phi^{1}, \covD_{X} \phi^{2}} \bb) e_{A'} \\
= 	& \dlt^{A A'} \brk{e_{A} \cdot \phi^{1}, \phi^{2}} [a, e_{A'}]
		+ \dlt^{A A'} \brk{[a, e_{A}] \cdot \phi^{1}, \phi^{2}} e_{A'} \\
	& + \dlt^{A A'} \bb( \brk{e_{A} \cdot \covD_{X} \phi^{1}, \phi^{2}}  + \brk{e_{A} \cdot \phi^{1}, \covD_{X} \phi^{2}} \bb) e_{A'}.
\end{align*}
Note that the last line is exactly the right-hand side of \eqref{eq:leibniz-bbrk-key}.
Hence the difference between the left- and the right-hand sides of \eqref{eq:leibniz-bbrk-key} is equal to
\begin{align*}
& \hskip-2em
	\dlt^{A A'} \brk{e_{A} \cdot \phi^{1}, \phi^{2}} \LieBr{a}{e_{A'}} + \dlt^{A A'} \brk{\LieBr{a}{e_{A}} \cdot \phi^{1}, \phi^{2}} e_{A'} \\
	= & a^{C} c_{CA'}^{D} \dlt^{A A'} \brk{e_{A} \cdot \phi^{1}, \phi^{2}} e_{D} + a^{C}  c_{CA}^{D} \dlt^{A A'} \brk{e_{D} \cdot \phi^{1}, \phi^{2}} e_{A'} \\
	= & (c_{CD}^{A'} \dlt^{D A} + c_{CD}^{A} \dlt^{D A'}) a^{C}  \brk{e_{A} \cdot \phi^{1}, \phi^{2}} e_{A'}.
\end{align*}
Therefore, to establish \eqref{eq:leibniz-bbrk-key}, it suffices to show
\begin{equation} \label{eq:str-const}
c_{CD}^{A'} \dlt^{D A} + c_{CD}^{A} \dlt^{D A'} = 0.
\end{equation}
The identity \eqref{eq:str-const} is a consequence of the bi-invariance of $\brk{\cdot, \cdot}_{\LieAlg}$, i.e., 
\begin{equation*}
	\brk{[e_{C}, e_{A'}], e_{A}}_{\LieAlg} +\brk{[e_{C}, e_{A}], e_{A'}}_{\LieAlg} = 0.
\end{equation*}
This completes the proof. 
 \qedhere
\end{proof}
\begin{remark}
In the case of abelian Chern--Simons--Higgs case (Example \ref{ex:a-CSH}), Lemma \ref{lem:leibniz-bbrk} is verified by 
simply computing
\[ \covD_X ( i\phi^1 \overline{ \phi^2} - i\overline{ \phi^1} \phi^2) 
=i \covD_X \phi^1 \overline{ \phi^2} - i\overline  {\covD_X \phi^1}\phi^2  
+ i \phi^1 \overline{ \covD_X\phi^2 }- i\overline{ \phi^1}   {\covD_X \phi^2}.\] 
\end{remark}

The anti-symmetric form $\bbrk{\cdot, \cdot}$ induces a $\LieAlg$-valued wedge product $\bbrk{v^{1} \wedge v^{2}}$ of $V$-valued forms, characterized by the relation
\begin{equation}	\label{eq:bbrk-wedge}
	\bbrk{(\phi^{1} \otimes \omg^{1}) \wedge (\phi^{2} \otimes \omg^{2})}
	= \bbrk{\phi^{1}, \phi^{2}} \otimes (\omg^{1} \wedge \omg^{2})
\end{equation}
for $V$-valued functions $\phi^{1}, \phi^{2}$ and real-valued differential forms $\omg^{1}, \omg^{2}$.

\begin{lemma} \label{lem:extr-calc-bbrk}
Let $v, w$ be $V$-valued $k$- and $\ell$-forms, respectively. Then we have
\begin{equation}	\label{eq:bbrk-comm}
	\bbrk{v \wedge w} = (-1)^{k \ell + 1} \bbrk{w \wedge v}. 
\end{equation}
Moreover, let $A$ be a connection 1-form. For any vector field $X$, the following Leibniz rules hold.
\begin{align}
	\covLD_{X} \bbrk{v \wedge w}
	= & \bbrk{\covLD_{X} v \wedge w} + \bbrk{v \wedge \covLD_{X} w},	\label{eq:leibniz-bbrk-LD} \\
	\iota_{X} \bbrk{v \wedge w}	
	= & \bbrk{\iota_{X} v \wedge w} + (-1)^{k} \bbrk{v \wedge \iota_{X} w}, 	\label{eq:leibniz-bbrk-iota} \\
	\covud \bbrk{v \wedge w}
	= & \bbrk{\covud v \wedge w} + (-1)^{k} \bbrk{v \wedge \covud w}.		\label{eq:leibniz-bbrk-d}
\end{align}
\end{lemma}
\begin{proof} 
Identities \eqref{eq:bbrk-comm} and \eqref{eq:leibniz-bbrk-iota} are immediate from the defining relation \eqref{eq:bbrk-wedge}, whereas
\eqref{eq:leibniz-bbrk} and \eqref{eq:leibniz-bbrk-d} follow from Lemmas~\ref{lem:extr-calc}, \ref{lem:leibniz-bbrk}. We omit the routine details. 	\qedhere
\end{proof}

To summarize, we have defined $V$-valued wedge products $v \wedge w$, $a \wedge v $ and 
$\LieAlg$-valued wedge products  $[a \wedge a']$, $\bbrk{v\wedge w}$ for $v, w \in V$ and $ a, a' \in \LieAlg$,
which obey appropriate Leibniz rules with respect to $\covLD _X$, $\iota_X$ and $\covud$. Note that in Lemma \ref{lem:leibniz-bbrk} and Lemma \ref{lem:extr-calc-bbrk}, such rules hold with covariant derivatives on both sides (in contrast to, say, \eqref{eq:leibniz-V} and \eqref{eq:leibniz-g}). 


We now turn to the properties concerning the Hodge star operator $\star$. In order to state the properties, we need a few definitions. First, we define the \emph{(exterior) codifferential} operator $\dlt$ for real-valued differential forms by the formula
\begin{align}
	\int \eta^{-1}(\ud \omg^{1}, \omg^{2}) \, \ud \Vol_{\bbR^{1+2}}
	=&	\int \eta^{-1}(\omg^{1}, \dlt \omg^{2}) \, \ud \Vol_{\bbR^{1+2}}, \label{eq:dlt}
\end{align}
where $\omg^{1}$ and $\omg^{2}$ are real-valued $k$- and $k+1$-forms, respectively.
The \emph{covariant codifferential} $\covdlt$ of a $V$-valued differential form $v$ is defined similarly using $\covud$ and the Minkowski metric for $V$-valued differential forms, which is naturally defined using $\brk{\cdot, \cdot}_{V}$ (for explicit alternative formulae for $\dlt$ and $\covdlt$, see \eqref{eq:star-dlt} below).
Given a vector field $X$, we will denote its metric dual 1-form by $X^{\flat}$, i.e., $X^{\flat}_{b} := X^{a} \met_{ab}$.

In the following two lemmas, we record some useful properties of $\star$.
\begin{lemma} \label{lem:star}
Let the base manifold be $\bbR^{1+2}$ with the Minkowski metric with signature $(-1, +1, +1)$. 
Given a real-valued $k$-form $\omg$ and a vector field $X$, we have
\begin{align}
	\star \star \omg =& - \omg,	 \label{eq:star-star} \\
	\iota_{X} \star \omg =& \star (\omg \wedge X^{\flat}),  \label{eq:star-iX}\\
	\dlt \omg =&  	(-1)^{k+1} \star \ud \star \omg.		\label{eq:star-dlt}
\end{align}
Moreover, if $Z$ is a Killing vector field, then $\LD_{Z}$ commutes with $\star$ and $\flat$, i.e., 
\begin{align}
	\LD_{Z} \star \omg =& \star \LD_{Z} \omg, \label{eq:star-LDZ}\\
	\LD_{Z} X^{\flat} =& (\LD_{Z} X)^{\flat}	
	= [Z, X]^{\flat}.						\label{eq:star-flat}
\end{align}
Finally, the formulae \eqref{eq:star-star}--\eqref{eq:star-LDZ} hold for any $V$- or $\LieAlg$-valued $k$-form, where $\dlt$, $\ud$ and $\LD_{Z}$ are replaced by the covariant counterparts $\covdlt$, $\covud$ and $\covLD_{Z}$.
\end{lemma}

\begin{proof} 
The identity \eqref{eq:star-star} is a quick consequence of the definition \eqref{eq:star-def}, whereas \eqref{eq:star-iX} follows from the fact that $\iota_{X}$ is dual to $X^{\flat} \wedge$, i.e.,
\begin{equation*}
	\eta^{-1}(\iota_{X} \omg^{1}, \omg^{2}) = \eta^{-1}(\omg^{1}, X^{\flat} \wedge \omg^{2}).
\end{equation*}
The identity \eqref{eq:star-dlt} follows from \eqref{eq:dlt}. For \eqref{eq:star-LDZ} and \eqref{eq:star-flat}, note that if $Z$ is Killing then
\begin{equation*}
	\LD_{Z} \eta = 0, \quad \LD_{Z} \eta^{-1} = 0, \quad \LD_{Z} \eps = 0.
\end{equation*}
From these facts, \eqref{eq:star-flat} follows immediately. To prove \eqref{eq:star-LDZ}, we compute
\begin{align*}
	\LD_{Z} \big( \eta^{-1} (\omg^{1} , \omg^{2}) \eps \big)
	= & \eta^{-1}(\LD_{Z} \omg^{1}, \omg^{2}) \eps + \eta^{-1} (\omg^{1}, \LD_{Z} \omg^{2}) \eps \\
	= & \LD_{Z} \omg^{1} \wedge \star \omg^{2} + \omg^{1} \wedge \star \LD_{Z} \omg^{2}, \\
	\LD_{Z} \big( \omg^{1} \wedge \star \omg^{2} \big)
	= & \LD_{Z} \omg^{1} \wedge \star \omg^{2} + \omg^{1} \wedge \LD_{Z} \star \omg^{2},
\end{align*}
and observe that the two left-hand sides are equal by \eqref{eq:star-def}.

Finally, note that the same proof goes through for $V$- or $\LieAlg$-valued differential forms.
\end{proof}

\begin{lemma} \label{lem:star-aux}
Let the base manifold be $\bbR^{1+2}$ with the Minkowski metric with signature $(-1, +1, +1)$. 
Given a $\LieAlg$-valued $k$-form $a$ and a $V$-valued $k$-form $v$ $(0 \leq k \leq 3)$, we have
\begin{equation} \label{eq:star-aux:k-form}
	a \wedge \star v = \star a \wedge v.
\end{equation}
Moreover, if $\phi$ is a $V$-valued 0-form (i.e., a $V$-valued function), then
\begin{equation} \label{eq:star-aux:0-form}
	(\star a) \wedge \phi = \star (a \wedge \phi).
\end{equation}
\end{lemma}

\begin{proof} 
This lemma follows from the real-valued counterparts
\begin{equation*}
	\omg_{1} \wedge \star \omg_{2} = \star \omg_{1} \wedge \omg_{2}, \quad
	(\star \omg_{1}) \wedge f = \star (\omg_{1} \wedge f),
\end{equation*}
which is easily seen to hold for real-valued $k$-forms $\omg_{1}, \omg_{2}$ and a $0$-form (i.e., a real-valued function) $f$. \qedhere
\end{proof}

In view of the characterizing relation \eqref{eq:star-def} of $\star$, we \emph{define} the real-valued bilinear form $\eta^{-1}(a \cdot v)$ for a $\LieAlg$-valued $k$-form $a$ and a $V$-valued $k$-form $v$ $(0 \leq k \leq 2)$ so that 
\begin{align}\label{eq:star-gV}
a \wedge \star v = \star a \wedge v = \eta^{-1}(a\cdot v) \eps.
\end{align}

The d'Alembertian operator $\Box$ can be expressed in terms of $\ud$ and $\star$ as follows.
Recalling the definition of the divergence, \eqref{eq:dlt} and \eqref{eq:star-dlt}, for any real-valued 1-form $\omg$, we have
\begin{equation*}
	\mathrm{div} \, \omg = - \dlt \omg = - \star \ud \star \omg.
\end{equation*}
Hence, it follows that $\Box f = \mathrm{div} (\ud f) = - \star \ud \star f$. By an entirely analogous computation, the covariant d'Alembertian operator can be expressed in terms of $\covud$ and $\star$ as well. We record this result as a lemma.
\begin{lemma} \label{lem:covBox}
Given a $V$-valued function $\phi$, we have
\begin{equation} \label{eq:covBox}
	\covBox \phi = - \star \covud \star \covud \phi.
\end{equation}
\end{lemma}

Because of the Chern--Simons equation $F = \star J$, expressions of the form $\iota_{Z} \star$ often need to be considered. Our final lemma in this subsection is a technical result, which can be used to compute the commutator between $\covud$ (or $\covdlt$) and $\iota_{Z} \star$.
\begin{lemma} \label{lem:d-i-star}
Let $Z$ be a Killing vector field and $v$ be a $V$-valued $k$-form.
\begin{align}
	\covud \, \iota_{Z} \star v
	=& (-1)^{k+1} \iota_{Z} \star \covdlt v + \star \covLD_{Z} v, \label{eq:d-i-star}\\
	\covdlt \, \iota_{Z} \star v
	=& (-1)^{k} \iota_{Z} \star \covud v + \star (v \wedge \ud Z^{\flat}). \label{eq:dlt-i-star}
\end{align}
\end{lemma}
\begin{proof} 
For \eqref{eq:d-i-star}, we compute using Lemmas~\ref{lem:extr-calc-V} and \ref{lem:star} as follows:
\begin{align*}
	\covud \, \iota_{Z} \star v
	=& - \iota_{Z} \covud \star v + \covLD_{Z} \star v \\
	=& (-1)^{k+1} \iota_{Z} \star \covdlt v + \star \covLD_{Z} v.
\end{align*}
Similarly, for \eqref{eq:dlt-i-star}, we again use Lemmas~\ref{lem:extr-calc-V} and \ref{lem:star} to compute
\begin{align*}
	\covdlt \, \iota_{Z} \star v
	=& (-1)^{3 - k} \star \covud \star \iota_{Z} \star v \\
	=& (-1)^{k} \star \covud (v \wedge Z^{\flat}) \\
	=& (-1)^{k} \iota_{Z} \star \covud v + \star (v \wedge \ud Z^{\flat}).	\qedhere
\end{align*}
\end{proof}

\subsection{Commutation relation for the covariant Klein--Gordon operator} \label{subsec:comm-covBox} 
Our goal here is to compute the commutator between the covariant Klein--Gordon operator $\covBox - 1$ and $\covZ^{(k)}$.
The basic computation is contained in the following lemma.
\begin{lemma} \label{lem:comm-covBox}
Let $J$ be a $\LieAlg$-valued 1-form, and $A$ be a connection 1-form satisfying the Chern--Simons equation $F = \star J$. Then given any $V$-valued function (viewed as a 0-form) $\phi$ and a Killing vector field $Z$, we have
\begin{equation}	\label{eq:comm-covBox}
	[\covZ, \covBox] \phi
	= \iota_{Z} \star \covud (J \wedge \phi)
		- \iota_{Z} \star (J \wedge \covud \phi)
	- \star (J \wedge \phi \wedge \ud Z^{\flat}).
\end{equation}
Moreover, given in addition a $\LieAlg$-valued 1-form $\Gmm$ and Killing vector fields $Z_{1}$ and $Z_{2}$, we have
\begin{align} 
	\covZ_{2} (\iota_{Z_{1}} \star \covud (\Gmm \wedge \phi))
	= &	\iota_{Z_{1}} \star \covud(\covLD_{Z_{2}} \Gmm \wedge \phi) 
		+ \iota_{Z_{1}} \star \covud(\Gmm \wedge \covZ_{2} \phi) \label{eq:comm-covBox-1} \\
	&  	+ \iota_{[Z_{2}, Z_{1}]} \star \covud (\Gmm \wedge \phi)
		+ \eta^{-1}(J \wedge Z_{2}^{\flat} \cdot (\Gmm \wedge \phi \wedge Z_{1}^{\flat})), \notag \\
	\covZ_{2} (- \iota_{Z_{1}} \star (\Gmm \wedge \covud \phi))
	= &	- \iota_{Z_{1}} \star (\covLD_{Z_{2}} \Gmm \wedge \covud \phi)
		- \iota_{Z_{1}} \star (\Gmm \wedge \covud (\covZ_{2} \phi)) \label{eq:comm-covBox-2}\\
	&	- \iota_{[Z_{2}, Z_{1}]} \star(\Gmm \wedge \covud \phi) 
		+ \eta^{-1}(\Gmm \wedge Z_{1}^{\flat} \cdot (J \wedge Z_{2}^{\flat} \wedge \phi)), \notag \\
	\covZ_{2} \star (\Gmm \wedge \phi \wedge \ud Z_{1}^{\flat})
	= &	\star (\covLD_{Z_{2}} \Gmm \wedge \phi \wedge \ud Z_{1}^{\flat})	\label{eq:comm-covBox-3}
		+ \star (\Gmm \wedge \covZ_{2} \phi \wedge \ud Z_{1}^{\flat})\\
	&	+ \star (\Gmm \wedge \phi \wedge \ud [Z_{2}, Z_{1}]^{\flat}).		\notag
\end{align}
\end{lemma}


We postpone the proof until the end of this section, and proceed to the computation of $[\covZ^{(k)}, \covBox - 1]$.
Given a $\LieAlg$-valued 1-form $\Gmm$, a $V$-valued function $\phi$ and a vector field $Z$, define
\begin{align}
	\frkN_{1}[\Gmm, \phi; Z] = & \iota_{Z} \star ( (\covud \Gmm) \wedge \phi ),		\label{eq:frkN-1} \\ 
	\frkN_{2}[\Gmm, \phi; Z] = & - 2 \iota_{Z} \star (\Gmm \wedge \covud \phi) ,	\label{eq:frkN-2}\\
	\frkN_{3}[\Gmm, \phi; Z] = & - \star (\Gmm \wedge \phi \wedge \ud Z^{\flat}) .		\label{eq:frkN-3}
\end{align}
Note that, by the Leibniz rule (Lemma~\ref{lem:leibniz-gV}), the sum $\frkN_{1} + \frkN_{2}$ satisfies
\begin{equation} \label{eq:N12-leibniz}
\iota_{Z} \star \covud (\Gmm \wedge \phi)
	- \iota_{Z} \star (\Gmm \wedge \covud \phi)		
	= \frkN_{1}[\Gmm, \phi; Z] + \frkN_{2}[\Gmm, \phi; Z].		
\end{equation}
The reason why we split \eqref{eq:N12-leibniz} into $\frkN_{1}$ and $\frkN_{2}$, rather than $\iota_{Z} \star \covud (\Gmm \wedge \phi)$ and $- \iota_{Z} \star (\Gmm \wedge \covud \phi)$, is simply because the former pair turns out to be more convenient to estimate.

For $\LieAlg$-valued 1-forms $\Gmm^{1}, \Gmm^{2}$, a $V$-valued function $\phi$ and vector fields $Z_{1}, Z_{2}$, we also define
\begin{equation} \label{eq:frkN-4}
\begin{aligned}
	\frkN_{4}[\Gmm^{1}, \Gmm^{2}, \phi; Z_{1}, Z_{2}] 
	=& \eta^{-1}(\Gmm^{2} \wedge Z_{2}^{\flat} \cdot (\Gmm^{1} \wedge \phi \wedge Z_{1}^{\flat})) \\
	& + \eta^{-1}(\Gmm^{1} \wedge Z_{1}^{\flat} \cdot (\Gmm^{2} \wedge \phi \wedge Z_{2}^{\flat})) .
\end{aligned}
\end{equation}

In application, keeping track of the exact Killing vector field $Z$ (or $Z_{1}, Z_{2}$) involved in these formulae is not important. Accordingly, in what follows we often use the simple schematic notation
\begin{equation*}
	\frkN_{j}[\Gmm, \phi] = \frkN_{j}[\Gmm, \phi; Z] \quad (j=1,2,3), \quad
	\frkN_{4}[\Gmm^{1}, \Gmm^{2}, \phi] = \frkN_{4}[\Gmm^{1}, \Gmm^{2}, \phi; Z_{1}, Z_{2}]
\end{equation*}
where $Z$, $Z_{1}$ and $Z_{2}$ are understood to be one of the vector fields $Z_{\mu \nu}$.

With this convention in mind, we are finally able to state the main result of this section in a fairly compact form.
\begin{proposition} \label{prop:comm-covKG}
Let $J$ be a $\LieAlg$-valued 1-form, and $A$ a connection 1-form satisfying the Chern--Simons equation $F = \star J$.
Let $\phi$ be a $V$-valued function. Then for $m \geq 1$, the following schematic commutation formula holds:
\begin{equation} \label{eq:comm-covKG}
\begin{aligned}
{[\covZ^{(m)}, \covBox \m 1]} \phi 
=& \sum_{k_{1}+k_{2} \leq m-1} \frkN_{1}[\covLD_{Z}^{(k_{1})} J, \covZ^{(k_{2})} \phi] 
	+ \sum_{k_{1}+k_{2} \leq m-1} \frkN_{2}[\covLD_{Z}^{(k_{1})} J, \covZ^{(k_{2})} \phi] \\
&	+ \sum_{k_{1}+k_{2} \leq m-1} \frkN_{3}[\covLD_{Z}^{(k_{1})} J, \covZ^{(k_{2})} \phi] 
	+ \sum_{k_{1}+k_{2}+k_{3} \leq m-2} \frkN_{4}[\covLD_{Z}^{(k_{1})} J, \covLD_{Z}^{(k_{2})} J, \covZ^{(k_{3})} \phi]
\end{aligned}
\end{equation}
where the last sum should be omitted in the case $m = 1$.
\end{proposition}
We remind the reader that by a schematic formula, we mean that the left-hand side equals a linear combination of terms on the right-hand side.

This proposition is an immediate consequence of Lemma~\ref{lem:comm-covBox}, the definitions \eqref{eq:frkN-1}--\eqref{eq:frkN-4}, and the Lie algebra relation among $\set{Z_{\mu \nu}}$.
Hence it is only left to establish Lemma~\ref{lem:comm-covBox}; this is done using the tools developed in Section~\ref{subsec:extr-calc-2}.
\begin{proof} [Proof of Lemma~\ref{lem:comm-covBox}]
We begin by establishing \eqref{eq:comm-covBox}.
On $V$-valued functions, the differential operator $\covZ$ is identical to $\covLD_{Z}$. Using \eqref{eq:covBox} and the fact that $Z$ is a Killing vector field (hence $\covLD_{Z}$ commutes with $\star$), the left-hand side of \eqref{eq:comm-covBox} is equal to
\begin{equation*}
-  \star \bb( [\covLD_Z, \covud] \star \covud \phi \bb) 
- \star \covud \star [ \covLD_Z, \covud] \phi.
\end{equation*} 
Using \eqref{eq:covLDcovud} and the Chern--Simons equation $F = \star J$, we may replace the commutator by $(\iota _Z  \star J) \wedge$. Applying \eqref{eq:star-iX}, the preceding expression is then equal to
\begin{align*}
& \hskip-2em
	- \star \bb( \star (J \wedge Z^{\flat}) \wedge \star \covud \phi \bb) - \star \covud \star \bb( \star (J \wedge Z^{\flat}) \wedge \phi \bb) \\
& =
	- \star \bb( \star \star (J \wedge Z^{\flat}) \wedge \covud \phi \bb) - \star \covud \star \star (J  \wedge Z^{\flat} \wedge \phi)  \\
& =
	\star (J \wedge Z^{\flat} \wedge \covud \phi ) + \star \covud (J \wedge Z^{\flat} \wedge \phi),
\end{align*}
where we have used Lemma~\ref{lem:star-aux} and \eqref{eq:star-star}. The desired identity \eqref{eq:comm-covBox} now follows from Lemma~\ref{lem:leibniz-gV} and \eqref{eq:star-iX}. 

The identities \eqref{eq:comm-covBox-1}--\eqref{eq:comm-covBox-3} follow from routine computation, as in the preceding proof of \eqref{eq:comm-covBox}; hence we only sketch the proofs and leave the details to the reader. All these identities are proved by first replacing $\covZ_{2}$ by $\covLD_{Z_{2}}$, applying \eqref{eq:star-iX} and then using the Leibniz rule \eqref{eq:leibniz-LD}. Since $Z_{2}$ is Killing, note that $\covLD_{Z_{2}}$ commutes with $\star$. We remark that the last terms in \eqref{eq:comm-covBox-1} and \eqref{eq:comm-covBox-2} arise from the commutator $[\covLD_{Z_{2}}, \covud]$, \eqref{eq:covLDcovud}, the Chern--Simons equation $F = \star J$, and the definition \eqref{eq:star-gV}.  \qedhere
\end{proof}

\subsection{Covariant Lie derivatives of $J_{\CSH}$} \label{subsec:comm-CSH}
In Proposition~\ref{prop:comm-covKG}, the commutator between $\covZ^{(m)}$ and $\covBox - 1$ was computed in terms of $\covZ^{(k)} \phi$ and the covariant Lie derivatives of $J$. In this subsection, we compute the latter in terms of $\varphi$ in the case of \eqref{eq:CSH}.

We begin by defining the following $\LieAlg$-valued differential form, which is bilinear (over $\bbR$) in the $V$-valued functions $\varphi^{1}, \varphi^{2}$:
\begin{equation} \label{eq:Gmm-CSH-0}
	\Gmm_{\CSH}^{(0)}[\varphi^{1}, \varphi^{2}]
	=  2 \bbrk{\varphi^{1} \wedge \covud \varphi^{2}},
\end{equation}
where $\bbrk{\cdot \wedge \cdot}$ has been defined in \eqref{eq:bbrk-def} and \eqref{eq:bbrk-wedge}. Recall that
\[ J_{\CSH}(\varphi) = \brk{ \calT \varphi, \covud \varphi} + \brk{ \covud \varphi, \calT \varphi}.\]
Hence, we may write
\begin{equation*}
J_{\CSH}(\varphi) = \Gmm_{\CSH}^{(0)}[\varphi, \varphi] = 2 \bbrk{\varphi \wedge \covud \varphi}.
\end{equation*} 
We also define
\begin{equation} \label{eq:Gmm-CSH-1}
	\Gmm_{\CSH}^{(1)}[\varphi^{1}, \varphi^{2}, \varphi^{3}, \varphi^{4}; Z]
	=  2 \bbrk{\varphi^{1} \wedge (\iota_{Z}  \star \Gmm_{\CSH}^{(0)}[\varphi^{3}, \varphi^{4}] \wedge \varphi^{2})},
\end{equation}
where each $\varphi^{j}$ is a $V$-valued function and $Z$ is a vector field. The relevance of $\Gmm_{\CSH}^{(1)}$ is made clear by the following lemma, in which the commutation relation of $\Gmm_{\CSH}^{(0)}$ is computed.
\begin{lemma} \label{lem:comm-J-CSH-1}
Let $A$ be a connection 1-form and $\varphi$ a $V$-valued function, which satisfy the Chern--Simons equation $F = \star J_{\CSH}(\varphi)$. 
Then for any $V$-valued functions $\varphi^{1}, \varphi^{2}$ and a Killing vector field $Z$, we have
\begin{equation*}
	\covLD_{Z} \Gmm_{\CSH}^{(0)}[\varphi^{1}, \varphi^{2}]
	= \Gmm_{\CSH}^{(0)}[\covZ \varphi^{1}, \varphi^{2}] + \Gmm_{\CSH}^{(0)}[\varphi^{1}, \covZ \varphi^{2}]
	+ \Gmm_{\CSH}^{(1)}[\varphi^{1}, \varphi^{2}, \varphi, \varphi; Z]
\end{equation*}
\end{lemma}
\begin{proof} 
This identity  follows from the Leibniz rule 
\eqref{eq:leibniz-bbrk-LD} and \eqref{eq:covLDcovud}.
\end{proof}

Based on the observation that $\Gmm_{\CSH}^{(0)}$ occurs in $\Gmm_{\CSH}^{(1)}$, higher order covariant Lie derivative $\covLD_{Z}^{(m)} J_{\CSH}(\varphi)$ can also be computed using Lemma~\ref{lem:comm-J-CSH-1} and the Leibniz rule \eqref{eq:leibniz-bbrk-LD}. We begin by making the following recursive definition of $\Gmm_{\CSH}^{(m)}$ for all integers $m \geq 1$:
\begin{equation*}
\begin{aligned}
& \hskip-2em
	\Gmm_{\CSH}^{(m)}[\varphi^{1}, \varphi^{2}, \ldots, \varphi^{2m+2}; Z_{1}, Z_{2} \ldots, Z_{m}] \\ 
	= & 2 \bbrk{\varphi^{1} \wedge (\iota_{Z_{1}}  \star \Gmm_{\CSH}^{(m-1)}[\varphi^{3}, \ldots, \varphi^{2m+2}; Z_{2}, \ldots, Z_{m}] \wedge \varphi^{2})}.
\end{aligned}
\end{equation*}
Here, each $\varphi^{j}$ is a $V$-valued function, and $Z_{j}$ is a vector field. 

Using this definition, it is not difficult to write down a formula for $\covLD_{Z_{1}} \cdots \covLD_{Z_{m}} J_{\CSH}(\varphi)$ for any $m$. However, the exact formula is rather long and unwieldy. As discussed earlier, we need not keep track of each $Z_{j}$, so we may simply write 
\begin{equation*}
\Gmm_{\CSH}^{(m)}[\varphi^{1}, \varphi^{2}, \ldots, \varphi^{2m+2}] 
= \Gmm_{\CSH}^{(m)}[\varphi^{1}, \varphi^{2}, \ldots, \varphi^{2m+2}; Z_{1}, Z_{2} \ldots, Z_{m}], 
\end{equation*}
with the understanding that each $Z_{j}$ is one of the Killing vector fields $Z_{\mu \nu}$. 
With this convention, we may write down the following compact schematic formula, which suffices for our use.
\begin{proposition} \label{prop:comm-J-CSH}
Let $A$ be a connection 1-form and $\varphi$ a $V$-valued function, which satisfy the Chern--Simons equation $F = \star J_{\CSH}(\varphi)$. 
Then for any $m \geq 1$, the following schematic formula holds:
\begin{equation} \label{eq:comm-J-CSH}
	\covLD^{(m)}_{Z} J_{\CSH}(\varphi)
 	=  \sum_{\ell=0}^{m} \bb( \sum_{k_{1} + \cdots k_{2\ell+2} \leq m - \ell} \Gmm_{\CSH}^{(\ell)}[\covZ^{(k_{1})} \varphi, \cdots, \covZ^{(k_{\ell}+2)} \varphi] \bb).
\end{equation}
\end{proposition}
This proposition may be proved by induction on $m$, using the Leibniz rule \eqref{eq:leibniz-bbrk-LD}, the commutation formulae \eqref{eq:covLDiX} and \eqref{eq:covLDcovud}, and the Lie algebra relation among $\set{Z_{\mu \nu}}$. We omit the straightforward details.

For specific values $m = 2, 3$, $\Gmm_{\CSH}^{(m)}$ takes the following form:
\begin{align}
\Gmm_{\CSH}^{(2)}[\varphi^{1}, \varphi^{2}, \ldots, \varphi^{6}]  \label{eq:Gmm-CSH-2}
= &  2^{2} \bbrk{\varphi^{1} \wedge (\bbrk{\varphi^{3} \wedge ((\iota_{Z} \star)^{2} \Gmm_{\CSH}^{(0)}[\varphi^{5}, \varphi^{6}] \wedge \varphi^{4}) } \wedge\varphi^{2})}, \\
\Gmm_{\CSH}^{(3)}[\varphi^{1}, \varphi^{2}, \ldots, \varphi^{8}]
=&	2^{3} \bbrk{\varphi^{1} \wedge (\bbrk{\varphi^{3} \wedge (\bbrk{\varphi^{5} \wedge ((\iota_{Z} \star)^{3} \Gmm_{\CSH}^{(0)}[\varphi^{7}, \varphi^{8}] \wedge \varphi^{6})} \wedge \varphi^{4}) }\wedge \varphi^{2})}.
	\label{eq:Gmm-CSH-3}
\end{align}
These turn out to be only cases needed for our application in Section~\ref{sec:BA}.

We end with computation of $\covud \Gmm_{\CSH}^{(0)}$ and $\covdlt \Gmm_{\CSH}^{(0)}$; it will be used in conjunction with Lemma~\ref{lem:d-i-star} and Proposition~\ref{prop:comm-J-CSH} to compute $\covud \covLD_{Z}^{(m)} J_{\CSH}$.
\begin{lemma} \label{lem:covud-covdlt-CSH}
Let $A$ be a connection 1-form and $\varphi$ a $V$-valued function, which satisfy the Chern--Simons equation $F = \star J_{\CSH}(\varphi)$. Then for any $V$-valued functions $\varphi^{1}$ and $\varphi^{2}$, we have
\begin{align}
	\covud \Gmm_{\CSH}^{(0)}[\varphi^{1}, \varphi^{2}]
	=& \bbrk{\covud \varphi^{1} \wedge \covud \varphi^{2}}
		+ \bbrk{\varphi^{1} \wedge (\star \Gmm_{\CSH}^{(0)}[\varphi, \varphi] \wedge \varphi^{2})}, \label{eq:covud-CSH} \\
	\covdlt \Gmm_{\CSH}^{(0)}[\varphi^{1}, \varphi^{2}]
	=& \star \bbrk{\covud \varphi^{1} \wedge \star \covud \varphi^{2}} - \bbrk{\varphi^{1}, (\covBox \m 1) \varphi^{2}} \m \bbrk{\varphi^{1}, \varphi^{2}}. 			\label{eq:covdlt-CSH}
\end{align}
\end{lemma}
\begin{remark} 
Observe that $\bbrk{\covud \varphi^{1} \wedge \covud \varphi^{2}}$ has a structure similar to the classical null form $Q_{\mu \nu}(f, g) = \rd_{\mu} f \rd_{\nu} g - \rd_{\nu} f \rd_{\mu} g$. Moreover, a further computation using the definition of $\star$ shows that 
\begin{equation*}
\star \bbrk{\covud \varphi^{1} \wedge \star \covud \varphi^{2}} = - \eta^{\mu \nu} \bbrk{\covT_{\mu} \varphi^{1}, \covT_{\nu} \varphi^{2}}.
\end{equation*}
which resembles the classical null form $Q_{0}(f, g) = \rd^{\mu} f \rd_{\mu} g$. These structures do not play any role in the analysis that follows, due to the fast enough pointwise decay rate of the Klein--Gordon equation. However, they should be essential in the case of \emph{massless} Chern--Simons--Higgs equation.
\end{remark}

\begin{proof} 
Using \eqref{eq:leibniz-bbrk-d}, the Leibniz rule and \eqref{eq:covud-covud}, we have
\begin{align*}
	\covud \Gmm_{\CSH}^{(0)}[\varphi^{1}, \varphi^{2}]
	=& \bbrk{\covud \varphi^{1} \wedge \covud \varphi^{2}}
		+ \bbrk{\varphi^{1} \wedge (\star \Gmm_{\CSH}^{(0)}[\varphi, \varphi] \wedge \varphi^{2})},
\end{align*}
which proves \eqref{eq:covud-CSH}.

To prove \eqref{eq:covdlt-CSH}, we begin by writing out
\begin{align*}
	\covdlt \Gmm_{\CSH}^{(0)}[\varphi^{1}, \varphi^{2}]
	=& \star \covud \star \bbrk{\varphi^{1} \wedge \covud \varphi^{2}}.
\end{align*}
Recalling the definition \eqref{eq:bbrk-wedge} of $\bbrk{\cdot \wedge \cdot}$, it is immediate that
\begin{align*}
	\star \bbrk{\varphi \wedge v} = \bbrk{\varphi \wedge \star v}
\end{align*}
when $\varphi$ is a $V$-valued $0$-form (i.e., a $V$-valued function) and $v$ is a $V$-valued $k$-form. Using \eqref{eq:leibniz-bbrk-d}, the Leibniz rule and \eqref{eq:covBox}, we have
\begin{align*}
	\covdlt \Gmm_{\CSH}^{(0)}[\varphi^{1}, \varphi^{2}]
	=& \star \covud \bbrk{\varphi^{1} \wedge \star \covud \varphi^{2}} \\
	=& \star \bbrk{\covud \varphi^{1} \wedge \star \covud \varphi^{2}} + \star \bbrk{\varphi^{1} \wedge (\covud \star \covud \varphi^{2})} \\
	=& \star \bbrk{\covud \varphi^{1} \wedge \star \covud \varphi^{2}} - \bbrk{\varphi^{1}, \covBox \varphi^{2}},	
\end{align*} 
which finishes the proof.
\end{proof}

\subsection{Covariant Lie derivatives of $J_{\CSD}$} \label{subsec:comm-CSD}
Here we compute $\covLD_{Z}^{(m)} J$ in the case of \eqref{eq:CSD}.
We remind the reader that 
\[J_{\CSD} (\psi)= - \brk{i \alp \calT \psi,  \psi},\] where 
$\alp = \eta_{\mu \nu} \alp^{\mu} \ud x^{\nu}$ is a $2 \times 2$ matrix-valued 1-form with $ \alp^{\mu} = \ga^0 \ga^{\mu}$. 
The matrix $\alp^{\mu}$ acts on $\calT \psi$ naturally, i.e., $\alp^{\mu} \calT \psi = \sum_{A} e_{A} \otimes \alp^{\mu} \calT^{A} v$ for any orthonormal basis $\set{e_{A}}$ of $\LieAlg$. 

Since $\gmm^{0}$ is hermitian, $\gmm^{j}$ is anti-hermitian and $\gmm^{0} \gmm^{j} + \gmm^{j} \gmm^{0} = 0$ $(j=1,2)$, it follows that $\alp^{\mu}$ is \emph{hermitian} for $\mu=0, 1, 2$. On the other hand, each $\calT^{A}$ is \emph{anti-hermitian}. Finally, since $\alp^{\mu}$ and $\calT^{A}$ commute, we have $\alp \calT \psi = \calT \alp \psi$. Putting these observations together, we may write
\begin{equation*}
J_{\CSD} (\psi)= \frac{1}{2} \brk{\calT \psi, i \alp \psi} + \frac{1}{2} \brk{i \alp \psi, \calT \psi}
= \bbrk{\psi \wedge i \alp \psi}.
\end{equation*}

Motivated by the preceding computation, we define $\LieAlg$-valued differential forms $\Gmm_{\CSD}^{(k)}$, which are multilinear (over $\bbR$) in the inputs:
\begin{align*}
	\Gmm_{\CSD}^{(0)}[\psi^{1}, \psi^{2}] = & \bbrk{\psi^{1} \wedge i \alp \psi^{2}} \\
	\Gmm_{\CSD}^{(k)}[\psi^{1}, \psi^{2}; Z_{1}, \ldots, Z_{k}] = & \bbrk{\psi^{1} \wedge i (\calL_{Z_{k}} \cdots \calL_{Z_{1}}\alp) \psi^{2}}.
\end{align*}
Here $\psi^{1}, \psi^{2}$ are $V = \Dlt \otimes W$-valued functions, and each $Z_{j}$ is a vector field. 
By the above computation and definition, we have 
\begin{equation*}
 J_{\CSD}(\psi)= \Gmm_{\CSD}^{(0)}[\psi, \psi]  = \bbrk{\psi \wedge i \alp \psi}.
\end{equation*}
As in the case of \eqref{eq:CSH}, we often do not keep track of each $Z_{j}$ and simply write
\begin{equation*}
	\Gmm_{\CSD}^{(k)}[\psi^{1}, \psi^{2}] = \Gmm_{\CSD}^{(k)}[\psi^{1}, \psi^{2}; Z_{1}, \ldots, Z_{k}],
\end{equation*}
where each $Z_{j}$ is understood to be one of the Killing vector fields $Z_{\mu \nu}$.

The following analogue of Lemma~\ref{lem:comm-J-CSH-1} holds.
\begin{lemma} \label{lem:comm-J-CSD-1}
Let $A$ be a connection 1-form and $\psi$ a $V$-valued function, which satisfy the Chern--Simons equation $F = \star J_{\CSD}(\psi)$. 
Then for any $V$-valued functions $\psi^{1}, \psi^{2}$ and a Killing vector field $Z$, we have
\begin{equation} \label{eq:comm-J-CSD-1}
	\covLD_{Z} \Gmm_{\CSD}^{(0)}[\psi^{1}, \psi^{2}]
	= \Gmm_{\CSD}^{(0)}[\covZ \psi^{1}, \psi^{2}] + \Gmm_{\CSD}^{(0)}[\psi^{1}, \covZ \psi^{2}]
	+ \Gmm_{\CSD}^{(1)}[\psi^{1}, \psi^{2}]	
\end{equation}
\end{lemma}
\begin{proof} 
This identity holds thanks to the Leibniz rule \eqref{eq:leibniz-bbrk-LD} and the formula
$\covLD_Z (\alp \psi^2) = (\LD_Z \alp) \psi^2 + \alp \covLD_{Z} \psi^2$.  \qedhere
\end{proof}
As a simple consequence of the previous lemma, the following formula for $\covLD_{Z}^{(m)} J_{CSD}(\psi)$ holds.
\begin{proposition} \label{prop:comm-J-CSD}
Let $A$ be a connection 1-form and $\psi$ a $V$-valued function, which satisfy the Chern--Simons equation $F = \star J_{\CSD}(\psi)$. 
Then for any $m \geq 1$, the following schematic formula holds.
\begin{equation} \label{eq:comm-J-CSD}
	\covLD_{Z}^{(m)} J_{\CSD}(\psi)
	= \sum_{\ell=0}^{m} \bb( \sum_{k_{1} + k_{2} \leq m-\ell} \Gmm_{\CSD}^{(\ell)}[\covZ^{(k_{1})} \psi, \covZ^{(k_{2})}\psi] \bb).
\end{equation}
\end{proposition}

As in the case of \eqref{eq:CSH}, we end with a lemma that computes $\covud \Gmm_{\CSH}^{(k)}$. Combined with Proposition~\ref{prop:comm-J-CSD}, this lemma allows us to compute $\covud \covLD^{(m)} J_{\CSD}$.
\begin{lemma} \label{lem:covud-CSD}
Let $A$ be a connection 1-form. Then for any $V$-valued functions $\psi^{1}$ and $\psi^{2}$, we have
\begin{equation} \label{eq:covud-CSD}
\begin{aligned}
& \hskip-2em
	\covud \Gmm_{\CSD}^{(k)}[\psi^{1}, \psi^{2}; Z_{1}, \ldots, Z_{k}]  \\
	=& \bbrk{\covud \psi^{1} \wedge i (\calL_{Z_{k}} \cdots \calL_{Z_{1}} \alp) \psi^{2}}
	- \bbrk{\psi^{1} \wedge (i (\calL_{Z_{k}} \cdots \calL_{Z_{1}} \alp) \wedge \covud \psi^{2})}.
\end{aligned}
\end{equation}
\end{lemma}
We omit the proof, which is a straightforward application of \eqref{eq:leibniz-bbrk-d} and the Leibniz rule, combined with the fact that $\ud \alp = 0$.

\subsection{Commutation relation for $U(\phi)$}	\label{subsec:comm-U}
In this subsection, we establish the commutation properties of the $V$-valued potential $U(\phi)$.

We begin with the case of \eqref{eq:CSH}. By \eqref{eq:CSH-ptnl} and the convention $v = \kpp = 1$, $U_{\CSH}$ can be decomposed into
\begin{equation*}
	U_{\CSH}(\varphi) 
	= U_{3}[\varphi, \varphi, \varphi] + U_{5}[\varphi, \varphi, \varphi, \varphi, \varphi]
\end{equation*}
where
\begin{align*}
	U_{3}[\varphi^{1}, \varphi^{2}, \varphi^{3}]
	= & \dlt_{AA'} \brk{\calT^{A} \varphi^{1}, \varphi^{2}} \calT^{A'} \varphi^{3}, \\
	U_{5}[\varphi^{1}, \varphi^{2}, \varphi^{3}, \varphi^{4}, \varphi^{5}]
	= & \dlt_{AA'} \dlt_{BB'} \bb( \brk{\calT^{A} \varphi^{1}, \varphi^{2}} \brk{(\calT^{A'} \calT^{B'} + \calT^{B'} \calT^{A'}) \varphi^{3}, \varphi^{4}} \calT^{B} \varphi^{5} \\
	& \phantom{\dlt_{AA'} \dlt_{BB'} \bb(} + \brk{\calT^{A} \varphi^{1}, \varphi^{2}} \brk{\calT^{B} \varphi^{3}, \varphi^{4}} \calT^{A'} \calT^{B'} \varphi^{5} \bb).
\end{align*}
\begin{lemma} \label{lem:comm-U-CSH}
Let $X$ be any vector field, and $\varphi^{1}, \ldots, \varphi^{5}$ be $V$-valued functions. The multilinear forms $U_{3}$ and $U_{5}$ obey the following Leibniz rules.
\begin{align}
\covD_{X} U_{3}[\varphi^{1}, \varphi^{2}, \varphi^{3}]
=& U_{3}[\covD_{X} \varphi^{1}, \varphi^{2}, \varphi^{3}]
+ \cdots
+ U_{3}[\varphi^{1}, \varphi^{2}, \covD_{X} \varphi^{3}] \label{eq:comm-U-CSH-3} \\
\covD_{X} U_{5}[\varphi^{1}, \varphi^{2}, \varphi^{3}, \varphi^{4}, \varphi^{5}]
=& U_{5}[\covD_{X} \varphi^{1}, \varphi^{2}, \varphi^{3}, \varphi^{4}, \varphi^{5}]
+ \cdots \label{eq:comm-U-CSH-5} \\
& + U_{5}[\varphi^{1}, \varphi^{2}, \varphi^{3}, \varphi^{4}, \covD_{X} \varphi^{5}].  \notag
\end{align}
\end{lemma}

\begin{proof} 
The idea of the proof is similar to that of Lemma~\ref{lem:leibniz-bbrk}. To exemplify, we will show that
\begin{equation*}
\tilde{U}_{5}[\varphi^{1}, \varphi^{2}, \varphi^{3}, \varphi^{4}, \varphi^{5}]
= \dlt_{AA'} \dlt_{BB'} \brk{\calT^{A} \varphi^{1}, \varphi^{2}} \brk{\calT^{A'} \calT^{B'} \varphi^{3}, \varphi^{4}} \calT^{B} \varphi^{5},
\end{equation*} 
which is a part of the quintilinear form $U_{5}$, obeys the Leibniz rule 
\begin{equation} \label{eq:comm-U-CSH-ex}
\covD_{X} \tilde{U}_{5}[\varphi^{1}, \varphi^{2}, \varphi^{3}, \varphi^{4}, \varphi^{5}]
=\tilde{U}_{5}[\covD_{X} \varphi^{1}, \varphi^{2}, \varphi^{3}, \varphi^{4}, \varphi^{5}]
+ \cdots
+ \tilde{U}_{5}[\varphi^{1}, \varphi^{2}, \varphi^{3}, \varphi^{4}, \covD_{X} \varphi^{5}].
\end{equation}
The desired identities \eqref{eq:comm-U-CSH-3} and \eqref{eq:comm-U-CSH-5} may be proved by a similar argument.

As in the proof of Lemma~\ref{lem:leibniz-bbrk}, we introduce the shorthand $a = A(X)$, and fix an orthonormal basis $\set{e_{A}}$ of $\LieAlg$ so that $\calT^{A} \varphi = \dlt^{AA'} e_{A'} \cdot \varphi$. We define the structure constants $c_{AB}^{C}$ by $\LieBr{e_{A}} {e_{B}} = c_{AB}^{C} e_{C}$. The difference between the left- and right-hand sides of \eqref{eq:comm-U-CSH-ex} can then be computed as follows:
\begin{align*}
& a^{C} c_{CA}^{D} \dlt^{A A'} \dlt^{B B'} \brk{e_{D} \cdot \varphi^{1}, \varphi^{2}} \brk{e_{A'} \cdot (e_{B'} \cdot \varphi^{3}), \varphi^{4}} e_{B} \cdot \varphi^{5} \\
 & + a^{C} c_{CA'}^{D} \dlt^{A A'} \dlt^{B B'} \brk{e_{A} \cdot \varphi^{1}, \varphi^{2}} \brk{e_{D} \cdot (e_{B'} \cdot \varphi^{3}), \varphi^{4}} e_{B} \cdot \varphi^{5} \\
 & + a^{C} c_{CB'}^{D} \dlt^{A A'} \dlt^{B B'} \brk{e_{A} \cdot \varphi^{1}, \varphi^{2}} \brk{e_{A'} \cdot (e_{D} \cdot \varphi^{3}), \varphi^{4}} e_{B} \cdot \varphi^{5} \\
 & + a^{C} c_{CB}^{D} \dlt^{A A'} \dlt^{B B'} \brk{e_{A} \cdot \varphi^{1}, \varphi^{2}} \brk{e_{A'} \cdot (e_{B'} \cdot \varphi^{3}), \varphi^{4}} e_{D} \cdot \varphi^{5} 
\end{align*}
Relabeling the indices, we see that this expression vanishes, and hence \eqref{eq:comm-U-CSH-ex} follows, if 
\begin{equation*}
	c_{CD}^{A} \dlt^{DA'} 
	+ c_{CD}^{A'} \dlt^{A D}  = 0, \quad
	c_{CD}^{B'}  \dlt^{BD}
	+ c_{CD}^{B} \dlt^{DB'} = 0.
\end{equation*}
However, these are precisely \eqref{eq:str-const}. \qedhere
\end{proof}

Next we turn to the case of \eqref{eq:CSD}, where $V = \Dlt \otimes W$. 
Recall the definition of $U_{\CSD}(\psi)$ given in \eqref{eq:U-CSD}. 
%
Using the notation $\bbrk{\cdot, \cdot}$, the $V$-valued potential $U_{\CSD}(\psi)$ may be rewritten as
\begin{equation*}
U_{\CSD} (\psi) 
= \tilde{U}_{3}[\psi, \psi, \psi],
\end{equation*}
where
\begin{equation*}
	\tilde{U}_{3}[\psi^{1}, \psi^{2}, \psi^{3}]
	= \frac{1}{2} \eps(T_{\mu}, T_{\nu}, T_{\lmb}) \gmm^{\mu} \gmm^{\nu} \bbrk{\psi^{1}, i \alp^{\lmb} \psi^{2}} \psi^{3}.
\end{equation*}

\begin{lemma} \label{lem:comm-U-CSD}
Let $X$ be any vector field, and $\psi^{1}, \ldots, \psi^{3}$ be $V$-valued functions. The multilinear form $\tilde{U}_{3}$ obeys the following Leibniz rule.
\begin{align}
\covD_{X} \tilde{U}_{3}[\psi^{1}, \psi^{2}, \psi^{3}]
=& \tilde{U}_{3}[\covD_{X} \psi^{1}, \psi^{2}, \psi^{3}]
+ \cdots
+ \tilde{U}_{3}[\psi^{1}, \psi^{2}, \covD_{X} \psi^{3}]. \label{eq:comm-U-CSD} 
\end{align}
\end{lemma}
\begin{proof} 
Since $\eps(T_{\mu}, T_{\nu}, T_{\lmb})$, $\gmm^{\mu}$ and $\alp^{\mu}$ are constant, the desired conclusion follows from Lemmas~\ref{lem:leibniz-gV} and \ref{lem:leibniz-bbrk}. \qedhere
\end{proof}

\subsection{Pointwise bounds} \label{subsec:ptwise}
Let $\Omg$ be a $V$-, $\LieAlg$- or real-valued differential form. Recall that the norm $\abs{\Omg}$ is defined by the formula
\begin{align*}
	\abs{\Omg}^{2} = \sum_{\mu_{1} < \cdots < \mu_{k}} \abs{\Omg(T_{\mu_{1}}, \ldots, T_{\mu_{k}})}^{2}.
\end{align*}
where we use $\brk{\cdot, \cdot}$ [resp. $\brk{\cdot, \cdot}_{\LieAlg}$] on the right-hand side when $\Omg$ is $V$- [resp. $\LieAlg$-] valued.
The following bounds are obvious, yet useful:
\begin{equation} \label{eq:ptwise-easy}
\begin{aligned}
	\abs{\star \Omg} 
	\leq & \abs{\Omg}, \\
	\abs{\iota_{T_{\mu}} \Omg}
	\leq & \abs{\Omg}, \\
	\abs{\iota_{Z_{\mu \nu}} \Omg}
	\leq &\hT \cosh \hY \abs{\Omg} ,\\
	\abs{\iota_{S} \Omg}
	\leq & \hT \cosh \hY \abs{\Omg}, \\
	\abs{\iota_{N} \Omg}
	\leq & \cosh \hY \abs{\Omg}.
\end{aligned}
\end{equation}
Let $\Omg^{1}, \Omg^{2}$ be $V$-, $\LieAlg$- or real-valued forms for which the wedge product $\Omg^{1} \wedge \Omg^{2}$ can be defined as in Section~\ref{subsec:extr-calc-2} (e.g., $\Omg^{1}$ is $\LieAlg$-valued and $\Omg^{2}$ is $V$-valued). Then we have the inequality
\begin{align*}
	\abs{\Omg^{1} \wedge \Omg^{2}} \leq C \abs{\Omg^{1}} \abs{\Omg^{2}}.
\end{align*}
Similarly, for a $\LieAlg$-valued form $a$ and a $V$-valued form $v$, we have
\begin{align*}
	\abs{\bbrk{a \wedge v}} \leq C \abs{a} \abs{v}.
\end{align*}

Recall the multilinear expressions $\frkN_{1}, \ldots, \frkN_{4}$, which were defined in Section~\ref{subsec:comm-covBox} to facilitate the computation of $[\covZ^{(m)}, \covBox \m 1]$. The following pointwise bounds hold for these expressions.
\begin{lemma} \label{lem:ptwise-N}
Let $\Gmm$, $\Gmm^{j}$ $(j=1,2)$ be $\LieAlg$-valued 1-forms and $\phi$ be a $V$-valued function. Then we have
\begin{align}
	\abs{\frkN_{1}[\Gmm, \phi]}
	\leq & C  \hT \abs{\iota_{N} \covud \Gmm} \abs{\phi} \label{eq:ptwise-N1} \\
	\leq & C \hT \cosh \hY \abs{\covud \Gmm} \abs{\phi},  \label{eq:ptwise-N1-easy} \\
	\abs{\frkN_{2}[\Gmm, \phi]} 
	\leq & C \hT \bb( \abs{\Gmm} \abs{\covN \phi} + \abs{\iota_{N} \Gmm} \abs{\covT \phi} \bb)\label{eq:ptwise-N2}  \\
	\leq & C \hT \cosh \hY \abs{\Gmm} \abs{\covT \phi}, \label{eq:ptwise-N2-easy} \\
	\abs{\frkN_{3}[\Gmm, \phi]}
	\leq & C \abs{\Gmm} \abs{\phi}, \label{eq:ptwise-N3}\\
	\abs{\frkN_{4}[\Gmm^{1}, \Gmm^{2}, \phi]}
	\leq & C \hT^{2} \cosh^{2} \hY \abs{\Gmm^{1}} \abs{\Gmm^{2}} \abs{\phi}. \label{eq:ptwise-N4}
\end{align}
\end{lemma}
\begin{proof}
Recalling the definitions \eqref{eq:frkN-1} and \eqref{eq:frkN-2} of $\frkN_{1}$ and $\frkN_{2}$, we see that we need to bound $\abs{\iota_{Z} \star \Omg}$ where
$\Omg = \covud \Gmm \wedge \phi$ or $\Gmm \wedge \covud \phi$. In rectilinear coordinates, it can be checked that 
\begin{equation*}
	x_{\mu} \eps_{\nu \kpp \lmb} - x_{\nu} \eps_{\mu \kpp \lmb}
	= \eps_{\mu \nu \kpp} x_{\lmb} - \eps_{\mu \nu \lmb} x_{\kpp}.
\end{equation*} 
Let $\Omg$ be a ($V$-, $\LieAlg$- or real-valued) 2-form $\Omg$. By the preceding identity, we have
\begin{equation*}
	\iota_{Z_{\mu \nu}} \star \Omg = 2 \hT \sum_{\kpp} \eps_{\mu \nu \kpp} \iota_{N} \iota_{T_{\kpp}} \Omg.
\end{equation*}
Recalling the definition of $\abs{\Omg}$, we see that
\begin{equation*}
	\iota_{Z_{\mu \nu}} \star \Omg \leq C \hT \abs{\iota_{N} \Omg}.
\end{equation*}
The bounds \eqref{eq:ptwise-N1} and \eqref{eq:ptwise-N2} now follow, using the Leibniz rule for $\iota_{N}$. 

The inequalities \eqref{eq:ptwise-N1-easy} and \eqref{eq:ptwise-N2-easy} are immediate consequences of \eqref{eq:ptwise-easy}. Moreover, \eqref{eq:ptwise-N3} and \eqref{eq:ptwise-N4} are straightforward to establish; hence we omit their proofs. \qedhere
\end{proof}

For $\Gmm_{\CSH}^{(k)}$ [resp. $\Gmm_{\CSD}^{(k)}$], which was defined in Section~\ref{subsec:comm-CSH} and arise in the computation of $\covLD_{Z}^{(m)} J_{\CSH}$ [resp. $\covLD_{Z}^{(m)} J_{\CSD}$], the following pointwise bounds hold. 
\begin{lemma} \label{lem:ptwise-CSH}
Let $\phi^{j}$ $(j=1,2,\ldots)$ be $V$-valued functions. Then for any integer $k = 0, 1, \ldots$, we have
\begin{align}
	\abs{\Gmm_{\CSH}^{(k)}[\phi^{1}, \phi^{2}, \ldots, \phi^{2k+2}]} 
	\leq & C_{k} (\hT \cosh \hY)^{k} \abs{\phi^{1}} \cdots \abs{\phi^{2k+1}} \abs{\covT \phi^{2k+2}}. \label{eq:ptwise-CSH}
\end{align}
In the case $k = 0$, for any vector $X$ we have the following refined bound:
\begin{equation} \label{eq:ptwise-CSH:X}
	\abs{\iota_{X} \Gmm_{\CSH}^{(0)}[\phi^{1}, \phi^{2}]} 
	\leq C \abs{\phi^{1}} \abs{\covD_{X} \phi^{2}}.
\end{equation}
\end{lemma}

\begin{lemma} \label{lem:ptwise-CSD}
Let $\phi^{1}, \phi^{2}$ be $V$-valued functions. Then for any integer $k = 0, 1, \ldots$, we have
\begin{align}
	\abs{\Gmm_{\CSD}^{(k)}[\phi^{1}, \phi^{2}]} 
	\leq & C_{k} \abs{\phi^{1}} \abs{\phi^{2}}. \label{eq:ptwise-CSD}
\end{align}
\end{lemma}

Finally, the following pointwise bounds for the multilinear forms $U_{3}$, $U_{5}$ and $\tilde{U}_{3}$ (defined in Section~\ref{subsec:comm-U}) hold.

\begin{lemma} \label{lem:ptwise-U}
Let $\phi^{j}$ $(j=1, \ldots, 5)$ be $V$-valued functions. Then we have
\begin{align}
	\abs{U_{3}[\phi^{1}, \phi^{2}, \phi^{3}]} \leq & C \abs{\phi^{1}} \abs{\phi^{2}} \abs{\phi^{3}}, \label{eq:ptwise-U-CSH-3} \\
	\abs{U_{5}[\phi^{1}, \phi^{2}, \phi^{3}, \phi^{4}, \phi^{5}]} \leq & C \abs{\phi^{1}} \abs{\phi^{2}} \abs{\phi^{3}} \abs{\phi^{4}} \abs{\phi^{5}}, \label{eq:ptwise-U-CSH-5} \\
	\abs{\tilde{U}_{3}[\phi^{1}, \phi^{2}, \phi^{3}]} \leq & C \abs{\phi^{1}} \abs{\phi^{2}} \abs{\phi^{3}}. \label{eq:ptwise-U-CSD}
\end{align}
\end{lemma}

We omit the straightforward proofs of the preceding lemmas.

\section{Proof of the main a priori estimates} \label{sec:BA}
In this section, we carry out the proof of the main a priori estimates (Proposition~\ref{prop:main}). In Section~\ref{subsec:BAs}, we reduce the proof of Proposition~\ref{prop:main} to a bootstrap argument. In particular, we list the bootstrap assumptions, and introduce a few conventions that will simplify the further presentation. In the remainder of the section, we show that the bootstrap assumptions can be improved provided that $\dlt_{\star}(R)$ in the hypothesis of Proposition~\ref{prop:main} is chosen sufficiently small.

\subsection{Reduction to a bootstrap argument} \label{subsec:BAs}
Throughout this section, we assume that $(A, \phi)$ is a solution to \eqref{eq:CS-uni} satisfying the hypotheses of Proposition~\ref{prop:main}, where $\dlt_{\ast} = \dlt_{\ast}(R)$ is to be specified below. 

We begin with a bound on the initial hyperboloid $\calH_{2R}$, which is the starting point of the proof of Proposition~\ref{prop:main}.
\begin{lemma} \label{lem:BA:ini}
If $\eps$ is sufficiently small depending on $R$, then we have
\begin{equation} \label{eq:BA:ini}
	\sum_{k =0}^{4} \bb( \int_{\calH_{2R}} \ed_{\calH_{2R}}[\covZ^{(m)} \phi] \, \ud \Vol_{\calH_{2R}}\bb)^{1/2}
	\leq C(R) \eps,
\end{equation}
where the energy density $\ed_{\calH_{2R}}$ was defined in \eqref{eq:ed}.
\end{lemma}

This lemma is a consequence of \eqref{eq:initial-hyp:est}, Lemma~\ref{lem:ini-en} and the formula for the commutator $[\covZ^{(m)}, \covBox \m 1]$ derived in Section~\ref{sec:comm}. Observe that $\abs{\covZ^{(m)} \phi} \leq C(m, R) \sum_{k=0}^{m} \abs{\covT^{(k)} \phi}$ in the region $\calR_{t=2R}^{\hT = 2R}$ defined in \eqref{eq:ini-en:region}, thanks to the support property \eqref{eq:initial-hyp:fsp}. This observation allows us to use \eqref{eq:initial-hyp:est} (combined with the Sobolev inequality) to bound the error terms in Lemma~\ref{lem:ini-en}. We omit further details.

Next, we turn to the statement of the central bootstrap assumptions. In order to proceed, we introduce the following notation: By $\hT^{\alp +}$ [resp. $\hT^{\alp-}$] for some $\alp \in \bbR$, we mean $\hT^{\alp + \dlt}$ [resp. $\hT^{\alp - \dlt}$] for a fixed absolute constant $0 < \dlt \ll 1$. 

For $2R \leq \hT \leq T'$ (where $T' \leq T$), the following \emph{bootstrap assumptions} will be made:
\begin{itemize}
\item {\bf $L^{2}$ bounds with growth.} For $0 \leq m \leq 4$,
\begin{equation} \label{eq:BA:L2}
	\wnrm{\cosh \hY \covZ^{(m)} \phi}_{L^{2}_{\hT}}
	+ \wnrm{\covT \covZ^{(m)} \phi}_{L^{2}_{\hT}}
	\leq 10 \eps_{1}  \log^{m} (1+\hT) .
\end{equation}
For $0 \leq m \leq 3$,
\begin{equation} \label{eq:BA:L2:S}
\wnrm{\cosh \hY \covZ^{(m)} \covN \phi}_{L^{2}_{\hT}} 
\leq 10 \eps_{1}  \log^{m+1} (1+\hT) .
\end{equation}
\item {\bf Sharp $L^{\infty}$ decay.}
\begin{equation} 	\label{eq:BA:Linfty}
	\wnrm{\cosh \hY \phi}_{L^{\infty}_{\hT}} 
	+ \wnrm{\cosh \hY \covN \phi}_{L^{\infty}_{\hT}} 
	+ \wnrm{\covT \phi}_{L^{\infty}_{\hT}} 
	\leq  10 \eps_{1} \hT^{-1}.
%
\end{equation}
\item {\bf Nonlinearity estimates.} 
For $0 \leq m \leq 4$,
\begin{equation} \label{eq:BA:KG}
	\wnrm{\cosh \hY (\covBox \m 1) \covZ^{(m)} \phi}_{L^{2}_{\hT}} 
	\leq 10 \eps_{1}^{3-} \hT^{-1} \log^{m-1} (1+\hT).
%
\end{equation}
\end{itemize}

Here, $\eps_{1} = B_{0} \eps$ and $B_{0}$ is a large absolute constant to be chosen later; recall from the hypothesis of Proposition~\ref{prop:main} that $\eps_{1} = B_{0} \eps \leq B_{0} \dlt_{\ast}$. 

From Lemma~\ref{lem:BA:ini} and a simple computation involving the Klainerman--Sobolev inequality (Proposition~\ref{prop:KlSob}) and the formula for $[\covZ^{(m)}, \covBox \m 1]$ in Section~\ref{subsec:comm-covBox}, it follows that \eqref{eq:BA:L2}--\eqref{eq:BA:KG} hold on the initial hyperboloid $\hT = 2R$ \emph{without} the factor of 10 on the right-hand side if $B_{0}$ is chosen sufficiently large. In the rest of Section~\ref{sec:BA}, our goal is to show that if $B_{0}$ is large enough and $\dlt_{\ast}$ is sufficiently small, then the above bootstrap assumptions may be improved, in the sense that \eqref{eq:BA:L2}--\eqref{eq:BA:KG} hold for $\hT \in [2R, T']$ \emph{without} the factor of 10 on the right-hand side. By a routine continuity argument in $T'$, we may conclude that \eqref{eq:BA:L2}--\eqref{eq:BA:KG} hold for all $\hT \in [2R, T]$; Proposition~\ref{prop:main} would then follow.

We end with a few conventions that will be in effect for the rest of this section. First, in view of the fact that $\dlt_{\ast}$ would be chosen very small at the end, {\bf we assume that $0 < \eps_{1} \leq 1$}. Second, unless otherwise stated, {\bf all the estimates are for $\hT \in [2R, T']$.} Finally, since $R$ is fixed, {\bf we suppress specifying dependence of constants on $R$.} 

\subsection{Consequences of the bootstrap assumptions} \label{subsec:BA-conseq}
Henceforth, our goal is to improve the bootstrap assumptions \eqref{eq:BA:L2}--\eqref{eq:BA:KG}.
We begin by deriving some quick consequences of the bootstrap assumptions in Section~\ref{subsec:BAs}. We start with some decay estimates for $\covZ^{(m)} \phi$ and $\covZ^{(m)} \covN \phi$, which follow from the Klainerman--Sobolev inequality (Proposition~\ref{prop:KlSob}).

\begin{lemma} \label{lem:BA:KlSob}
Suppose that $(A, \phi)$ satisfies the bootstrap assumptions \eqref{eq:BA:L2} and \eqref{eq:BA:L2:S}.
Then for $0 \leq m \leq 2$, we have
\begin{equation} \label{eq:weakLinfty}
	\wnrm{\cosh \hY \covZ^{(m)} \phi}_{L^{\infty}_{\hT}}
	+ \wnrm{\cosh \hY \covZ^{(m-1)} \covN \phi}_{L^{\infty}_{\hT}}
	\leq C \eps_{1} \, \hT^{-1} \log^{m+2} (1+\hT).
\end{equation}
where the last term on the left-hand side should be omitted in the case $m = 0$.

For $m = 3$, we have
\begin{equation} \label{eq:weakLp}
	\wnrm{\cosh \hY \covZ^{(3)} \phi}_{L^{4}_{\hT}}
	+ \wnrm{\cosh \hY \covZ^{(2)} \covN \phi}_{L^{4}_{\hT}}
	\leq C \eps_{1} \, \hT^{-1+\frac{2}{p}} \log^{4} (1+\hT).
\end{equation}
\end{lemma}
\begin{proof} 
The inequality \eqref{eq:weakLinfty} follows from \eqref{eq:BA:L2}, \eqref{eq:BA:L2:S} and the Klainerman--Sobolev inequality (Proposition~\ref{prop:KlSob}). Then \eqref{eq:weakLp} follows by application of Lemma~\ref{lem:Z}. \qedhere
\end{proof}

Our argument below requires bounds for $\covN \covZ^{(m)} \phi$. The following estimate for the commutator $[\covZ^{(m)}, \covN] \phi$ may be used to show that such estimates follow from the corresponding bounds for $\covZ^{(m)} \covN \phi$.
\begin{lemma} \label{lem:BA:comm-NZ} 
Suppose that $(A, \phi)$ satisfies the bootstrap assumptions \eqref{eq:BA:L2} and \eqref{eq:BA:L2:S}. 
Then for $1 \leq m \leq 3$ and $1 \leq p \leq \infty$, we have
\begin{equation} \label{eq:comm-NZ}
\begin{aligned}
& \hskip-2em
	\wnrm{\cosh \hY [\covZ^{(m)}, \covN] \phi}_{L^{p}_{\hT}} \\
	\leq & C \eps_{1}^{2} \hT^{-1+} \sum_{k \leq m-1} 
		\bb( \wnrm{\cosh \hY \covZ^{(k)} \phi}_{L^{p}_{\hT}} 
			+ \wnrm{\cosh \hY \covN \covZ^{(k)} \phi}_{L^{p}_{\hT}}
			+ \wnrm{\covT \covZ^{(k)} \phi}_{L^{p}_{\hT}} \bb).
\end{aligned}
\end{equation}
\end{lemma}

In order to prove Lemma~\ref{lem:BA:comm-NZ}, it is convenient to establish certain bounds for $\covN \covZ^{(m)} \phi$ and $\covT \covZ^{(m)} \phi$ simultaneously, since these expressions arise in the commutator $[\covN, \covZ^{(m)}] \phi$. We record these bounds as a lemma:

\begin{lemma} \label{lem:BA:NZ}
Suppose that $(A, \phi)$ satisfies the bootstrap assumptions \eqref{eq:BA:L2} and \eqref{eq:BA:L2:S}. For $0 \leq m \leq 3$, we have
\begin{equation}  \label{eq:NZ-L2}
	\wnrm{\cosh \hY \covN \covZ^{(m)} \phi}_{L^{2}_{\hT}}  
	\leq C \eps_{1} \log^{m}(1+\hT). 
\end{equation}
Moreover, the following $L^{p}$ estimates also hold:
\begin{align} 
	\wnrm{\cosh \hY \covN \covZ^{(m)} \phi}_{L^{\infty}_{\hT}} + \wnrm{\covT \covZ^{(m)} \phi}_{L^{\infty}_{\hT}} 
	\leq & C \eps_{1} \hT^{-1} \log^{m+2}(1+\hT) \quad \hbox{ for } 0 \leq m \leq 1, \label{eq:NZ-Linfty} \\
	\wnrm{\cosh \hY \covN \covZ^{(m)} \phi}_{L^{4}_{\hT}} + \wnrm{\covT \covZ^{(m)} \phi}_{L^{4}_{\hT}}
	\leq & C \eps_{1} \hT^{-\frac{1}{2}} \log^{4}(1+\hT) \quad \hbox{ for } m = 2. \label{eq:NZ-L4}
\end{align}
\end{lemma}

\begin{proof} [Proof of Lemmas~\ref{lem:BA:comm-NZ} and \ref{lem:BA:NZ}]
We begin with a simple observation: Once \eqref{eq:comm-NZ} is established in the range $1 \leq m \leq m_{0}$ for some $m_{0} \leq 3$, then \eqref{eq:NZ-L2}, \eqref{eq:NZ-Linfty} and \eqref{eq:NZ-L4} for the same range of $m$ follow. Indeed, proceeding inductively, we may assume that \eqref{eq:comm-NZ}, \eqref{eq:NZ-L2}, \eqref{eq:NZ-Linfty} and \eqref{eq:NZ-L4} hold for up to some $m-1$, where $0 \leq m \leq m_{0}$ (in the case $m = 0$, we make no induction hypothesis). Then the claim for $m$ is obvious for \eqref{eq:NZ-L2} and the term involving $\covN \covZ^{(m)} \phi$ in \eqref{eq:NZ-Linfty} and \eqref{eq:NZ-L4}. To bound $\covT \covZ^{(m)}$, we use the pointwise inequality
\begin{equation} \label{eq:T-N-Z}
	\abs{\covT \tilde{\phi}} \leq C \cosh \hY ( \abs{\covN \tilde{\phi}} + \hT^{-1} \abs{\covZ \tilde{\phi}})
\end{equation}
with $\tilde{\phi} = \covZ^{(m)} \phi$. By the preceding discussion, the first term on the right-hand side is acceptable. On the other hand, the last term obeys (thanks to Lemma~\ref{lem:BA:KlSob}) a far better decay rate than needed:
\begin{equation} \label{eq:Z-lower-order}
	\wnrm{\cosh \hY \hT^{-1} \covZ^{(m+1)} \phi}_{L^{p}_{\hT}} \leq C \eps_{1} \hT^{-2 + \frac{2}{p}} \log^{4}(1+\hT) ,
\end{equation}
for any $1 \leq p \leq \infty$ when $0 \leq m \leq 1$ and $2 \leq p \leq 4$ when $m = 2$.

Consequently, not only would Lemma~\ref{lem:BA:NZ} follow once we prove \eqref{eq:comm-NZ} for all $m$, but we are also allowed to employ the bounds in Lemma~\ref{lem:BA:NZ} for $m$ for which \eqref{eq:comm-NZ} has already been proved. As discussed above, this will be useful because expressions of the form $\covN \covZ^{(k)} \phi$ and $\covT \covZ^{(k)} \phi$ with $k \leq m-1$ appear in the commutator $[\covN, \covZ^{(m)}] \phi$; see, in particular, the case of \eqref{eq:CSH} below.

Our next task is to compute $[\covN, \covZ^{(m)}] \phi$. The commutator between $\covZ_{\mu \nu}$ and $\covS = \hT \covN$ is 
\begin{align*}
	[\covS, \covZ_{\mu \nu}] \phi
=&	- (\iota_{S} \iota_{Z_{\mu \nu}} \star J )\phi,
\end{align*}
since $S$ and $Z_{\mu \nu}$ commute. From this computation, \eqref{eq:covLDiX}, \eqref{eq:star-LDZ} and the Lie algebra relations among $\set{S, Z_{\mu \nu}}$, we may derive the following schematic commutation formula for $[\covZ^{(m)}, \covN] \phi$:
\begin{align*}
	[\covN, \covZ^{(m)}] \phi
=&	\hT^{-1} \sum_{k_{1} + k_{2} \leq m-1} (\iota_{S} \iota_{Z} \star \covLD_{Z}^{(k_{1})} J) \covZ^{(k_{2})} \phi \quad \hbox{ for } m \geq 1.
\end{align*}
Given a ($V$-, $\LieAlg$- or real-valued) $1$-form $\Gmm$, we compute using the rectilinear coordinates $(t = x^{0}, x^{1}, x^{2})$
\begin{align*}
\iota_{S} \iota_{Z_{\mu \nu}} (\star \Gmm)
=& (\met^{-1})^{\alp \bt} (x^{\kpp} x_{\mu} \eps_{\nu \kpp \alp} - x^{\kpp} x_{\nu} \eps_{\mu \kpp \alp}) \Gmm(T_{\bt}) \\
=& (\met^{-1})^{\alp \bt} \eps_{\mu \nu \kpp} x^{\kpp} x_{\alp} \Gmm(T_{\bt})
	- (\met^{-1})^{\alp \bt} \eps_{\mu \nu \alp} x_{\kpp} x^{\kpp} \Gmm(T_{\bt}).
\end{align*}
Observe furthermore that $(\eta^{-1})^{\alp \bt} x_{\alp} T_{\bt} = x^{\bt} T_{\bt = }  \hT N$ and $x_{\kpp} x^{\kpp} = - \hT^{2}$.
Hence 
\begin{equation}
	\abs{\iota_{S} \iota_{Z} \star \Gmm}
	\leq C ( \hT^{2} \cosh \hY \abs{\iota_{N} \Gmm} + \hT^{2} \abs{\Gmm}).
\end{equation}
We therefore arrive at the pointwise bound
\begin{equation} \label{eq:key-comm-NZ}
\begin{aligned}
\cosh \hY \abs{[\covN, \covZ^{(m)}] \phi} 
\leq C \hT \cosh \hY \sum_{k_{1} + k_{2} \leq m-1} \bb(\cosh \hY \abs{\iota_{N}\covLD_{Z}^{(k_{1})} J} + \abs{\covLD_{Z}^{(k_{1})} J} \bb) \abs{\covZ^{(k_{2})} \phi}. 
\end{aligned}
\end{equation}

In order to proceed, we divide into two cases: $J = J_{\CSH}$ and $J = J_{\CSD}$. 
\paragraph*{\bfseries - Case 1: $J = J_{\CSH}$}
When $m = 1$, by \eqref{eq:ptwise-CSH:X} and \eqref{eq:key-comm-NZ}, we may estimate
\begin{align*}
	\cosh \hY  \abs{[\covN, \covZ] \phi} 
	\leq & C \hT \cosh \hY \bb( \abs{\iota_{N} \Gmm_{\CSH}^{(0)}[\phi, \phi]} + \abs{\Gmm_{\CSH}^{(0)}[\phi, \phi]} \bb) \abs{\phi} \\
	\leq & C \hT \cosh \hY \abs{\phi}^{2} (\cosh \hY \abs{\covN \phi} + \abs{\covT \phi} ).
\end{align*}
We now take the $\wnrm{\cdot}_{L^{p}_{\hT}}$ norm and apply H\"older's inequality, where we estimate $\cosh \hY \covN \phi$ and $\covT \phi$ using $\wnrm{\cdot}_{L^{p}_{\hT}}$ and $\cosh \hY \abs{\phi}^{2}$ using $\wnrm{\cdot}_{L^{\infty}_{\hT}}$. By Lemma~\ref{lem:BA:KlSob}, the desired inequality \eqref{eq:comm-NZ} follows. As a dividend (see the remarks at the beginning of the proof), we may now use the bounds in Lemma~\ref{lem:BA:NZ} up to $m = 1$.

For $m \geq 2$, using Proposition~\ref{prop:comm-J-CSH} and \eqref{eq:key-comm-NZ}, $\cosh \hY  \abs{[\covN, \covZ] \phi}$ is bounded from the above by
\begin{equation} \label{eq:comm-NZ:CSH}
\begin{aligned}
	& C_{m} \sum_{j_{1} + j_{2} + k_{2} \leq m-1} \hT \cosh \hY  \bb( \cosh \hY \abs{\iota_{N} \Gmm_{\CSH}^{(0)} [\covZ^{(j_{1})} \phi, \covZ^{(j_{2})} \phi]} + \abs{\Gmm_{\CSH}^{(0)} [\covZ^{(j_{1})} \phi, \covZ^{(j_{2})} \phi] } \bb) \abs{\covZ^{(k_{2})} \phi} \\
	& + C_{m} \sum_{\substack{k_{1}, k_{2}, \ell : \\ k_{1} + k_{2} \leq m-1 \\ 1 \leq \ell \leq k_{1}}} 
		\sum_{j_{1} + \cdots + j_{2 \ell + 2} \leq k_{1} - \ell} 
		\hT \cosh^{3} \hY 
		\abs{\Gmm_{\CSH}^{(\ell)}[\covZ^{(j_{1})} \phi, \ldots, \covZ^{(j_{2 \ell + 2})} \phi]} \abs{\covZ^{(k_{2})} \phi}, 
\end{aligned}
\end{equation}
where we used the crude bound 
\begin{equation} \label{eq:crude-Gmm-N}
\cosh \hY \abs{\iota_{N} \Gmm} + \abs{\Gmm} \leq C \cosh^{2} \hY \abs{\Gmm}
\end{equation}
on the second line\footnote{Such a simplifying procedure is possible in this case since there is enough factors of $\covZ^{(j)} \phi$ to absorb the extra weights of $\cosh \hY$; see the discussion following \eqref{eq:comm-NZ:CSH-nonmain}.}. In what follows, we treat the sum on each line separately.

Consider a summand from the first line of \eqref{eq:comm-NZ:CSH}, which can be bounded using \eqref{eq:ptwise-CSH:X} by
\begin{equation} \label{eq:comm-NZ:CSH-main}
C_{m} \hT \cosh \hY  \abs{\covZ^{(j_{1})} \phi} ( \cosh \hY \abs{\covN \covZ^{(j_{2})} \phi} + \abs{\covT \covZ^{(j_{2})} \phi}\bb) \abs{\covZ^{(k_{2})} \phi}.
\end{equation}
Taking the $\wnrm{\cdot}_{L^{p}_{\hT}}$ norm, we apply H\"older's inequality and bound the highest order factor with an appropriate weight of $\cosh \hY$ (i.e., either $\cosh \hY \covZ^{(k)} \phi$, $\cosh \hY \covN \covZ^{(k)} \phi$ or $\covT \covZ^{(k)} \phi$) using $\wnrm{\cdot}_{L^{p}_{\hT}}$ and the rest using $\wnrm{\cdot}_{L^{\infty}_{\hT}}$. Since we are only considering $m = 2, 3$, the non-highest order cannot exceed $1$; hence the non-highest order factors obey a pointwise upper bound $C \hT^{-1} \cosh^{-1} \hY \log^{3}(1+\hT)$ by Lemma~\ref{lem:BA:KlSob} and \eqref{eq:NZ-Linfty} for $0 \leq m \leq 1$.  From such a consideration, the contribution of \eqref{eq:comm-NZ:CSH-main} is easily seen to be acceptable for the proof of \eqref{eq:comm-NZ}.

Next, consider a summand from the second line of \eqref{eq:comm-NZ:CSH}, which can be bounded by
\begin{equation} \label{eq:comm-NZ:CSH-nonmain}
C_{m} \hT^{\ell+1} \cosh^{\ell+2} \hY \abs{\covZ^{(j_{1})} \phi} \cdots \abs{\covZ^{(j_{2 \ell + 1})} \phi} \abs{\covT \covZ^{(j_{2 \ell + 2}) } \phi} \abs{\covZ^{(k_{2})} \phi},
\end{equation}
using \eqref{eq:ptwise-CSH}. As before, we take the $\wnrm{\cdot}_{L^{p}_{\hT}}$ norm and apply H\"older's inequality to bound the highest order factor (with an appropriate weight of $\cosh \hY$) using $\wnrm{\cdot}_{L^{p}_{\hT}}$ and the rest using $\wnrm{\cdot}_{L^{\infty}_{\hT}}$. As in the preceding case, the non-highest order cannot exceed $1$, so the non-highest order factors are bounded from the above by $C \hT^{-1} \cosh^{-1} \hY \log^{3}(1+\hT)$. Since there are $2\ell+2$ such factors (where $\ell \geq 1$), it can be readily checked that the contribution of \eqref{eq:comm-NZ:CSH-nonmain} is acceptable as well.

\paragraph*{\bfseries - Case 2: $J = J_{\CSD}$}
In this case, we immediately apply \eqref{eq:key-comm-NZ} and the crude bound \eqref{eq:crude-Gmm-N} to estimate
\begin{equation*}
	\cosh \hY \abs{[\covN, \covZ^{(m)}] \phi}
	\leq C_{m} \hT \cosh^{3} \hY \sum_{k_{1} + k_{2} \leq m-1} \abs{\covLD^{(k_{1})}_{Z} J_{\CSD}(\phi)} \abs{\covZ^{(k_{2})} \phi}.
\end{equation*}
By Proposition~\ref{prop:comm-J-CSD} and Lemma~\ref{lem:ptwise-CSD}, it follows that
\begin{equation} \label{eq:comm-NZ:CSD}
	\cosh \hY \abs{[\covN, \covZ^{(m)}] \phi}
	\leq C_{m} \sum_{j_{1} + j_{2} + k_{2} \leq m-1} \hT \cosh^{3} \hY \abs{\covZ^{(j_{1})} \phi} \abs{\covZ^{(j_{2})} \phi} \abs{\covZ^{(k_{2})} \phi}.
\end{equation}
From this pointwise bound, which is far simpler than the case of \eqref{eq:CSH}, it is straightforward to prove \eqref{eq:comm-NZ} (hence Lemma~\ref{lem:BA:NZ} as well); we omit the details.
\end{proof}

The bounds derived so far are insufficient to close the bootstrap; in particular, they cannot be used to bound the commutator $[\covZ^{(m)}, \covBox \m 1] \phi$, because they grows too fast in $\hT$. The cost we incur is (at least) $\log^{2} (1+\hT)$, which arises from the loss of two $\covZ$ derivatives in application of the Klainerman--Sobolev inequality. To estimate the commutator $[\covZ^{(m)}, \covBox \m 1] \phi$, we use the sharp $L^{p}$ bounds established in the following lemma, which are proved by essentially interpolating between the $L^{2}$ bounds \eqref{eq:BA:L2}, \eqref{eq:BA:L2:S} and the sharp $L^{\infty}$ decay \eqref{eq:BA:Linfty}.
\begin{lemma} \label{lem:sharpLp}
Suppose that $(A, \phi)$ satisfies the bootstrap assumptions \eqref{eq:BA:L2}, \eqref{eq:BA:L2:S} and \eqref{eq:BA:Linfty}.
If $\eps_{1} > 0$ is sufficiently small, then the following inequalities hold: For $0 \leq m \leq 2$, we have
\begin{align} \label{eq:sharpL3}
	\wnrm{\cosh \hY \covZ^{(m)} \phi}_{L^{3}_{\hT}}
	+ \wnrm{\cosh \hY \, \covN \covZ^{(m)} \phi}_{L^{3}_{\hT}}
	+ \wnrm{\covT \covZ^{(m)} \phi}_{L^{3}_{\hT}}
	\leq & C \eps_{1} \, \hT^{-\frac{1}{3}} \log^{m} (1+\hT).
\end{align}
Moreover, for $0 \leq m \leq 1$, we have
\begin{align} 
	\wnrm{\cosh \hY \covZ^{(m)} \phi}_{L^{4}_{\hT}}
	+\wnrm{\cosh \hY \, \covN \covZ^{(m)} \phi}_{L^{4}_{\hT}}
	+ \wnrm{\covT \covZ^{(m)} \phi}_{L^{4}_{\hT}}
	\leq & C \eps_{1} \, \hT^{-\frac{1}{2}} \log^{m} (1+\hT) \label{eq:sharpL4}\\
	\wnrm{\cosh \hY \covZ^{(m)} \phi}_{L^{6}_{\hT}}
	+ \wnrm{\cosh \hY \, \covN \covZ^{(m)} \phi}_{L^{6}_{\hT}}
	+ \wnrm{\covT \covZ^{(m)} \phi}_{L^{6}_{\hT}}
	\leq & C \eps_{1} \, \hT^{-\frac{2}{3}} \log^{m} (1+\hT). \label{eq:sharpL6}
\end{align}
\end{lemma}
\begin{remark} 
Heuristically, the bootstrap assumptions \eqref{eq:BA:L2}, \eqref{eq:BA:L2:S}, \eqref{eq:BA:Linfty} and the sharp $L^{p}$ bounds \eqref{eq:sharpL3}, \eqref{eq:sharpL4}, \eqref{eq:sharpL6} can be conveniently summarized as follows: 
\begin{itemize}
\item $\wnrm{\cosh \hY \, \phi}_{L^{p}_{\hT}} \leq \hT^{-1+\frac{2}{p}}$ for $2 \leq p \leq \infty$;
\item every $\covZ$ costs $\log (1+\hT)$;
\item $\covN \covZ^{(m)} \phi$ and $(\cosh \hY)^{-1} \covT \covZ^{(m)} \phi$ obey the same bounds as $\covZ^{(m)} \phi$.
\end{itemize}
In reality, we only establish such bounds for certain exponents $p$ and $m$, but these suffice for our application below. 
\end{remark}

\begin{proof} 
First, by applying \eqref{eq:GN} to $\phi$ with \eqref{eq:BA:L2} and \eqref{eq:BA:Linfty}, the estimates \eqref{eq:sharpL3}, \eqref{eq:sharpL4} and \eqref{eq:sharpL6} for $\covZ^{(m)} \phi$ follow immediately. 
Next, applying the pointwise inequality \eqref{eq:T-N-Z} for $\tilde{\phi} = \covZ^{(m)} \phi$, we obtain
\begin{equation*} 
	\abs{\covT \covZ^{(m)} \phi}
	\leq C \cosh \hY (\abs{\covN \covZ^{(m)} \phi} +  \hT^{-1} \abs{\covZ^{(m+1)} \phi}).
\end{equation*}
Recall that the last term on the right obeys the bound \eqref{eq:Z-lower-order}, which is stronger than what we need to prove \eqref{eq:sharpL3}, \eqref{eq:sharpL4} and \eqref{eq:sharpL6} for $\covT \covZ^{(m)} \phi$. Therefore, to prove the lemma, it only remains to establish the estimates \eqref{eq:sharpL3}, \eqref{eq:sharpL4} and \eqref{eq:sharpL6} for $\covN \covZ^{(m)} \phi$.

Using \eqref{eq:GN} to interpolate the $L^{2}_{\hT}$ bound \eqref{eq:BA:L2:S} and the sharp decay estimate \eqref{eq:BA:Linfty} for $\covN \phi$, it follows that
\begin{align*}
	\wnrm{\cosh \hY \covZ^{(m)} \covN \phi}_{L^{p}_{\hT}}
	\leq & C \eps_{1} \, \hT^{-1+\frac{2}{p}} \log^{m} (1+\hT)
\end{align*}
for $p = 3, 4, 6$ when $0 \leq m \leq 1$, and $p = 3$ when $m = 2$.
Applying Lemma~\ref{lem:BA:comm-NZ}, and using \eqref{eq:BA:L2} and Lemma~\ref{lem:BA:NZ} to bound the right-hand side of \eqref{eq:comm-NZ}, it follows that the commutator $[\covN, \covZ^{(m)}] \phi$ obeys stronger estimates than needed to establish \eqref{eq:sharpL3}, \eqref{eq:sharpL4} and \eqref{eq:sharpL6}. This completes the proof of the lemma. \qedhere
\end{proof}

We separately state the following lemma, which is a trivial consequence of Lemma \ref{lem:sharpLp} and H\"older's inequality, since it will be used a number of times below.
\begin{lemma} \label{lem:BA:bilin}
For non-negative integers $k_{1}, k_{2}$ such that $k_{1} + k_{2} \leq 2$, we have
\begin{equation} \label{eq:bilinL2}
	\wnrm{\, \abs{\covT \covZ^{(k_{1})} \phi} \abs{\covT \covZ^{(k_{2})} \phi} \, }_{L^{2}_{\hT}}
	\leq C \eps_{1}^{2} \hT^{-1} \log^{k_{1}+k_{2}} (1+\hT).
\end{equation}
\end{lemma}
%
%
%

\subsection{Improving the nonlinearity estimates}
Our next goal is to obtain the following improvement of \eqref{eq:BA:KG}: For $0 \leq m \leq 4$ we wish to show that
\begin{equation} \label{eq:BA:KG:improved}
	\wnrm{\cosh \hY (\covBox \m 1) \covZ^{(m)} \phi}_{L^{2}_{\hT}} 
	\leq C \eps_{1}^{3}  \hT^{-1} \log^{m-1} (1+\hT).
\end{equation}
In particular, note that the power of $\eps_{1}$ is $3$ in \eqref{eq:BA:KG:improved}, in contrast to $3-$ in \eqref{eq:BA:KG}; hence for a sufficiently small $\eps_{1}$ (independent of $T$), \eqref{eq:BA:KG:improved} would imply
\begin{equation*}
	\wnrm{\cosh \hY (\covBox \m 1) \covZ^{(m)} \phi}_{L^{2}_{\hT}} 
	\leq \eps_{1}^{3-} \hT^{-1} \log^{m-1} (1+\hT).	
\end{equation*}
which improves \eqref{eq:BA:KG}.

By \eqref{eq:CS-uni} and the commutation formula \eqref{eq:comm-covKG}, we have the schematic formula
\begin{equation*}
\begin{aligned}
(\covBox \m 1) \covZ^{(m)} \phi 
=& \covZ^{(m)} U(\phi) + \sum_{k_{1}+k_{2} \leq m-1} \frkN_{1}[\covLD_{Z}^{(k_{1})} J, \covZ^{(k_{2})} \phi] \\
&	+ \sum_{k_{1}+k_{2} \leq m-1} \frkN_{2}[\covLD_{Z}^{(k_{1})} J, \covZ^{(k_{2})} \phi] 
	+ \sum_{k_{1}+k_{2} \leq m-1} \frkN_{3}[\covLD_{Z}^{(k_{1})} J, \covZ^{(k_{2})} \phi] \\
&	+ \sum_{k_{1}+k_{2}+k_{3} \leq m-2} \frkN_{4}[\covLD_{Z}^{(k_{1})} J, \covLD_{Z}^{(k_{2})} J, \covZ^{(k_{3})} \phi],
\end{aligned}
\end{equation*}
where the last four terms are dropped in the case $m = 0$, and the last term is dropped when $m = 1$.

Hence, in order to establish \eqref{eq:BA:KG:improved}, we need to bound $\covZ \itr{m} U(\phi)$ and each $\frkN_{i}$ in $\wnrm{\cosh \hY (\cdot)}_{L^{2}_{\hT}}$. We achieve this task for \eqref{eq:CSH} and \eqref{eq:CSD} separately. In what follows, we assume that $(A, \phi)$ obeys the bootstrap assumptions \eqref{eq:BA:L2}--\eqref{eq:BA:KG}.


\subsubsection{Chern--Simons--Higgs equations}
We begin by handling the contribution of the $V$-valued potential $U_{\CSH}$.
Recall from Section~\ref{subsec:comm-U} that we may write $U_{\CSH}(\phi)
= U_{3}[\phi, \phi, \phi] + U_{5}[\phi, \phi, \phi, \phi, \phi]$, where $U_{3}$ and $U_{5}$ obey the Leibniz rules in Lemma~\ref{lem:comm-U-CSH}, and the pointwise bounds in Lemma~\ref{lem:ptwise-U}. Using the simple bounds in Lemma~\ref{lem:BA:KlSob} and H\"older's inequality, the following improved bounds can be shown:

\begin{proposition} \label{prop:U-CSH}
Let $(A, \phi)$ obey the bootstrap assumptions in Section~\ref{subsec:BAs}.
Then for $0 \leq m \leq 4$, we have
\begin{align} 
	\wnrm{\cosh \hY \covZ^{(m)} U_{3}[\phi, \phi, \phi]}_{L^{2}_{\hT}}
	\leq & C \eps_{1}^{3} \hT^{-2+} \\
	\wnrm{\cosh \hY \covZ^{(m)} U_{5}[\phi, \phi, \phi, \phi, \phi]}_{L^{2}_{\hT}}
	\leq & C \eps_{1}^{5} \hT^{-4+}.
\end{align}
\end{proposition}
We omit the straightforward details. These bounds show that the contribution of $U_{\CSH}$ is acceptable for the proof of \eqref{eq:BA:KG:improved}.


Next, we turn to the terms $\frkN_{j}$. To complete the proof of \eqref{eq:BA:KG:improved}, it suffices to establish the following bounds:
\begin{proposition} \label{prop:N-CSH}
Let $(A, \phi)$ obey the bootstrap assumptions in Section~\ref{subsec:BAs}. Then for $0 \leq m \leq 4$, we have
\begin{align}
	\sum_{k_{1}+k_{2} \leq m-1}
	\wnrm{\cosh \hY \, \frkN_{1}[\covLD_{Z}^{(k_{1})} J_{\CSH}, \covZ^{(k_{2})} \phi] }_{L^{2}_{\hT}} 
\leq& C \eps_{1}^{3}  \hT^{-1} \log^{m-1} (1+\hT) \label{eq:CSH:N1} \\
	\sum_{k_{1}+k_{2} \leq m-1}
	\wnrm{\cosh \hY \, \frkN_{2}[\covLD_{Z}^{(k_{1})} J_{\CSH}, \covZ^{(k_{2})} \phi] }_{L^{2}_{\hT}} 
\leq& C \eps_{1}^{3}  \hT^{-1} \log^{m-1} (1+\hT) \label{eq:CSH:N2} \\
	\sum_{k_{1}+k_{2} \leq 3}
	\wnrm{\cosh \hY \, \frkN_{3}[\covLD_{Z}^{(k_{1})} J_{\CSH}, \covZ^{(k_{2})} \phi] }_{L^{2}_{\hT}} 
\leq& C \eps_{1}^{3}  \hT^{-2+} \label{eq:CSH:N3} \\
	\sum_{k_{1}+k_{2}+k_{3} \leq 2}
	\wnrm{\cosh \hY \, \frkN_{4}[\covLD_{Z}^{(k_{1})} J_{\CSH}, \covLD_{Z}^{(k_{2})} J_{\CSH}, \covZ^{(k_{3})} \phi] }_{L^{2}_{\hT}} 
\leq& C \eps_{1}^{5}  \hT^{-2+} \label{eq:CSH:N4}
\end{align}
\end{proposition}

To prove this proposition, it is useful to first establish the following preliminary bilinear estimates, which may be done only using the bounds in Lemmas~\ref{lem:BA:KlSob} and \ref{lem:BA:NZ}.
\begin{lemma} \label{lem:CSH:weakBilin}
Let $(A, \phi)$ obey the bootstrap assumptions in Section~\ref{sec:BA}. Then we have
\begin{align} 
	\sum_{k_{1}+k_{2} \leq 3} \wnrm{\cosh \hY \, \abs{\Gmm_{\CSH}^{(0)} [\covZ^{(k_{1})} \phi, \covZ^{(k_{2})} \phi]} \, }_{L^{2}_{\hT}}
	\leq & C \eps_{1}^{2} \hT^{-1+},	 \label{eq:CSH:Gmm0} \\
	\sum_{k_{1}+k_{2} \leq 2} \wnrm{\cosh \hY \, \abs{\covLD_{Z} \Gmm_{\CSH}^{(0)} [\covZ^{(k_{1})} \phi, \covZ^{(k_{2})} \phi]} \, }_{L^{2}_{\hT}}
	\leq&  C \eps_{1}^{2} \hT^{-1+},	 \label{eq:CSH:ZGmm0} \\
	\sum_{k_{1}+k_{2} \leq 2} \wnrm{\covud \Gmm_{\CSH}^{(0)} [\covZ^{(k_{1})} \phi, \covZ^{(k_{2})} \phi]}_{L^{2}_{\hT}}
	\leq&  C \eps_{1}^{2} \hT^{-1+},	 \label{eq:CSH:dGmm0} \\
	\sum_{k_{1}+k_{2} \leq 2} \wnrm{\covdlt \Gmm_{\CSH}^{(0)} [\covZ^{(k_{1})} \phi, \covZ^{(k_{2})} \phi]}_{L^{2}_{\hT}}
	\leq&  C \eps_{1}^{2} \hT^{-1+}.  \label{eq:CSH:dltGmm0} 
\end{align}
\end{lemma}
\begin{proof} 
All four bounds are proved in the same way: First, apply formulae and pointwise bounds in Section~\ref{sec:comm}, then use H\"older's inequality, Lemmas~\ref{lem:BA:KlSob} and \ref{lem:BA:NZ}, as well as  \eqref{eq:BA:L2}. We give a detailed proof only for \eqref{eq:CSH:dltGmm0}, which is the most involved, and leave the rest to the reader.

Fix a pair of non-negative intergers $k_{1}, k_{2}$ such that $k_{1} + k_{2} \leq 2$. By \eqref{eq:covdlt-CSH} and the pointwise bounds in Section~\ref{subsec:ptwise}, we have
\begin{align*}
& \hskip-2em
	\abs{\covdlt \Gmm_{\CSH}^{(0)} [\covZ^{(k_{1})} \phi, \covZ^{(k_{2})} \phi]} \\
	\leq & C \bb( \abs{\covT \covZ^{(k_{1})} \phi} \abs{\covT \covZ^{(k_{2})} \phi} 
		+ \abs{\covZ^{(k_{1})} \phi} \abs{(\covBox \m 1) \covZ^{(k_{2})} \phi}
		+ \abs{\covZ^{(k_{1})} \phi} \abs{\covZ^{(k_{2})}\phi} \bb).
\end{align*}
We bound the $\wnrm{\cdot}_{L^{2}_{\hT}}$ norm of the preceding expression as follows: The contribution of the first term is treated by applying H\"older's inequality, bounding the higher order factor in $\wnrm{\cdot}_{L^{2}_{\hT}}$ and the other in $\wnrm{\cdot}_{L^{\infty}_{\hT}}$, and then appealing to \eqref{eq:BA:L2} and Lemma~\ref{lem:BA:NZ}. The third term is handled similarly, where we replace Lemma~\ref{lem:BA:NZ} by Lemma~\ref{lem:BA:KlSob}. Finally, for the second term, we bound $\covZ^{(k_{1})} \phi$ in $\wnrm{\cdot}_{L^{\infty}_{\hT}}$ and $(\covBox \m 1) \covZ^{(k_{2})} \phi$ in $\wnrm{\cdot}_{L^{2}_{\hT}}$, then use Lemma~\ref{lem:BA:KlSob} for the former and \eqref{eq:BA:KG} for the latter. \qedhere
\end{proof}

We are now ready to prove \eqref{eq:CSH:N1}--\eqref{eq:CSH:N4}, which would prove Proposition~\ref{prop:N-CSH}.

\begin{proof}[Proof of \eqref{eq:CSH:N1}]
In order to prove \eqref{eq:CSH:N1}, it suffices by \eqref{eq:comm-J-CSH} and the facts that $\eps_{1} \leq 1$, $\hT \geq 2R$ to establish the following bounds: For $1 \leq m \leq 4$,
\begin{align}
	\sum_{k_{1}+k_{2}+k_{3} \leq m-1}	\wnrm{\cosh \hY \,\frkN_{1}[\Gmm_{\CSH}^{(0)}[\covZ^{(k_{1})} \phi, \covZ^{(k_{2})} \phi], \covZ^{(k_{3})} \phi] }_{L^{2}_{\hT}} 	
	\leq& C \eps_{1}^{3}  \hT^{-1} \log^{m-1} (1+\hT) 	\label{eq:CSH:N1:1} \\
	\sum_{k_{1}+\cdots+k_{5} \leq 2}
	\wnrm{\cosh \hY \,\frkN_{1}[\Gmm_{\CSH}^{(1)}[\covZ^{(k_{1})} \phi, \ldots, \covZ^{(k_{4})} \phi], \covZ^{(k_{5})} \phi] }_{L^{2}_{\hT}} 
	\leq& C \eps_{1}^{5}  \hT^{-2+}		\label{eq:CSH:N1:2}\\
	\sum_{k_{1}+\cdots+k_{7} \leq 1}
	\wnrm{\cosh \hY \,\frkN_{1}[\Gmm_{\CSH}^{(2)}[\covZ^{(k_{1})} \phi, \ldots, \covZ^{(k_{6})} \phi], \covZ^{(k_{7})} \phi] }_{L^{2}_{\hT}} 
	\leq& C \eps_{1}^{7}  \hT^{-3+}  \label{eq:CSH:N1:3}\\
	\wnrm{\cosh \hY \,\frkN_{1}[\Gmm_{\CSH}^{(3)}[\phi, \phi, \ldots, \phi], \phi] }_{L^{2}_{\hT}} 
	\leq& C \eps_{1}^{9}  \hT^{-4+}	\label{eq:CSH:N1:4}
\end{align}
To simplify the notation, we will often use the shorthand $\phi^{j} = \covZ^{(k_{j})} \phi$ in what follows.

\paragraph*{\bfseries - Proof of \eqref{eq:CSH:N1:1}}
By \eqref{eq:ptwise-N1} and \eqref{eq:covud-CSH}, we first derive the following pointwise bound:
\begin{align*}
\cosh \hY \abs{\frkN_{1}[\Gmm_{\CSH}^{(0)}[\phi^{1}, \phi^{2}], \phi^{3}}
\leq & C \hT \cosh \hY \abs{\iota_{N} \covud \Gmm_{\CSH}^{(0)}[\phi^{1}, \phi^{2}]}\abs{\phi^{3}} \\
\leq & C \hT \cosh \hY \bb( \abs{\covT \phi^{1}} \abs{\covN \phi^{2}} + \abs{\covN \phi^{1}} \abs{\covT \phi^{2}} + \abs{\phi} \abs{\phi} \abs{\phi^{1}} \abs{\phi^{2}} \bb)\abs{\phi^{3}} .
\end{align*}
%
%
We now take the $\wnrm{\cdot}_{L^{2}_{\hT}}$ norm. The cubic terms can be estimated by $C \eps_{1}^{3} \hT^{-1} \log^{m-1} (1+\hT)$, using \eqref{eq:BA:L2}, \eqref{eq:NZ-L2} and Lemma \ref{lem:sharpLp}. 
The quintic term can be easily bounded by $\eps_{1}^{5} \hT^{-4+}$ using \eqref{eq:BA:L2} for the highest order factor and \eqref{eq:weakLinfty} for the rest. The point is that all factors except the highest order factor have at most two $\covZ$ derivatives, and thus \eqref{eq:weakLinfty} is applicable.

\paragraph*{\bf - Proof of \eqref{eq:CSH:N1:2}, \eqref{eq:CSH:N1:3} and \eqref{eq:CSH:N1:4}}
For the remaining cases \eqref{eq:CSH:N1:2}, \eqref{eq:CSH:N1:3} and \eqref{eq:CSH:N1:4}, the nonlinearity is quintic or higher and the total number of $\covZ$ derivatives is $\leq 2$.  These cases turn out to be much less delicate compared to \eqref{eq:CSH:N1:1}, and can be treated using just \eqref{eq:BA:L2}, \eqref{eq:weakLinfty} (in Lemma \ref{lem:BA:KlSob}), Lemma~\ref{lem:BA:bilin} and Lemma~\ref{lem:CSH:weakBilin}. Furthermore, we may rely on the crude pointwise bound \eqref{eq:ptwise-easy} to treat $\iota_{Z} \star$. Since the arguments are similar, we only present the case of \eqref{eq:CSH:N1:2} in detail, and briefly sketch the others.

To prove \eqref{eq:CSH:N1:2}, we distinguish two types of terms, namely those which do not involve commutation of $\covud$ with $\iota_{Z} \star$, and those which arise from this commutation. More precisely, by \eqref{eq:Gmm-CSH-1}, \eqref{eq:ptwise-easy}, \eqref{eq:ptwise-N1-easy} and Lemma~\ref{lem:d-i-star}, we first bound $\cosh \hY \abs{\frkN_{1}[\Gmm_{\CSH}^{(1)}[\phi^{1}, \ldots, \phi^{4}], \phi^{5}]}$ by
\begin{align}
& C\hT \cosh^{2} \hY \abs{\covud \phi^{1} \wedge (\iota_{Z} \star \Gmm_{\CSH}^{(0)}[\phi^{3}, \phi^{4}]) \wedge \phi^{2}} \abs{\phi^{5} } \label{eq:CSH:N1:2:1} \\
&		+ C\hT \cosh^{2} \hY \abs{\phi^{1} \wedge (\iota_{Z} \star \Gmm_{\CSH}^{(0)}[\phi^{3}, \phi^{4}]) \wedge \covud \phi^{2}} \abs{\phi^{5} } \label{eq:CSH:N1:2:1.5} \\
&		+ C\hT \cosh^{2} \hY \abs{\phi^{1}} \abs{\phi^{2}} \abs{\iota_{Z} \star \covdlt \Gmm_{\CSH}^{(0)}[\phi^{3}, \phi^{4}] } \abs{\phi^{5}} 
\label{eq:CSH:N1:2:2} \\
&		+ C\hT \cosh^{2} \hY \abs{\phi^{1}} \abs{\phi^{2}} \abs{\star \covLD_{Z} \Gmm_{\CSH}^{(0)}[\phi^{3}, \phi^{4}] } \abs{\phi^{5}}  
\label{eq:CSH:N1:2:3}
\end{align}
where we recall that $k_{1} + \cdots k_{5} \leq 2$. Note that \eqref{eq:CSH:N1:2:1} and \eqref{eq:CSH:N1:2:1.5} are precisely the terms where $\covud$ does not fall on $\iota_{Z} \star$, whereas \eqref{eq:CSH:N1:2:2} and \eqref{eq:CSH:N1:2:3} arise from commuting $\covud$ with $\iota_{Z} \star$ using Lemma~\ref{lem:d-i-star}.

For \eqref{eq:CSH:N1:2:1} and \eqref{eq:CSH:N1:2:1.5}, we first apply \eqref{eq:ptwise-easy} to derive the pointwise estimate 
\begin{align*}
\eqref{eq:CSH:N1:2:1} + \eqref{eq:CSH:N1:2:1.5}
\leq &	C\hT^{2} (\cosh \hY)^{3} (\abs{\covT \phi^{1}}\abs{\phi^{2}} 
							+ \abs{\phi^{1}}\abs{\covT \phi^{2}})
							 \abs{\phi^{3}} \abs{\covT \phi^{4}} \abs{\phi^{5}}
\end{align*}
and then estimate the $\wnrm{\cdot}_{L^{2}_{\hT}}$ norm of the right-hand side by $C\eps_{1}^{5} \hT^{-2+}$, using Lemma \ref{lem:BA:bilin} for the factors with $\covT$ and \eqref{eq:weakLinfty} for the rest.

For \eqref{eq:CSH:N1:2:2}, we begin with the pointwise bound
\begin{align*}
\eqref{eq:CSH:N1:2:2}
\leq &	C\hT^{2} (\cosh \hY)^{3} \abs{\phi^{1}} \abs{\phi^{2}} \abs{\covdlt \Gmm_{\CSH}^{(0)} [\phi^{3}, \phi^{4}]} \abs{\phi^{5}},
\end{align*}
which follows from \eqref{eq:ptwise-easy}, and then estimate the $\wnrm{\cdot}_{L^{2}_{\hT}}$ norm of the right-hand side by $C \eps_{1}^{5} \hT^{-2+} $ using \eqref{eq:weakLinfty} and \eqref{eq:CSH:dltGmm0}.

Finally, for \eqref{eq:CSH:N1:2:3}, we apply \eqref{eq:ptwise-easy} to estimate
\begin{align*}
	\eqref{eq:CSH:N1:2:3}
\leq &	C\hT (\cosh \hY)^{2} \abs{\phi^{1}} \abs{\phi^{2}} \abs{\covLD_{Z} \Gmm_{\CSH}^{(0)} [\phi^{3}, \phi^{4}]} \abs{\phi^{5}},
\end{align*}
and then the $\wnrm{\cdot}_{L^{2}_{\hT}}$ norm of the right-hand side is estimated by $C \eps_{1}^{5} \hT^{-3+}$ using \eqref{eq:weakLinfty} and \eqref{eq:CSH:ZGmm0}.

Now we sketch the proofs of \eqref{eq:CSH:N1:3} and \eqref{eq:CSH:N1:4}. 
For \eqref{eq:CSH:N1:3}, we begin by estimating $\cosh \hY \abs{\frkN_{1}[\Gmm_{\CSH}^{(2)}[\phi^{1}, \ldots, \phi^{6}], \phi^{7}]}$ by
\begin{align}
&	C \hT^{3} (\cosh \hY)^{4} \bb( \abs{\covT \phi^{1}} \abs{\phi^{2}}  \abs{\phi^{3}} \abs{\phi^{4}} + \cdots + \abs{\phi^{1}} \abs{\phi^{2}} \abs{\phi^{3}} \abs{\covT \phi^{4}} \bb) \abs{\phi^{5}} \abs{\covT \phi^{6}} \abs{\phi^{7}} \label{eq:CSH:N1:3:1} \\
&	+ 	C \hT \cosh^{2} \hY \abs{\phi^{1}} \cdots \abs{\phi^{4}}
		\abs{(\iota_{Z} \star)^{2} \covud \Gmm^{(0)}_{\CSH} [\phi^{5}, \phi^{6}]} \abs{\phi^{7}}
	\label{eq:CSH:N1:3:2} \\
&	+	C\hT \cosh^{2} \hY \abs{\phi^{1}} \cdots \abs{\phi^{4}}
		\abs{\star \covLD_{Z} \iota_{Z} \star  \Gmm^{(0)}_{\CSH} [\phi^{5}, \phi^{6}]} \abs{\phi^{7}} 
	\label{eq:CSH:N1:3:3}	\\
&	+	C \hT \cosh^{2} \hY \abs{\phi^{1}} \cdots \abs{\phi^{4}}
		\abs{\iota_{Z} (\Gmm^{(0)}_{\CSH} [\phi^{5}, \phi^{6}] \wedge \ud Z^{\flat})} \abs{\phi^{7}} .
	\label{eq:CSH:N1:3:4}
\end{align}
Note that \eqref{eq:CSH:N1:3:1} is precisely the contribution of the terms with $\covud$ not falling on $(\iota_{Z} \star)^{2}$ (where we used the trivial bounds \eqref{eq:ptwise-easy} to simplify the expression), whereas \eqref{eq:CSH:N1:3:2}--\eqref{eq:CSH:N1:3:4} arise from commuting $\covud$ with $(\iota_{Z} \star)^{2}$ using Lemma~\ref{lem:d-i-star} (twice).

%
%

Recall that $k_{1} + \cdots + k_{7} \leq 1$. We estimate $\wnrm{\eqref{eq:CSH:N1:3:1}}_{L^{2}_{\hT}} \leq C \eps_{1}^{7} \hT^{-3+}$ using \eqref{eq:weakLinfty}, \eqref{eq:ptwise-easy} and Lemma \ref{lem:BA:bilin}, as in the case of \eqref{eq:CSH:N1:2:1}. Furthermore, using \eqref{eq:weakLinfty}, \eqref{eq:CSH:Gmm0}, \eqref{eq:CSH:ZGmm0}, \eqref{eq:CSH:dGmm0} and \eqref{eq:ptwise-easy}, we have
$\wnrm{\eqref{eq:CSH:N1:3:2}+\eqref{eq:CSH:N1:3:4}}_{L^{2}_{\hT}} \leq C \eps_{1}^{7} \hT^{-3+}$ and 
$\wnrm{\eqref{eq:CSH:N1:3:3}}_{L^{2}_{\hT}} \leq C \eps_{1}^{7} \hT^{-4+}$, from which \eqref{eq:CSH:N1:3} follows.

It only remains to prove \eqref{eq:CSH:N1:4}. Using \eqref{eq:ptwise-N1-easy}, we bound $\cosh \hY \abs{\frkN_{1}[\Gmm_{\CSH}^{(3)}[\phi, \ldots, \phi], \phi]}$  by
\begin{align}
& C \hT^{4} (\cosh \hY)^{5} \abs{\covT \phi}^{2} \abs{\phi}^{7}
\label{eq:CSH:N1:4:1} \\
&	+	C\hT \cosh^{2} \hY \abs{\phi}^{7} \bb(
		\abs{(\iota_{Z} \star)^{3} \covdlt \Gmm_{\CSH}^{(0)}[\phi, \phi]} 
		+ C\abs{\iota_{Z}   ((\iota_{Z} \star \Gmm_{\CSH}^{(0)}[\phi, \phi]) \wedge \ud Z^{\flat})} \bb)
		\label{eq:CSH:N1:4:2}  \\
&	+	C \hT \cosh^{2} \hY \abs{\phi}^{7} \bb(
		\abs{\iota_{Z} \star \iota_{Z} \covLD_{Z} \Gmm_{\CSH}^{(0)}[\phi, \phi]} 
		+ C \abs{\star \covLD_{Z} (\iota_{Z} \star)^{2} \Gmm_{\CSH}^{(0)}[\phi, \phi]} \bb). 	\label{eq:CSH:N1:4:3}
\end{align}
where \eqref{eq:CSH:N1:4:1} is again the contributions of the terms with $\covud$ not falling on $(\iota_{Z} \star)^{3}$, and \eqref{eq:CSH:N1:4:2}--\eqref{eq:CSH:N1:4:3} arise from applying Lemma~\ref{lem:d-i-star} (three times).

We have $\wnrm{\eqref{eq:CSH:N1:4:1}} \leq C \eps_{1}^{9} \hT^{-4+}$ using Lemma \ref{lem:BA:bilin}, \eqref{eq:weakLinfty} and \eqref{eq:ptwise-easy}, again as in  \eqref{eq:CSH:N1:2:1}. Using \eqref{eq:weakLinfty}, \eqref{eq:CSH:Gmm0}, \eqref{eq:CSH:ZGmm0}, \eqref{eq:CSH:dltGmm0} and \eqref{eq:ptwise-easy}, we also have
$\wnrm{\eqref{eq:CSH:N1:4:2}} \leq C \eps_{1}^{9} \hT^{-4+}$ and
$\wnrm{\eqref{eq:CSH:N1:4:3}} \leq C \eps_{1}^{9} \hT^{-4+}$. 
The desired estimate \eqref{eq:CSH:N1:4} now follows, which concludes our proof of \eqref{eq:CSH:N1}. \qedhere
\end{proof}

\begin{proof}[Proof of \eqref{eq:CSH:N2}]
This case obeys the same estimate as \eqref{eq:CSH:N1} (see \eqref{eq:CSH:N2:1}); however, since there is no need to compute $\covud J_{\CSH}$, the amount of work needed is much less than \eqref{eq:CSH:N1}.

As in our proof of \eqref{eq:CSH:N1}, it suffices to establish the following bounds: For $1 \leq m \leq 4$,
\begin{align}
	\sum_{k_{1}+k_{2}+k_{3} \leq m-1}	\wnrm{\cosh \hY \,\frkN_{2}[\Gmm_{\CSH}^{(0)}[\covZ^{(k_{1})} \phi, \covZ^{(k_{2})} \phi], \covZ^{(k_{3})} \phi] }_{L^{2}_{\hT}} 	
	\leq& C \eps_{1}^{3}  \hT^{-1} \log^{m-1} (1+\hT) 	\label{eq:CSH:N2:1} \\
	\sum_{k_{1}+\cdots+k_{5} \leq 2}
	\wnrm{\cosh \hY \,\frkN_{2}[\Gmm_{\CSH}^{(1)}[\covZ^{(k_{1})} \phi, \ldots, \covZ^{(k_{4})} \phi], \covZ^{(k_{5})} \phi] }_{L^{2}_{\hT}} 
	\leq& C \eps_{1}^{5}  \hT^{-2+}		\label{eq:CSH:N2:2}\\
	\sum_{k_{1}+\cdots+k_{7} \leq 1}
	\wnrm{\cosh \hY \,\frkN_{2}[\Gmm_{\CSH}^{(2)}[\covZ^{(k_{1})} \phi, \ldots, \covZ^{(k_{6})} \phi], \covZ^{(k_{7})} \phi] }_{L^{2}_{\hT}} 
	\leq& C \eps_{1}^{7}  \hT^{-3+}  \label{eq:CSH:N2:3}\\
	\wnrm{\cosh \hY \,\frkN_{2}[\Gmm_{\CSH}^{(3)}[\phi, \phi, \ldots, \phi], \phi] }_{L^{2}_{\hT}} 
	\leq& C \eps_{1}^{9}  \hT^{-4+}	\label{eq:CSH:N2:4}
\end{align}
As before, we use the shorthand $\phi^{j} = \covZ^{(k_{j})} \phi$ in what follows.
\paragraph*{\bfseries - Proof of \eqref{eq:CSH:N2:1}}
By \eqref{eq:ptwise-N1}, we have 
\begin{align*}
\cosh \hY \abs{\frkN_{2}[\Gmm_{\CSH}^{(0)}[\phi^{1}, \phi^{2}], \phi^{3}]}
\leq & C \hT \cosh \hY \abs{\phi^{1}} 
	(\abs{\covN \phi^{2}} \abs{\covT \phi^{3}}
	+ \abs{\covT \phi^{2}} \abs{\covN  \phi^{3}}),
\end{align*}
where $k_{1}+k_{2}+k_{3} \leq m-1$ and $1 \leq m \leq 4$. Then using \eqref{eq:BA:L2}, \eqref{eq:NZ-L2} and Lemma \ref{lem:sharpLp}, the $\wnrm{\cdot}_{L^{2}_{\hT}}$ norm of the right-hand side can be estimated by $C \eps_{1}^{3} \hT^{-1} \log^{(m-1)} (1 + \hT)$ as desired.

\paragraph*{\bfseries - Proof of \eqref{eq:CSH:N2:2}, \eqref{eq:CSH:N2:3} and \eqref{eq:CSH:N2:4}}
As in our preceding proof of \eqref{eq:CSH:N1:2}--\eqref{eq:CSH:N1:4}, there is more room in this case. This case can be treated using just \eqref{eq:weakLinfty} and Lemma \ref{lem:BA:bilin}, relying on the pointwise bounds \eqref{eq:ptwise-N2-easy} and \eqref{eq:ptwise-CSH}. We omit the straightforward details. \qedhere

\end{proof}
\begin{proof}[Proof of \eqref{eq:CSH:N3} and \eqref{eq:CSH:N4}]
The estimates \eqref{eq:CSH:N3} and \eqref{eq:CSH:N4} are easier than the preceding cases, and can be proved with similar techniques as before. The key ingredients are: the pointwise bounds \eqref{eq:ptwise-N3} and \eqref{eq:ptwise-N4} for $\frkN_{3}$ and $\frkN_{4}$, respectively; Proposition~\ref{prop:comm-J-CSH}, which allows us to expand $\covLD_{Z}^{(k)} J_{\CSH}$ in terms of $\Gmm_{\CSH}^{(\ell)}$ as in \eqref{eq:comm-J-CSH}; the general pointwise bound \eqref{eq:ptwise-CSH} for $\Gmm_{\CSH}^{(\ell)}$; Lemma~\ref{lem:BA:KlSob}, Lemma~\ref{lem:BA:NZ} and the bound \eqref{eq:BA:L2}. We omit the routine proof.\qedhere

%
\end{proof}

\subsubsection{Chern--Simons--Dirac equations}
We now consider the case of \eqref{eq:CSD} and establish \eqref{eq:BA:KG:improved}. As before, we first handle the contribution of $U_{\CSD}$. Recall from Section~\ref{subsec:comm-U} that $U_{\CSD}(\phi) = \tilde{U}_{3}[\phi, \phi, \phi]$, where $\tilde{U}_{3}$ obeys the Leibniz rules in Lemma~\ref{lem:comm-U-CSD} and the pointwise bounds in Lemma~\ref{lem:ptwise-U}. Then by Lemma~\ref{lem:BA:KlSob} and H\"older's inequality, we obtain:
\begin{proposition} \label{prop:U-CSD}
Let $(A, \phi)$ obey the bootstrap assumptions in Section~\ref{subsec:BAs}. Then for $0 \leq m \leq 4$, we have
\begin{equation} 
	\wnrm{\cosh \hY \covZ^{(m)} \tilde{U}_{3}[\phi, \phi, \phi]}_{L^{2}_{\hT}}
	\leq C \eps_{1}^{3} \hT^{-2+}.
\end{equation}
\end{proposition}
We omit the details. Proposition~\ref{prop:U-CSD} shows that the contribution of $U_{\CSD}$ is acceptable for proving \eqref{eq:BA:KG:improved}.

Next, we treat the terms $\frkN_{j}$. As in the case of \eqref{eq:CSH}, it is sufficient to establish the following bounds hold for $\frkN_{1}, \ldots, \frkN_{4}$:
\begin{proposition} \label{prop:N-CSD}
Let $(A, \phi)$ obey the bootstrap assumptions in Section~\ref{subsec:BAs}. Then for $0 \leq m \leq 4$, we have
\begin{align}
	\sum_{k_{1}+k_{2} \leq m-1}
	\wnrm{\cosh \hY \, \frkN_{1}[\covLD_{Z}^{(k_{1})} J_{\CSD}, \covZ^{(k_{2})} \phi] }_{L^{2}_{\hT}} 
\leq& C \eps_{1}^{3}  \hT^{-1} \log^{m-1} (1+\hT), \label{eq:CSD:N1} \\
	\sum_{k_{1}+k_{2} \leq m-1}
	\wnrm{\cosh \hY \, \frkN_{2}[\covLD_{Z}^{(k_{1})} J_{\CSD}, \covZ^{(k_{2})} \phi] }_{L^{2}_{\hT}} 
\leq& C \eps_{1}^{3}  \hT^{-1} \log^{m-1} (1+\hT), \label{eq:CSD:N2} \\
	\sum_{k_{1}+k_{2} \leq 3}
	\wnrm{\cosh \hY \, \frkN_{3}[\covLD_{Z}^{(k_{1})} J_{\CSD}, \covZ^{(k_{2})} \phi] }_{L^{2}_{\hT}} 
\leq& C \eps_{1}^{3}  \hT^{-2+}, \label{eq:CSD:N3} \\
	\sum_{k_{1}+k_{2}+k_{3} \leq 2}
	\wnrm{\cosh \hY \, \frkN_{4}[\covLD_{Z}^{(k_{1})} J_{\CSD}, \covLD_{Z}^{(k_{2})} J_{\CSD}, \covZ^{(k_{3})} \phi] }_{L^{2}_{\hT}} 
\leq& C \eps_{1}^{5}  \hT^{-2+}. \label{eq:CSD:N4}
\end{align}
\end{proposition}
We only give a sketch of the proof, since the method is not too different from the previous case of \eqref{eq:CSH}. In fact, 
the task of establishing these bounds is far simpler in the case of \eqref{eq:CSD}, thanks in large part to the absence of derivatives in $J_{\CSD}$. 
\begin{proof} [Sketch of proof]
By Proposition~\ref{prop:comm-J-CSD}, it suffices to establish the following bounds involving $\Gmm_{\CSD}^{(\ell)}$:
\begin{gather}
	\sum_{\ell + j_{1} + j_{2} + k \leq m-1}
	\wnrm{\cosh \hY \, \frkN_{1}[\Gmm_{\CSD}^{(\ell)}[\phi^{1}, \phi^{2}], \tilde{\phi}] }_{L^{2}_{\hT}} 
\leq  C \eps_{1}^{3}  \hT^{-1} \log^{m-1} (1+\hT), \label{eq:CSD:N1-Gmm} \\
	\sum_{\ell + j_{1} + j_{2} + k \leq m-1}
	\wnrm{\cosh \hY \, \frkN_{2}[\Gmm_{\CSD}^{(\ell)}[\phi^{1}, \phi^{2}], \tilde{\phi}] }_{L^{2}_{\hT}} 
\leq  C \eps_{1}^{3}  \hT^{-1} \log^{m-1} (1+\hT),  \label{eq:CSD:N2-Gmm} \\
	\sum_{\ell + j_{1} + j_{2} + k \leq 3}
	\wnrm{\cosh \hY \, \frkN_{3}[\Gmm_{\CSD}^{(\ell)}[\phi^{1}, \phi^{2}], \tilde{\phi}] }_{L^{2}_{\hT}} 
\leq  C \eps_{1}^{3}  \hT^{-2+}, \label{eq:CSD:N3-Gmm} \\
	\sum_{\ell_{1} + \ell_{2} + j_{1} + j_{2} + j_{3} + j_{4} + k \leq 2}
	\wnrm{\cosh \hY \, \frkN_{4}[\Gmm_{\CSD}^{(\ell_{1})}[\phi^{1}, \phi^{2}], \Gmm_{\CSD}^{(\ell_{2})}[\phi^{3}, \phi^{4}], \tilde{\phi}] }_{L^{2}_{\hT}} 
\leq  C \eps_{1}^{5}  \hT^{-2+},  \label{eq:CSD:N4-Gmm}
\end{gather}
where we have used the shorthand $\phi^{i} = \covZ^{(j_{i})} \phi$ for $i=1,2,3,4$ and $\tilde{\phi} = \covZ^{(k)} \phi$. By Lemma~\ref{lem:ptwise-N} and \eqref{eq:ptwise-easy}, we may derive the following pointwise bounds:
\begin{align*}
	\cosh \hY \abs{\frkN_{1}[\Gmm_{\CSD}^{(\ell)}[\phi^{1}, \phi^{2}], \tilde{\phi}]}	
&	\leq C \hT (\cosh \hY)^{2}
	\bb( \sum_{\set{i_{1}, i_{2}} = \set{1,2}} \abs{\covT \phi^{i_{1}}}\abs{\phi^{i_{2}}} \bb) \abs{\tilde{\phi}}	\\
	\cosh \hY \abs{\frkN_{2}[\Gmm_{\CSD}^{(\ell)}[\phi^{1}, \phi^{2}], \tilde{\phi}]}	
&	\leq C \hT (\cosh \hY)^{2} \abs{\phi^{1}} \abs{\phi^{2}} \abs{\covT \tilde{\phi}} \\
	\cosh \hY \abs{\frkN_{3}[\Gmm_{\CSD}^{(\ell)}[\phi^{1}, \phi^{2}], \tilde{\phi}]}	
&	\leq 	C \cosh \hY \abs{\phi^{1}}\abs{\phi^{2}}\abs{\tilde{\phi}} \\
	\cosh \hY \abs{\frkN_{4}[\Gmm_{\CSD}^{(\ell_{1})}[\phi^{1}, \phi^{2}], \Gmm_{\CSD}^{(\ell_{2})}[\phi^{3}, \phi^{4}], \tilde{\phi}]}	
&	\leq  C \hT^{2} (\cosh \hY)^{3} \abs{\phi^{1}} \cdots \abs{\phi^{4}} \abs{\tilde{\phi}}
\end{align*}
Then by H\"older's inequality, \eqref{eq:BA:L2}, Lemma~\ref{lem:sharpLp} and Lemma~\ref{lem:BA:NZ} (only used for \eqref{eq:CSD:N4-Gmm}), it is straightforward to prove the bounds \eqref{eq:CSD:N1-Gmm}--\eqref{eq:CSD:N4-Gmm}. \qedhere
\end{proof}


\subsection{Improving energy and decay estimates}
In this subsection, we will prove
\begin{equation} \label{eq:BA:L2:improved}
\begin{aligned}  
&\hskip-2em
	\sum_{0 \leq m \leq 4} \bb( \wnrm{\cosh \hY \covZ^{(m)} \phi}_{L^{2}_{\hT}} + \wnrm{\covT \covZ^{(m)} \phi}_{L^{2}_{\hT}} \bb)
	+ \sum_{1 \leq m \leq 4} \wnrm{\cosh \hY \covN \covZ^{(m-1)} \phi}_{L^{2}_{\hT}} \\
&	\leq C_{1} (\eps + \eps_{1}^{3}  \log^{m} (1+\hT))
\end{aligned}
\end{equation}
and
\begin{align}  \label{eq:BA:Linfty:improved}
	\wnrm{\cosh \hY \phi}_{L^{\infty}_{\hT}} 
	+ \wnrm{\cosh \hY \covN \phi}_{L^{\infty}_{\hT}}
	+ \wnrm{\covT \phi}_{L^{\infty}_{\hT}} 
	\leq C_{1} (\eps + \eps_{1}^{3}) \frac{1}{\hT}
\end{align}
for some constant $0 < C_{1} < \infty$. Once these estimates are proved, choosing $\eps_{1} = B \eps \leq B \dlt_{\ast}(R)$ with $B = 2 C_{1}$ and taking $\dlt_{\ast}(R)$ sufficiently small, \eqref{eq:BA:L2}, \eqref{eq:BA:L2:S} and \eqref{eq:BA:Linfty} would improve and hence the proof of Proposition~\ref{prop:main} would be complete.

We begin with \eqref{eq:BA:L2:improved}. By the Chern--Simons equation $F = \star J$ and \eqref{eq:BA:Linfty}, for both \eqref{eq:CSH} and \eqref{eq:CSD} we have
\begin{equation*}
	\int_{2R}^{T} \sup_{\calH_{\hT}} \bb( \sum_{\mu} \abs{F(T_{\mu}, T_{0})}^{2} \bb)^{1/2} \, \ud \hT
	\leq C \eps_{1}^{2} \int_{2R}^{\infty} \hT^{-2} \, \ud \hT \leq C \eps_{1}^{2}.
\end{equation*}
Taking $\eps_{1}$ sufficiently small, we may apply the covariant energy inequality (Proposition~\ref{prop:en}) with $C_{F} = 1$ and $\hT_{0} = 2R$. Recall from \eqref{eq:BA:ini} that $\eps$ stands for the size of the data on $\calH_{2R}$. Using also the improved nonlinearity estimate \eqref{eq:BA:KG:improved}, we obtain \eqref{eq:BA:L2:improved}.

Next, we turn to proving \eqref{eq:BA:Linfty:improved}. 
By the pointwise estimate \eqref{eq:T-N-Z}, it suffices to establish
\begin{align}
	\wnrm{\cosh \hY \phi}_{L^{\infty}_{\hT}} 
	+ \wnrm{\cosh \hY \covN \phi}_{L^{\infty}_{\hT}}
	\leq & C_{1}' (\eps + \eps_{1}^{3}) \frac{1}{\hT}. \label{eq:BA:Linfty:improved-2}
\end{align}
for some constant $0 < C_{1}' < \infty$.  

To prove \eqref{eq:BA:Linfty:improved-2}, we now apply Proposition~\ref{prop:ODE}. The first term on the right-hand side of \eqref{eq:ODE} is bounded by $C \eps$ thanks to the Klainerman--Sobolev inequality. Next, using the Klainerman--Sobolev inequality and \eqref{eq:BA:L2:improved}, which we just proved, we bound the second term on the right-hand side of \eqref{eq:ODE} by
\begin{align*}
	\sum_{1 \leq k \leq 2} \int_{2R}^{T} \frac{\cosh \hY}{\hT'} \wnrm{\covZ^{k} \phi}_{L^{\infty}_{\hT'}} \, \ud \hT'
	\leq & C \sum_{1 \leq k \leq 4} \sup_{\hT: 2R \leq \hT < T} \wnrm{\cosh \hY \covZ^{(k)} \phi}_{L^{2}_{\hT}} \int_{2R}^{\infty} (\hT')^{-2+} \, \ud \hT' \\
	\leq & C ( \eps + \eps_{1}^{3}).
\end{align*}
To handle the last term in \eqref{eq:ODE}, we begin by noting that for both \eqref{eq:CSH} and \eqref{eq:CSD}, we have
\begin{equation*}
	\abs{(\covBox \m 1) \phi} \leq C \abs{\phi}^{2} \max \set{\abs{\phi}, \abs{\covT \phi}}.
\end{equation*}
Therefore, by \eqref{eq:BA:Linfty}, we have
\begin{align*}
	\int_{2R}^{T} \hT' \cosh \hY \wnrm{(\covBox \m 1) \phi}_{L^{\infty}_{\hT'}} \, \ud \hT'
	\leq & C \eps_{1}^{3}.
\end{align*}
Since the left-hand side of \eqref{eq:ODE} bounds $\hT \cosh \hY \abs{ \phi}$ and $\hT \cosh \hY\abs{\covN \phi}$ from the above, \eqref{eq:BA:Linfty:improved-2} follows as desired.

\appendix
\numberwithin{equation}{section}
\section{Reduced systems in the temporal and Cronstr\"om gauges} \label{app:gauge}
The goal of this section is to derive reduced systems for \eqref{eq:CSH} and \eqref{eq:CSD} in the temporal and Cronstr\"om gauge, for which local well-posedness and finite speed of propagation are evident. In Section~\ref{subsec:expand}, we first expand various covariant expressions in terms of $A$ and the usual (component-wise) differential operators for vector- and Lie algebra-valued objects. Then in Sections~\ref{subsec:temporal} and \ref{subsec:cronstrom}, we exhibit reduced systems in the temporal and the Cronstr\"om gauges, respectively.

\subsection{Expansion of the covariant expressions} \label{subsec:expand}
We begin with expansion of various covariant expressions that arise in the Chern--Simons systems considered in this paper.
\begin{lemma} \label{lem:expand}
The following identities hold.
\begin{align} 
	\covBox \phi 
	= & \Box \phi + 2 \star (A \wedge \star \ud \phi)
	+ \dlt A \phi
	+ \star (A \wedge \star (A \phi)). \label{eq:expand:cov-Box} \\
	J_{\CSH}(\varphi) =&  2\bbrk{\varphi \wedge \ud \varphi} + 2 \bbrk{\varphi \wedge (A \varphi)} \label{eq:expand:J-CSH} \\
	\ud J_{\CSH}(\varphi) = & 2 \bbrk{\ud \varphi \wedge \ud \varphi} + 2 \bbrk{\ud \varphi \wedge (A \varphi)}
	+ 2 \star \bbrk{\varphi \wedge (J_{\CSH}(\varphi) \, \varphi)} \label{eq:expand:dJ-CSH} \\
	& - \bbrk{\varphi \wedge ([A \wedge A] \varphi)} 
	- 2 \bbrk{\varphi \wedge (A \wedge \ud \varphi)} , \notag \\	
	J_{\CSD}(\varphi) =&  \bbrk{\psi \wedge i \alp \psi} \label{eq:expand:J-CSD}\\
	\ud J_{\CSD}(\psi) = & \bbrk{\ud \psi \wedge i \alp \psi} - \bbrk{\psi \wedge (i \alp \wedge \ud \psi)}.	\label{eq:expand:dJ-CSD}
\end{align}
where $\Box$ denotes the usual d'Alembertian $\Box_{\bbR^{1+2}} = \nb^{\mu} \nb_{\mu}$ acting component-wisely.
\end{lemma}
\begin{proof} 
 The key tool is the calculus developed in Sections~\ref{subsec:extr-calc} and \ref{subsec:extr-calc-2}, which applies in particular to the usual differential operators $\ud$, $\LD$, $\Box$ etc. For \eqref{eq:expand:cov-Box}, we use Lemmas~\ref{lem:star}, \ref{lem:star-aux}, \ref{lem:covBox} and the identity $\covud = \ud + A \wedge(\cdot)$ to compute
 \begin{align*}
	\covBox \phi =& - \star \covud \star \covud \phi \\
	=& - \star \covud \star \ud \phi - \star \covud \star (A \phi) \\
	=& - \star \ud \star \ud \phi 
		- \star (A \wedge \star \ud \phi)
	- \star \ud \star (A \phi)
	- \star (A \wedge \star (A \phi)) \\
	=& \Box \phi - 2 \star (A \wedge \star \ud \phi) - \dlt A \phi - \star(A \wedge \star(A \phi)).
\end{align*}
The identity \eqref{eq:expand:J-CSH} follow directly from the definitions. On the other hand, to prove \eqref{eq:expand:dJ-CSH} we use Lemma~\ref{lem:extr-calc-bbrk} and the Leibniz rule for $\ud$ to compute
\begin{align*}
	\ud J_{\CSH}(\varphi) 
	= & 2 \ud \bbrk{\varphi \wedge \covud \varphi}	\\
	= & 2 \bbrk{\ud \varphi \wedge \ud \varphi} + 2 \bbrk{\ud \varphi \wedge (A \wedge \varphi)}
	+ 2 \bbrk{\varphi \wedge \ud (A \wedge \varphi)} \\
	=& 2 \bbrk{\ud \varphi \wedge \ud \varphi} + 2 \bbrk{\ud \varphi \wedge (A \wedge \varphi)}
	+ 2 \bbrk{\varphi \wedge (\ud A \wedge \varphi)} - 2 \bbrk{\varphi \wedge (A \wedge \ud \varphi)} \\
	=& 2 \bbrk{\ud \varphi \wedge \ud \varphi} + 2 \bbrk{\ud \varphi \wedge (A \wedge \varphi)}
	+ 2 \bbrk{\varphi \wedge (F \wedge \varphi)} \\
	& - \bbrk{\varphi \wedge ([A \wedge A] \wedge \varphi)} 
	- 2 \bbrk{\varphi \wedge (A \wedge \ud \varphi)}.
\end{align*}
Then by the Chern--Simons equation $F = \star J_{\CSH}$, \eqref{eq:expand:dJ-CSH} follows. Finally, \eqref{eq:expand:J-CSD} and \eqref{eq:expand:dJ-CSD} are straightforward consequences of the definitions; we remark that $\ud \alp = 0$ is used for the latter, which is clear in the rectilinear coordinates $(t, x^{1}, x^{2})$.
\end{proof}

\subsection{Reduced system in the temporal gauge} \label{subsec:temporal}
In the temporal gauge, the system \eqref{eq:CS-uni} is equivalent to the following system, for which local well-posedness and finite speed of propagation is rather immediate.
\begin{lemma} \label{lem:reduce:temp}
Let $I$ be a connected interval, and let $(A, \phi)$ be a pair of (smooth) connection 1-form and $V$-valued function on $I \times \bbR^{2} \subseteq \bbR^{1+2}$ which obeys the temporal gauge condition $\iota_{\rd_{t}} A = 0$. Then $(A, \phi)$ solves \eqref{eq:CS-uni} on $I \times \bbR^{2}$ if and only if it solves the reduced system
\begin{equation} \label{eq:reduce:temp}
\left\{
\begin{aligned}
	(\Box  \m 1) \phi =& 2 \star (A \wedge \star \ud \phi) + b \phi + \star(A \wedge \star(A \phi)) + U(\phi) \\
	\calL_{\rd_{t}} A = & \star (J \wedge \ud t) \\
	\calL_{\rd_{t}} b = & - \star (\ud J \wedge \ud t)
\end{aligned}
\right.
\end{equation}
and obeys the constraints
\begin{equation} \label{eq:reduce:temp-c}
	(F - \star J) \restriction_{\Sgm_{t_{0}}}= 0, \qquad (\dlt A - b) \restriction_{\Sgm_{t_{0}}} = 0,
\end{equation}
on $\Sgm_{t_{0}} = \set{t = t_{0}}$ for some $t_{0} \in I$.
\end{lemma}
Here, the notation $(F - \star J) \restriction_{\Sgm_{t_{0}}}$ refers to the restriction (or pullback) of the 2-form $F - \star J$ to $\Sgm_{t_{0}}$; in coordinates,
\begin{equation*}
(F - \star J) \restriction_{\Sgm_{t_{0}}} = (F - \star J)_{12} \, \ud x^{1} \wedge \ud x^{2}.
\end{equation*}
We also remind the reader that $J$ and $\ud J$ were computed in Lemma~\ref{lem:expand}.

As it will be clear from the proof below, the system \eqref{eq:CS-uni} is in fact already equivalent to the first two equations of \eqref{eq:reduce:temp} if we take $b = \dlt A$. The reason for introducing the auxiliary variable $b$ and the third equation is to exploit the fact that $\dlt A$ obeys a `better' transport equation than a general derivative of $A$. In particular, in the case of \eqref{eq:CSH} we only have at most one derivative of $\varphi$ on the right-hand side of $\calL_{\rd_{t}} \dlt A$; in general, we expect to see two derivatives from differentiating $\calL_{\rd_{t}} A = \star (J_{\CSH} \wedge \ud t)$. This observation is crucial for establishing local well-posedness of \eqref{eq:CSH} in the temporal gauge, since $\dlt A$ appears on the right-hand side of the Klein--Gordon equation for $\phi$, and the latter equation only gains one derivative.

\begin{proof} [Proof of Lemma~\ref{lem:reduce:temp}]
First, we claim that if $(A, \phi)$ is a solution to \eqref{eq:CS-uni}, then \eqref{eq:reduce:temp} and \eqref{eq:reduce:temp-c} are satisfied with $b = \dlt A$. Indeed, by Lemma~\ref{lem:expand}, the equation $(\covBox \m 1) \phi = U(\phi)$ is equivalent to the first equation of \eqref{eq:reduce:temp} with $b = \dlt A$. The second equation follows from $F = \star J$ by taking $\iota_{\rd_{t}}$ and using Cartan's formula \eqref{eq:cartan-eq}. Finally, taking $\dlt$ of the second equation and using Lemma~\ref{lem:star}, we have
\begin{equation} \label{eq:reduce:temp-dltA}
	\LD_{\rd_{t}} \dlt A = \dlt \calL_{\rd_{t}} A = \star \ud \star \star (J \wedge \ud t) = - \star (\ud J \wedge \ud t).
\end{equation}
Here, we have crucially used the fact that $\rd_{t}$ is Killing to commute $\calL_{\rd_{t}}$ with $\dlt$. This equation implies the third equation of \eqref{eq:reduce:temp}.

To conclude the proof, it remains to show that a solution to \eqref{eq:reduce:temp} and \eqref{eq:reduce:temp-c} also solves \eqref{eq:CS-uni}. As a first step, we observe that the constraints \eqref{eq:reduce:temp-c} are propagated by \eqref{eq:reduce:temp}. Indeed, for the first equation of \eqref{eq:reduce:temp-c}, we have
\begin{align*}
	\covLD_{\rd_{t}} F =& (\LD_{\rd_{t}} + \iota_{\rd_{t}} A) (\ud A + \frac{1}{2} [ A \wedge A]) \\
	 =& \ud (\iota_{\rd_{t}} \star J) + [A \wedge (\iota_{\rd_{t}} \star J)] \\
	 = & \covud \iota_{\rd_{t}} \star J = \covLD_{\rd_{t}} \star J - \iota_{\rd_{t}} \covud \star J.
\end{align*}
Note that the last term vanishes, since $\covdlt J = \star \covud \star J = 0$ for both $J = J_{\CSH}$ and $J_{\CSD}$. Hence $\covLD_{\rd_{t}}(F - \star J) = 0$, which along with \eqref{eq:reduce:temp-c} implies that $(F - \star J) \restriction_{\Sgm_{t}} = 0$ for every $t \in I$. Next, by \eqref{eq:reduce:temp-dltA} we have $\calL_{\rd_{t}} (\dlt A - b) = 0$, which shows that $b = \dlt A$ for every $t \in I$ as well.

We are now ready to show that $(A, \phi)$ solves \eqref{eq:CS-uni}. As we have just seen, if \eqref{eq:reduce:temp} and \eqref{eq:reduce:temp-c} hold, then $(F - \star J) \restriction_{\Sgm_{t}} = 0$ for every $t \in I$, i.e., the tangential components of Chern--Simons equation holds. On the other hand, the remaining components $\iota_{\rd_{t}} (F - \star J)$ are precisely the second equation of \eqref{eq:reduce:temp}. Finally, since $b = \dlt A$, it follows from Lemma~\ref{lem:expand} that the covariant Klein--Gordon equation holds as well. \qedhere
\end{proof}
%

\subsection{Reduced system in the Cronstr\"om gauge} \label{subsec:cronstrom}
In the Cronstr\"om gauge, we have the following analogue of Lemma~\ref{lem:reduce:temp}.
\begin{lemma} \label{lem:reduce:cron}
Let $I$ be a connected interval, and let $(A, \phi)$ be a pair of (smooth) connection 1-form and $V$-valued function on $\set{\hT \in I} \subseteq \bbR^{1+2}$ which obeys the Cronstr\"om gauge condition $\iota_{\rd_{\hT}} A = 0$. Then $(A, \phi)$ solves \eqref{eq:CS-uni} on $\set{\hT \in I}$ if and only if it solves the reduced system
\begin{equation} \label{eq:reduce:cron}
\left\{
\begin{aligned}
	(\Box  \m 1) \phi =& 2 \star (A \wedge \star \ud \phi) + b \phi + \star(A \wedge \star(A \phi)) + U(\phi) \\
	\calL_{\rd_{\hT}} A = & \star (J \wedge \ud \hT) \\
	\bb( \calL_{\rd_{\hT}} + \frac{2}{\hT} \bb) b = & - \star (\ud J \wedge \ud \hT)
\end{aligned}
\right.
\end{equation}
and obeys the constraints
\begin{equation} \label{eq:reduce:cron-c}
	(F - \star J) \restriction_{\calH_{\hT_{0}}}= 0, \qquad (\dlt A - b) \restriction_{\calH_{\hT_{0}}} = 0,
\end{equation}
on $\calH_{\hT_{0}}$ for some $\hT \in I$.
\end{lemma}

\begin{proof} 
We only sketch the proof of the following analogue of \eqref{eq:reduce:temp-dltA}:
\begin{equation} \label{eq:reduce:cron-dltA}
	\bb( \LD_{\rd_{\hT}} + \frac{2}{\hT} \bb) \dlt A = - \star (\ud J \wedge \ud \hT),
\end{equation}
since rest of the proof is analogous to the temporal gauge case (Lemma~\ref{lem:reduce:temp}).

In Lemma~\ref{eq:reduce:temp-dltA}, commutation of $\calL_{\rd_{t}}$ and $\star$ was simple due to the fact that $\rd_{t}$ is Killing. In the present case, $\rd_{\hT}$ is \emph{not} Killing; however, we may exploit the fact that $S = \hT \rd_{\hT}$ is \emph{conformally Killing}. Indeed, on the $(1+d)$-dimensional Minkowski space, the scaling vector field $S$ obeys the identities
\begin{align*}
	\LD_{S} \eta = 2 \eta, \quad
	\LD_{S} \eta^{-1} = - 2 \eta^{-1}, \quad
	\LD_{S} \eps = (1+d) \eps.
\end{align*}
In our case, $1+d = 3$. Given two real-valued $k$-forms $\omg^{1}$ and $\omg^{2}$, we have
\begin{equation*}
	\LD_{S} \bb( \eta^{-1}(\omg^{1}, \omg^{2}) \eps \bb)
	= \eta^{-1}(\LD_{S} \omg^{1}, \omg^{2}) \eps 
		+ \eta^{-1}(\omg^{1}, \LD_{S} \omg^{2}) \eps 
		+ (3-2k) \eta^{-1}(\omg^{1}, \omg^{2}) \eps.
\end{equation*}
Recalling the characterization \eqref{eq:star-def} of $\star$, it follows that
\begin{equation*}
	\LD_{S} \star \omg = \star \LD_{S} \omg + (3-2k) \star \omg.
\end{equation*}

We now begin the proof in earnest. Assume that the second equation of \eqref{eq:reduce:cron} holds. 
Computing component-wisely for a $\LieAlg$-valued 1-form $A$, we have
\begin{align*}
	\dlt \LD_{S} A
	= & \star \ud \star \LD_{S} A \\
	= & \LD_{S} \star \ud \star A - (3 - 2 \cdot 3) \star \ud \star A - (3 - 2) \star \ud \star A 
	= (\LD_{S} + 2) \dlt A.
\end{align*}
By Cartan's formula \eqref{eq:cartan-eq} and the fact that $\iota_{S} A = \hT \iota_{\rd_{\hT}} A = 0$, the left-hand side equals
\begin{equation*}
	\dlt (\star J \wedge S^{\flat}) = - \star (\ud J \wedge S^{\flat}) + \star(J \wedge \ud S^{\flat}).
\end{equation*}
Note that $S^{\flat} = \hT \ud \hT = \frac{1}{2} \ud \hT^{2}$, hence $\ud S^{\flat} = \frac{1}{2} \ud^{2} \hT^{2} = 0$. It follows that
\begin{equation*}
	\hT \bb( \LD_{\rd_{\hT}} + \frac{2}{\hT} \bb) \dlt A = (\LD_{S} + 2) \dlt A = - \hT \star (\ud J \wedge \ud \hT).
\end{equation*}
Dividing by $\hT > 0$, \eqref{eq:reduce:cron-dltA} follows. \qedhere
\end{proof}

\bibliographystyle{amsplain}

\providecommand{\bysame}{\leavevmode\hbox to3em{\hrulefill}\thinspace}
\providecommand{\MR}{\relax\ifhmode\unskip\space\fi MR }
\providecommand{\MRhref}[2]{%
  \href{http://www.ams.org/mathscinet-getitem?mr=#1}{#2}
}
\providecommand{\href}[2]{#2}


\end{document}